\documentclass[12pt]{amsart}
\usepackage{amssymb,amscd}
\usepackage{pstricks,pstricks-add}

\newtheorem{theorem}{Theorem}[section]
\newtheorem{lemma}[theorem]{Lemma}
\newtheorem{prop}[theorem]{Proposition}

\newtheorem{definition}[theorem]{Definition}
\newtheorem{conj}[theorem]{Conjecture}

\theoremstyle{definition}
\newtheorem{remark}[theorem]{Remark}
\newtheorem{example}[theorem]{Example}

\numberwithin{equation}{section}

\newcommand{\Z}{\mathbb{Z}}
\newcommand{\R}{\mathbb{R}}
\newcommand{\C}{\mathbb{C}}

\newcommand{\PP}{\mathbb{P}}

\renewcommand{\O}{\mathcal{O}}
\newcommand{\F}{\mathcal{F}}
\newcommand{\W}{\mathcal{W}}

\newcommand{\Coh}{\mathrm{Coh}}
\newcommand{\Hom}{\mathrm{Hom}}
\newcommand{\Ext}{\mathrm{Ext}}
\newcommand{\Perf}{\mathrm{Perf}}
\newcommand{\sg}{{sg}}
\newcommand{\A}{\mathcal{A}}
\renewcommand{\o}{\overline}

\title[Speculations on HMS for hypersurfaces in $(\C^*)^n$]{Speculations on homological mirror symmetry for hypersurfaces in $(\C^*)^n$}
\author{Denis Auroux}
\address{Department of Mathematics, UC Berkeley, Berkeley CA 94720-3840,
USA\newline \indent
School of Mathematics, Institute for Advanced Study, Princeton NJ 08540,
USA}
\email{auroux@math.berkeley.edu}
\thanks{This work was partially supported by NSF grants DMS-1264662 and
DMS-1406274; by a Simons Foundation grant (\#\,385573, Simons
Collaboration on Homological Mirror Symmetry); by the Eilenberg Chair at
Columbia University; and by the Schmidt Fellowship and the IAS Fund for
Mathematics.}

\begin{document}
\begin{abstract}
Given an algebraic hypersurface $H=f^{-1}(0)$ in $(\C^*)^n$, homological mirror
symmetry relates the wrapped Fukaya category of $H$ to the derived category
of singularities of the mirror Landau-Ginzburg model. We propose an enriched
version of this picture which also features the wrapped Fukaya category of
the complement $(\C^*)^n\setminus H$ and the Fukaya-Seidel category of the
Landau-Ginzburg model $((\C^*)^n,f)$. We illustrate our speculations on simple examples,
and sketch a proof of homological mirror symmetry for higher-dimensional
pairs of pants.
\end{abstract}

\maketitle

\section{Introduction}

Let $H=f^{-1}(0)\subset (\C^*)^n$ be a smooth algebraic hypersurface (close
to a maximal degeneration limit), whose defining equation
is a Laurent polynomial of the form
\begin{equation}\label{eq:f}
f(x_1,\dots,x_n)=\sum_{\alpha\in A} c_\alpha \tau^{\rho(\alpha)} x_1^{\alpha_1}
\dots x_n^{\alpha_n},
\end{equation}
where $A$ is a finite subset of $\Z^n$, $c_\alpha \in \C^*$, 
$\tau\in \R_+$ is assumed to be sufficiently small, and 
$\rho:A\to\R$ is a convex function.

More precisely, we require that $\rho$ is the restriction to $A$ of
a convex piecewise linear function $\hat\rho$ defined on the convex hull 
$\mathrm{Conv}(A)\subset \R^n$. The maximal domains of linearity of
$\hat\rho$ define a polyhedral decomposition $\mathcal{P}$ of $\mathrm{Conv}(A)$,
whose set of vertices is required to be exactly $A$.
We further assume that all the cells of $\mathcal{P}$ are
congruent under the action of $GL(n,\Z)$ to standard simplices; this
ensures that the limit $\tau\to 0$ corresponds to a maximal degeneration,
and that the mirror is smooth.

It was first proposed by Hori and Vafa \cite{HV} that $H$ should arise as a mirror
to a toric Calabi-Yau manifold $Y$, or more precisely, a toric Landau-Ginzburg model
$(Y,W)$. A careful construction of the
mirror following the philosophy of the Strominger-Yau-Zaslow conjecture
is given in \cite{AAK}. The outcome can be described as follows.

Consider the piecewise linear function $\varphi:\R^n\to \R$ obtained by
``tropicalizing'' $f$,
\begin{equation}\label{eq:tropf}
\varphi(\xi)=\max \{ \langle \alpha,\xi\rangle-
\rho(\alpha)\,|\,\alpha \in A\},
\end{equation} and 
the (noncompact) polytope $\Delta_Y\subseteq \R^{n+1}$ defined by
\begin{equation}\label{eq:delta_Y}
\Delta_Y=\{(\xi,\eta)\in \R^n\oplus \R\,|\,\eta\ge \varphi(\xi)\}.
\end{equation}
Let $Y$ be the (noncompact) $(n+1)$-dimensional toric variety defined
by the moment polytope $\Delta_Y$. Equivalently, $Y$ is described by the
fan $\Sigma_Y\subseteq \R^n\oplus \R$
whose rays are generated by the vectors $(-\alpha,1)$, $\alpha\in A$, and
in which the vectors $(-\alpha_1,1),\dots,(-\alpha_k,1)$ span a cone if
and only if $\alpha_1,\dots,\alpha_k$ span a cell of $\mathcal{P}$.
Finally, we define 
\begin{equation}\label{eq:W}
W=-z^{(0,0,\dots,0,1)}\in \O(Y).
\end{equation}

The irreducible toric divisors of $Y$ are indexed by the elements of $A$;
denote by $Z_\alpha$ the divisor which corresponds to the ray
$(-\alpha,1)$ of $\Sigma_Y$, i.e.\ to the facet of $\Delta_Y$ given by the
graph of $\varphi$ over the region where the maximum in \eqref{eq:tropf} is
achieved by $\alpha$. The superpotential $W$ is then (up to sign) the toric
monomial which vanishes to order 1 on each toric divisor $Z_\alpha$.
Hence $W^{-1}(0)=\bigcup_{\alpha\in A} Z_\alpha$ (the union of all
toric strata).

The direction of homological mirror symmetry that we shall concern ourselves
with predicts an equivalence between the (derived) wrapped Fukaya category of $H$
\cite{AS,AbGenerate} and the
derived category of singularities of the Landau-Ginzburg model $(Y,W)$
\cite{Orlov}, i.e.\ the quotient $D^b_\sg(Z):=D^b\Coh(Z)/\Perf(Z)$ of the
derived category of coherent sheaves on the singular fiber
$Z:=W^{-1}(0)=\bigcup_{\alpha\in A} Z_\alpha$ by the full triangulated
subcategory of perfect complexes:
\begin{conj}\label{conj:hms} {\rm (Homological mirror symmetry)}
\begin{equation}\label{eq:hms}
\W(H)\simeq D^b_\sg(Z).
\end{equation}
\end{conj}

(The other direction of homological mirror symmetry, relating
coherent sheaves on $H$ to the Fukaya category of the Landau-Ginzburg model $(Y,W)$, is
established in work in progress of the author with Mohammed Abouzaid \cite{AA};
the methods used to approach the two directions are completely unrelated.)

It is possible, and even likely, that Conjecture \ref{conj:hms} should in fact be stated
at the level of the idempotent completions of the derived categories on each
side of \eqref{eq:hms}; for simplicity we ignore this issue here.

The wrapped Fukaya category $\W(H)$ depends only
on the set $A\subset \Z^n$, not on the coefficients $c_\alpha$ or
the function $\rho$ in \eqref{eq:f}, since the hypersurfaces corresponding
to different choices are deformation equivalent Liouville (or Stein)
submanifolds of $(\C^*)^n$.  Meanwhile, $Y$ depends on the
polyhedral decomposition $\mathcal{P}$ of $\mathrm{Conv}(A)$, so different
choices of $\rho$ can yield different mirrors; however these mirrors are
birational to each other (related by flops), and so
the resulting derived categories of singularities are expected to be
equivalent.

So far, homological mirror symmetry as stated in Conjecture \ref{conj:hms}
has only been established in the
1-dimensional case, i.e.\ for $H\subset (\C^*)^2$: the case of the
pair of pants (and other punctured spheres) is established in \cite{AAEKO},
and higher genus Riemann surfaces are treated in Heather Lee's thesis \cite{Lee}.
In higher dimensions, the first step is to consider
(generalized) pairs of pants. With the current technology, the
computation of the wrapped Fukaya category requires quite a bit of work;
we sketch a possible approach in \S \ref{s:pantscomp}. (Contrast with
Sheridan's computations for compact exact Lagrangians \cite{Sheridan}.)
We also note Nadler's recent introduction of ``wrapped microlocal sheaves'',
ultimately expected to be equivalent to the wrapped Fukaya category;
the analogue of Conjecture \ref{conj:hms} for wrapped
microlocal sheaves has already been verified for higher-dimensional pairs of pants
\cite{NadlerWrap}.

Since $D^b_\sg(Z)$ is by definition a
quotient of $D^b\Coh(Z)$, it is natural to ask for an interpretation of the
latter category on the symplectic side. We propose:

\begin{conj}\label{conj:hmscompl}
$Z=\bigcup_\alpha Z_\alpha\subset Y$ is mirror to the complement
$(\C^*)^n\setminus H$, and there is a commutative diagram
\begin{equation}\label{eq:hmscompl}
\begin{CD}
\W((\C^*)^n\setminus H) @>\simeq>> D^b\Coh(Z) \\
@V{\rho}VV @VV{q}V \\
\W(H) @>\simeq>> D^b_\sg(Z)
\end{CD}
\end{equation}
where $\rho$ is a restriction functor (see \S \ref{s:hmscompl}), $q$ is the projection to the quotient,
and the horizontal equivalences are predicted by homological mirror
symmetry. 
\end{conj}

We note that the categories in the top row are $\Z$-graded,
whereas those in the bottom row are only $\Z/2$-graded unless some
additional data is chosen.
 
Roughly speaking, the restriction functor $\rho$ singles out the ends of
a Lagrangian submanifold of $(\C^*)^n\setminus H$ which lie on the missing
divisor $H$. More precisely, $\rho$ is the
composition of restriction to a neighborhood of $H$ isomorphic to
the product of $H$ with a punctured disc $\mathbb{D}^*$, and
``projection'' from $H\times \mathbb{D}^*$ to $H$; see \S \ref{s:hmscompl}.

Two comments are in order. First, the top
row of \eqref{eq:hmscompl} fits into the general philosophy that removing a
divisor from a symplectic manifold (here $(\C^*)^n$) should correspond to a 
degeneration of its mirror (in our case $(\C^*)^n$) to a singular space
(namely~$Z$);
the level sets of $W$ provide exactly such a degeneration.
Seidel and Sheridan's formalism of relative Fukaya categories
\cite{Sheridan2} exhibits $\W((\C^*)^n)$ as a deformation of a
full subcategory of $\W((\C^*)^n\setminus H)$ (consisting of Lagrangians
with no ends on $H$, i.e.\ annihilated by $\rho$), just as the
derived categories of the regular fibers of $W$ arise as deformations of a
full subcategory of $D^b\Coh(Z)$ (in fact, $\Perf(Z)$).

Second, $(\C^*)^n\setminus H$ can itself be viewed as a hypersurface in
$(\C^*)^{n+1}$, defined by
$$\hat{f}(x_1,\dots,x_{n+1})=f(x_1,\dots,x_n)+x_{n+1}=0.$$
(This is in fact one way to define the Liouville structure on the complement
of $H$). The tropicalization of $\hat{f}$ is $\hat\varphi(\xi_1,\dots,
\xi_{n+1})=\max(\varphi(\xi_1,\dots,\xi_n),\xi_{n+1})$, and thus the
construction in \cite{AAK} predicts that the mirror to this hypersurface is 
the toric Landau-Ginzburg model $(\hat{Y},\hat{W})=(\C\times Y, yW)$ 
(where $y$ is the coordinate on the $\C$ factor). On the other hand,
Orlov's ``Kn\"orrer periodicity'' result \cite{OrlovKnorrer} implies that
the derived category of singularities of the Landau-Ginzburg model
$(\hat{Y},\hat{W})$ is equivalent to the derived category of coherent
sheaves of $W^{-1}(0)=Z\subset Y$. Thus, the two predictions for the
mirrors of $(\C^*)^n\setminus H$, namely the Landau-Ginzburg model 
$(\hat{Y},\hat{W}$) and the singular variety $Z$, are consistent with each
other.

We can further enrich the picture by considering the Fukaya-Seidel category
of the Landau-Ginzburg model $((\C^*)^n,f)$, using $f$ to view
$(\C^*)^n\setminus H$ as the total space of a fibration over $\C^*$.
Specifically, assume that $0\in
A$, so that $f$ has a non-trivial constant term. The version of the
Fukaya-Seidel category that we consider is essentially 
that introduced by Abouzaid in \cite{AbToric,AbToric2}, at least when $0$ is in the
interior of $\mathrm{Conv}(A)$; specifically, the objects are
admissible Lagrangian submanifolds of
$(\C^*)^n$ with boundary on a fiber of $f$, e.g.\ $f^{-1}(0)=H$, which are moreover required to
lie in the subset of $(\C^*)^n$ where the constant term dominates all the
other monomials in $f$. Due to this latter restriction, our category is
often smaller than that defined by Seidel; to avoid 
confusion, we denote the restricted version by $\F^\circ((\C^*)^n,f)$.
One notable difference from Abouzaid's setup is that when 
$0$ is not in the interior of $\mathrm{Conv}(A)$ the region where the
constant term dominates is non-compact and the category we consider
involves some wrapping. See \S \ref{s:FS}.

There are ``acceleration'' functors $\alpha_0$ and $\alpha_\infty$
from 
$\F^\circ((\C^*)^n,f)$ to 
$\W((\C^*)^n\setminus H)$. 
The functor $\alpha_0$ takes admissible Lagrangian submanifolds in 
$((\C^*)^n,f)$ with boundary in $f^{-1}(0)=H$ and views them as Lagrangian submanifolds 
of $(\C^*)^n\setminus H$. Acceleration then ``turns on'' wrapping around
the central fiber $H=f^{-1}(0)$. By construction, $\rho\circ
\alpha_0:\F^\circ((\C^*)^n,f)\to \W(H)$ is expected to coincide with the ``restriction to the
fiber'' functor. Meanwhile, $\alpha_\infty$ takes admissible Lagrangians
with boundary in some other fiber $f^{-1}(c_0)$, and extends them by parallel
transport along a path from $c_0$ to infinity in order to obtain 
properly embedded Lagrangian submanifolds of $(\C^*)^n$ which avoid $H$
altogether. The construction of $\alpha_\infty$ is not canonical, however if the
following assumption holds:
\begin{equation}\label{eq:linebundlecase}
\text{$0\in A$ is a vertex of every maximal cell of the polyhedral
decomposition $\mathcal{P}$,}
\end{equation}
then there is a distinguished choice; see \S \ref{s:FSaccel}.
The two types of acceleration functors are manifestly different, 
as $\rho\circ \alpha_\infty$ is identically zero.

The interpretation of the acceleration functors $\alpha_0$ and
$\alpha_\infty$ under mirror symmetry is as follows.  The element $0\in A$
corresponds to a distinguished irreducible toric divisor $Z_0$ of $Y$.
When $Z_0$ is compact (which corresponds to $0$ being an interior point
of $\mathrm{Conv}(A)$), it follows from Abouzaid's thesis \cite{AbToric2} 
that the Fukaya-Seidel category $\F^\circ((\C^*)^n,f)$ considered here is derived 
equivalent to $D^b\Coh(Z_0)$. In fact Abouzaid's argument can be adapted to
show that the equivalence still holds in the non-compact case.
There is a natural functor $i_*:D^b\Coh(Z_0)\to
D^b\Coh(Z)$ induced by the inclusion $i:Z_0\hookrightarrow Z$. 
On the other hand, there is sometimes a preferred projection $\pi:Z\to
Z_0$; this is e.g.\ the case when \eqref{eq:linebundlecase} holds,
which causes $Y$ to be isomorphic to the total space of a line bundle
over $Z_0$. We then have a pullback functor $\pi^*:D^b\Coh(Z_0)\to
D^b\Coh(Z)$, whose image is contained in $\Perf(Z)$ since the maximal
degeneration assumption implies that $Z_0$ is smooth, i.e.\ $\Perf(Z_0)=
D^b\Coh(Z_0)$.

\begin{conj}\label{conj:FS} Under homological mirror symmetry, 
\begin{enumerate}
\item the acceleration functor
$\alpha_0:\F^\circ((\C^*)^n,f)\to \W((\C^*)^n\setminus H)$ corresponds to the
inclusion pushforward $i_*:D^b\Coh(Z_0)\to D^b\Coh(Z)$;
\item if \eqref{eq:linebundlecase} holds, then the functor\/
$\alpha_\infty:\F^\circ((\C^*)^n,f)\to \W((\C^*)^n\setminus H)$ corresponds to
the pullback $\pi^*:D^b\Coh(Z_0)\to \Perf(Z)\subset D^b\Coh(Z)$.
\end{enumerate}
\end{conj}

\noindent The functors $\alpha_0$ and $\alpha_\infty$ have a host of
further properties, which can ultimately be interpreted in terms of
push-pull adjunctions for $i:Z_0\to Z$ and $\pi:Z\to Z_0$ on the mirror side.
For example, there is a distinguished
natural transformation between the functors $\alpha_\infty$ and $\alpha_0$,
whose mapping cone involves a ``lifting'' functor
$$j:\W(H)\to \W((\C^*)^n\setminus H).$$ The functor $j$ is induced by the
parallel transport of Lagrangian submanifolds of $H$ over an arc connecting 
0 to  infinity in $\C^*$ (avoiding the critical values of $f$).
By construction, $j$ is a right (quasi)inverse to the restriction functor
$\rho$, i.e.\ $\rho \circ j\simeq \mathrm{id}$; assuming
\eqref{eq:linebundlecase}, the functor $j$ should correspond under mirror symmetry to
an explicit splitting of the quotient $q:D^b\Coh(Z)\to D^b_\sg(Z)$.
See \S \ref{s:FSj}.

\begin{remark}\label{rmk:rescale}
While the defining equation of the hypersurface $H$ can be rescaled by any
Laurent monomial, the category $\F^\circ((\C^*)^n,f)$ depends very much on the choice of 
normalization, and so do the functors
$\alpha_0,\alpha_\infty,j$ discussed above. Given $\alpha \in A$,
considering $x^{-\alpha}f$ instead of $f$ causes the distinguished
component of $Z$ to become $Z_\alpha$ instead of $Z_0$. We then get
one instance of Conjecture \ref{conj:FS} for each component of $Z$.
\end{remark}

\begin{remark}\label{rmk:sylvan}
As pointed out by Zack Sylvan, there is another way to shed light on the relationship between 
$\F((\C^*)^n,f)$ and $\W((\C^*)^n\setminus H)$, by viewing
$(\C^*)^n\setminus H$ as the outcome of gluing together the Landau-Ginzburg
models $((\C^*)^n,f)$ and $(\C^*\times H,z)$ (where $z$ is the coordinate on
the first factor) along their common fiber $H$. This is an instance of gluing together Liouville domains
with {\em stops} and their partially wrapped Fukaya categories \cite{Sylvan},
and one expects a pushout diagram \cite{Sylvan2, GPS}
$$\begin{CD}\W(H)@>>> \F((\C^*)^n,f)\\
@VVV @VVi_2V \\
\F(\C^*\times H,z) @>i_1>> \W((\C^*)^n\setminus H).
\end{CD}
$$
Because the Fukaya category of $(\C^*,z)$ is generated by one object with endomorphism algebra
$\C[t]$ (it is mirror to the affine line), the category $\F(\C^*\times H,z)$ is
related to $\W(H)$ by ``extension of scalars'' from $\C$ to $\C[t]$. (This
is {\em not} the left edge of the pushout diagram, which amounts to tensoring
with the torsion module $\C[t]/t$ rather than $\C[t]$.) Up to this,
$i_1$ is essentially the functor $j$ discussed above. Meanwhile, when
\eqref{eq:linebundlecase} holds, there is no difference between
$\F((\C^*)^n,f)$ and $\F^\circ((\C^*)^n,f)$, and $i_2$ coincides with 
$\alpha_\infty$. (Otherwise $\F^\circ$ is strictly smaller).
The pushout diagram then implies that $\W((\C^*)^n\setminus H)$ is generated
by the images of the functors $j$ and $\alpha_\infty$.
\end{remark}

To illustrate the various constructions and conjectures, we will primarily consider
two families of examples:

\begin{example}[Pairs of pants]\label{ex:pants}
$$f(x_1,\dots,x_n)=x_1+\dots+x_n+1,$$
$H=:\Pi_{n-1}$ is the $(n-1)$-dimensional pair of pants, and its complement is
isomorphic to the
$n$-dimensional pair of pants $\Pi_n$. The mirror is $(Y,W)\simeq
(\C^{n+1},-z_1\dots z_{n+1})$, and $Z=\{z_1\dots z_{n+1}=0\}$ is the union of the $n+1$
coordinate hyperplanes.
\end{example}

\begin{example}[Local $\PP^n$]\label{ex:Pn}
$$f(x_1,\dots,x_n)=x_1+\dots+x_n+\frac{\tau}{x_1\dots x_n}+1,$$
$Y$ is isomorphic to the total space of the anticanonical bundle
$\O(-(n+1))\to \PP^n$, and $Z\subset Y$ is the union of the zero section $Z_0\simeq
\PP^n$ and the total spaces of $\O(-(n+1))$ over the $n+1$ coordinate hyperplanes
of $\PP^n$.
\end{example}

We will in particular see in \S\S \ref{s:pants}--\ref{s:pantscomp} that Conjectures
\ref{conj:hmscompl} and \ref{conj:FS} provide a blueprint for understanding
the wrapped Fukaya categories of pairs of pants by induction on dimension,
using the fact that $\Pi_n\simeq (\C^*)^n\setminus \Pi_{n-1}$.

The rest of this paper is organized as follows. The first two sections are
expository: in Section \ref{s:wrapped} we briefly review the definition of
the wrapped Fukaya category, and Section~\ref{s:pants1} illustrates the
definition by considering the case of the (1-dimensional) 
pair of pants treated in \cite{AAEKO}, with an eye towards Conjecture
\ref{conj:hmscompl}. The next three sections describe the
various categories and functors that appear in Conjectures \ref{conj:hmscompl} and \ref{conj:FS}:
in Section \ref{s:hmscompl} we introduce the wrapped category $\W((\C^*)^n\setminus H)$ and the restriction functor
$\rho$; in Section \ref{s:FS} we discuss the category
$\F^\circ((\C^*)^n,f)$ and its main properties; and
Section \ref{s:conjFS} is devoted to the functors $\alpha_0$, $\alpha_\infty$ and $j$. 
Sections \ref{s:Pn} and \ref{s:pants}
illustrate these constructions for local $\PP^n$ and
for higher-dimensional pants; finally, Section~\ref{s:pantscomp}
sketches an approach to the computation of $\W(\Pi_n)$.

\subsection*{Acknowledgements}
I would like to thank Mohammed Abouzaid, Sheel Ganatra, Dmitri Orlov, Paul Seidel 
and Zack Sylvan for a number of useful comments and suggestions.
I am also thankful to Columbia University and IAS for their hospitality and support
during the preparation of this text. This work was partially supported by
NSF grants DMS-1264662 and DMS-1406274, the Simons Foundation
(grant \#\,385573, Simons Collaboration on Homological Mirror Symmetry),
the Eilenberg Chair at Columbia University, the Schmidt Fellowship
and the IAS Fund for Mathematics.

\section{Background: the wrapped Fukaya category}\label{s:wrapped}

Let $(X,\omega=d\lambda)$ be a Liouville manifold, i.e.\ an exact symplectic
manifold such that the flow of the Liouville vector field $Z$ defined by
$\iota_Z \omega=\lambda$ is complete and outward pointing at infinity.
In other terms, $X$ is the completion of a compact domain $X^{in}$
with contact boundary $(\partial X^{in},\alpha=\lambda_{|\partial X^{in}})$,
and the Liouville flow identifies
$X\setminus X^{in}$ with the positive symplectization 
$(1,+\infty)\times \partial X^{in}$
endowed with the exact symplectic form $\omega=d(r\alpha)$ (where $r$ is the
coordinate on $(1,+\infty)$). In this model, $Z=r\partial_r$.

The objects of the wrapped Fukaya category $\W(X)$ are properly embedded
exact Lagrangian submanifolds which are conical at infinity, 
i.e.\ any non-compact ends are modelled on the product of $(1,+\infty)$
with some Legendrian submanifold of $(\partial X^{in},\alpha)$.

The main feature of the wrapped Fukaya category is that Floer theory
theory is modified by suitable Hamiltonian perturbations so as to include
not only Lagrangian intersections inside $X^{in}$, but also Reeb chords 
between the Legendrians in $\partial X^{in}$. There are two main ways to
carry out the construction, which we briefly review. (We will mostly use the
first one.) We assume general familiarity with Lagrangian Floer homology and
ordinary Fukaya categories; see \cite{AuGuide,SeBook}.

\subsection{Construction via quadratic Hamiltonian perturbations}\label{ss:wrap_quadratic}

This setup for wrapped Fukaya categories is described in detail in
\cite{AbGenerate}. In this version, the wrapped Floer complex of a pair of objects $L_0,L_1$ is defined
using a specific class of Hamiltonians which grow quadratically
at infinity, say 
$H=\frac12 r^2$ outside of a compact set. 
Given two objects $L_0,L_1$, the
generating set $\mathcal{X}(L_0,L_1)$ of 
$CW(L_0,L_1)=CW(L_0,L_1;H)$ consists of time 1 trajectories
of the Hamiltonian vector field $X_H$ which start on $L_0$ and end on 
$L_1$, i.e.\ points of $\phi^1_H(L_0)\cap L_1$. Since at infinity $X_H$
is $r$ times the Reeb vector field of $(\partial
X^{in},\alpha)$, the generators in the cylindrical end can also be thought of as
Reeb chords (of arbitrary positive length) from $L_0$ to $L_1$ at the contact boundary.
(In practice one may need to perturb $H$ slightly in order to achieve
transversality.)

The differential $\partial=\mu^1$ on $CW(L_0,L_1)$ counts solutions to Floer's
equation \begin{equation}\label{eq:floereq}
\frac{\partial u}{\partial s}+J(t,u)\left(\frac{\partial u}{\partial t}-
X_H(t,u)\right)=0,
\end{equation}
where $u:\R\times [0,1]\to X$ is subject to the boundary conditions 
$u(s,0)\in L_0$ and $u(s,1)\in L_1$ and a finite energy condition.
Given two generators $x_-,x_+ \in \mathcal{X}(L_0,L_1)$, the coefficient of
$x_-$ in $\partial x_+$ is a (signed) count of index 1 solutions of
\eqref{eq:floereq} (up to reparametrization by translation) which converge
to $x_\pm$ as $s\to \pm\infty$.

Floer's equation can be recast as a plain
Cauchy-Riemann equation by the following trick: consider
$\tilde{u}(s,t)=\phi_H^{1-t}(u(s,t))$, where $\phi_H^{1-t}$ is the flow of
$X_H$ over the interval $[t,1]$. Then \eqref{eq:floereq}
becomes
$$\frac{\partial \tilde{u}}{\partial
s}+\tilde{J}(t,\tilde{u})\,\frac{\partial
\tilde{u}}{\partial t}=0,$$
where $\tilde{J}(t)=(\phi_H^{1-t})_*(J(t))$. Hence solutions to Floer's
equation correspond to
honest $\tilde{J}$-holomorphic strips with boundaries on $\phi_H^1(L_0)$ and
$L_1$.

The Floer product $\mu^2$ and higher compositions
$$\mu^k:CW(L_{k-1},L_k;H)\otimes \dots\otimes CW(L_0,L_1;H)\to
CW(L_0,L_k;H)[2-k]$$ are constructed similarly, with an important subtlety.
The $k$-fold product $\mu^k$ counts rigid solutions to a perturbed
Cauchy-Riemann equation of the form \begin{equation}\label{eq:perthol}
\Bigl(du-X_H\otimes \beta\Bigr)^{0,1}_J=0,
\end{equation}
where $u$ is a map from a domain $D$ biholomorphic to a disc with $k+1$ boundary punctures (viewed as
strip-like ends) to $X$ and $\beta$ is a closed 1-form on $D$ such that $\beta_{|\partial
D}=0$ and $\beta$ is standard in each strip-like end.

(Rather than the usual punctured discs, a 
convenient model for the domain $D$ which makes the strip-like ends readily apparent
is to take $D$ to be a strip $\R\times [0,k]$ with $k-1$ slits $(s_j,+\infty)
\times \{t_j\}$ removed. Away from the boundary of the moduli space one can
moreover take $t_j=j$. The conformal parameters are then
simply $s_1,\dots,s_{k-1}$ up to simultaneous translation, and
we can take $\beta=dt$.)

The issue is that counting solutions of \eqref{eq:perthol} with boundary on
$L_0,\dots,L_k$ naturally yields a map with values in $CW(L_0,L_k;kH)$,
whose generators are time $k$ (rather than time 1) trajectories of $X_H$ from $L_0$ to
$L_k$. While the usual construction of a continuation map from $CW(L_0,L_k;kH)$ to
$CW(L_0,L_k;H)$ fails due to lack of energy estimates,
a map can nonetheless be constructed via a rescaling trick
\cite{AbGenerate}. Namely, the time $\log k$ flow of the Liouville
vector field $Z$, which is conformally symplectic and rescales the $r$ coordinate by a factor of $k$,
conjugates time $k$ and time 1 trajectories of $X_H$. Denoting this flow by $\psi^k$, we have a
natural isomorphism \begin{equation}\label{eq:rescale_trick}
CW(L_0,L_k;H,J)\cong CW(\psi^k(L_0),\psi^k(L_k);
k^{-1} (\psi^k)^*H, \psi^k_* J),\end{equation} and since $k^{-1}(\psi^k)^*H=kH$
at infinity, there is a well-defined continuation map from
$CW(L_0,L_k;kH,J)$ to the latter complex. (This is easiest when the
Lagrangians under consideration are globally invariant under the Liouville
flow, as will be the case for our main examples; in general $\psi^k(L_i)$
differs from $L_i$ by a compactly supported Hamiltonian isotopy, which is
annoying but does not pose any technical difficulties.)
However, to ensure that the $A_\infty$-relations hold, 
the continuation homotopy should be incorporated directly into \eqref{eq:perthol},
making the Hamiltonians, almost-complex structures, and boundary conditions 
depend on $s$ so that the solutions converge at $s\to -\infty$ to
generators of the right-hand side of \eqref{eq:rescale_trick}
rather than $CW(L_0,L_k;kH,J)$; see \cite{AbGenerate}.

Assuming the Lagrangians under consideration are invariant under the
Liouville flow, we can use the same trick as above to recast 
\eqref{eq:perthol} as an unperturbed Cauchy-Riemann equation with respect to 
a different almost-complex structure. For instance, given generators
$x_1$ of $CW(L_0,L_1)$, $x_2$ of $CW(L_1,L_2)$, and $y$ of $CW(L_0,L_2)$
(viewed as points of $\phi_H^1(L_i)\cap L_j$),
the coefficient of $y$ in $\mu^2(x_2,x_1)$ can be viewed as a count of
pseudo-holomorphic discs with boundary on $\phi_H^2(L_0)$,
$\phi_H^1(L_1)$, and $L_2$, whose strip-like ends converge to
$\phi_H^1(x_1)\in \phi_H^2(L_0)\cap \phi_H^1(L_1)$, $x_2\in
\phi_H^1(L_1)\cap L_2$, and the inverse image $\tilde{y}\in \phi_H^2(L_0)\cap
L_2$ of $y$ under the rescaling map $\psi^2$. See \S \ref{ss:cyl_calc} for
an example.

\subsection{Construction by localization}\label{ss:wrap_loc}

A different setup for wrapped Floer theory, which is especially useful for 
comparisons with Fukaya-Seidel categories and for the construction of
restriction functors, uses finite wrapping and
localization with respect to certain continuation morphisms. The
construction we sketch here lies somewhere in between the original 
one due to Abouzaid-Seidel \cite{AS} and the works in progress by Abouzaid-Seidel
and Abouzaid-Ganatra \cite{AS2,AG}.

We now consider a Hamiltonian $h$ which has linear growth. For instance,
one could require that $h=r$ at infinity. However, when the contact boundary
(or part thereof) comes equipped with an open book structure or more
generally a compatible $S^1$-valued projection (such as
that induced by a Lefschetz fibration on its vertical boundary), it is often
advantageous to tweak the setup in order to arrange for the time $t$ flow
generated by $h$ to wrap ``by $t$ turns''. In any case, 
the time $t$ flow $\phi_h^t$ preserves the class of Lagrangians which are
conical at infinity. 

For every pair of objects $L_0,L_1$ under consideration, we assume that the set of 
times $t$ for which $L_0^t=\phi_h^t(L_0)$ and $L_1$ fail to intersect 
transversely is discrete. The Floer complex $CF(L_0^t,L_1)$ can be thought
of as a truncation of the previously considered wrapped Floer complex, where
the generators in the cylindrical end correspond only to Reeb chords of length at most $t$ from $L_0$ to
$L_1$ at the contact boundary.

For $\tau>0$ sufficiently small, $HF(L_0^{t+\tau},L_0^t)$ contains a
distinguished element, called ``quasi-unit'', generally defined via
Floer continuation (in the simplest cases it is the sum of the generators of the
Floer complex which correspond to the minima of $h$ on $L_0$).
The wrapped Fukaya category is then defined by localization with respect
to the class of quasi-units \cite{AS2,AG}; at the level of cohomology,
this means that \begin{equation}\label{eq:directlimit}HW(L_0,L_1):=\varinjlim\limits_{t\to\infty}
HF(L_0^t,L_1),\end{equation} where the Floer cohomology groups $HF(L_0^t,L_1)$ form
a direct system in which the connecting maps are given by multiplication with quasi-units. 
The chain-level construction of the quotient $A_\infty$-category is rather
cumbersome in general, and tends to be explicitly computable only in situations where
the continuation maps end up being chain-level isomorphisms or inclusions of
complexes for sufficiently large $t$. (See \cite{AS} for a more geometric
approach to the construction of the direct limit at chain level via
continuation maps.) 

The comparison between the two versions of wrapped Floer theory is well
beyond the scope of this survey. We simply note that, since $H$ grows faster
than $h$ at infinity, there are well-defined continuation maps from the
Floer complexes with linear Hamiltonians to those with quadratic
Hamiltonians; these chain maps are compatible with the quasi-units, 
and induce maps from the direct limit \eqref{eq:directlimit} to the
wrapped Floer cohomology defined in the previous section. The reverse
direction can be constructed by filtering the wrapped Floer complex by
action (i.e., length of Reeb chords) and approximating $H$ over arbitrarily
large portions of the cylindrical ends with linear-growth Hamiltonians.

\subsection{First examples: $\C^*$ and $(\C^*)^n$}\label{ss:cyl_calc}

As a warm-up, we consider $X=\C^*$, identified with $\R\times S^1$ via 
$z=\exp(r+i\theta)$, with the symplectic form
$\omega=dr\wedge d\theta=d\lambda$ where $\lambda=r\,d\theta$ is the
standard Liouville form of the cotangent bundle $T^*S^1$; the Liouville
vector field is $Z=r\partial_r$.
We view $X$ as the completion of $X^{in}=[-1,1]\times S^1$, whose boundary
$\partial X^{in}=\{\pm 1\}\times S^1$ carries the contact form
$\alpha=\lambda_{|\partial X^{in}}=\pm d\theta$, and identify $X\setminus X^{in}$
with the positive symplectization $(1,+\infty)\times \partial X^{in}$.
(We abusively denote by $r$ both the real coordinate on the whole space $X$
and the real positive coordinate on the symplectization of $\partial
X^{in}$, which is in fact $|r|$).

We calculate the wrapped Floer cohomology of $L_0=\R\times \{1\}\subset 
\R\times S^1$ (i.e., the real positive axis of $\C^*$).
The time 1 flow of the quadratic Hamiltonian $H=\frac12 r^2$ is given by
$\phi_H^1(r,\theta)=(r,\theta+r)$, and the generators of
the wrapped Floer complex $CW(L_0,L_0)$, i.e.\ the points of 
$\phi_H^1(L_0)\cap L_0$, are evenly spaced along the real axis (at integer
values of $r$); we
accordingly label them by integers: $\mathcal{X}(L_0,L_0)=\{x_i, i\in \Z\}$.
The generator $x_0$ which lies
at $r=0$ (the minimum of $H$) is an interior intersection point, whereas
the other generators $x_i$, $i\neq 0$ correspond to Reeb chords 
(in one cylindrical end or the other depending on the sign of $i$).

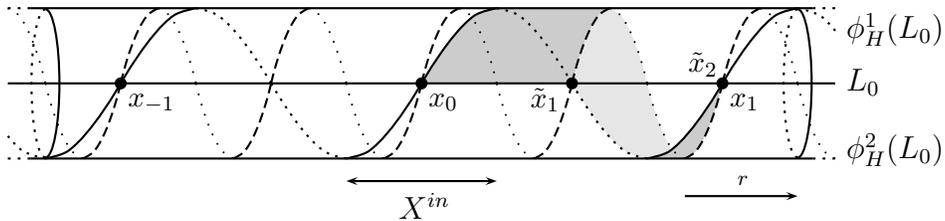
\begin{figure}[b]
\setlength{\unitlength}{1cm}
\begin{picture}(11,2.5)(-4.5,-1.5)
\psset{unit=\unitlength,dash=1pt 2.5pt}
\newgray{gray30}{0.8}
\newgray{gray15}{0.9}
\pscurve[linestyle=none,fillstyle=solid,fillcolor=gray15]%
 (3,-1)(2.7,-0.9)(2,0)(2,0)(1.3,0.9)(1,1)(1,1)(1,1)(2.5,1)(2.5,1)(2.5,1)%
 (2.65,0.9)(3,0)(3,0)(3.35,-0.9)(3.5,-1)
\pscurve[linestyle=none,fillstyle=solid,fillcolor=gray30]%
 (3,-1)(3.3,-0.9)(4,0)(4,0)(4,0)(3.65,-0.9)(3.5,-1)
\pscurve[linestyle=none,fillstyle=solid,fillcolor=gray30]%
 (0,0)(0.7,0.9)(1,1)(1,1)(1,1)(2.5,1)(2.5,1)(2.5,1)%
 (2.35,0.9)(2,0)
\psline(-5.2,-1)(5.2,-1)
\psline(-5.2,1)(5.2,1)
\psline[linestyle=dashed](-5.5,-1)(-5.2,-1)
\psline[linestyle=dashed](-5.5,1)(-5.2,1)
\psline[linestyle=dashed](5.5,-1)(5.2,-1)
\psline[linestyle=dashed](5.5,1)(5.2,1)
\psellipticarc(-5,0)(0.2,1){-90}{90}
\psellipticarc[linestyle=dashed](-5,0)(0.2,1){90}{270}
\psellipticarc(5,0)(0.2,1){-90}{90}
\psellipticarc[linestyle=dashed](5,0)(0.2,1){90}{270}
\put(-0.3,-1.8){$X^{in}$}
\psline{<->}(-1,-1.3)(1,-1.3)
\psline{->}(3.5,-1.5)(5,-1.5)
\put(4.2,-1.35){\tiny $r$}
\psline(-5.5,0)(5.5,0)
\put(5.65,-0.1){$L_0$}
\pscurve[linestyle=dashed](-5,-1)(-5.3,-0.9)(-5.5,-0.7)
\pscurve(-4,0)(-4.7,-0.9)(-5,-1)
\pscurve(-4,0)(-3.3,0.9)(-3,1)
\pscurve[linestyle=dashed](-2,0)(-2.7,0.9)(-3,1)
\pscurve[linestyle=dashed](-1,-1)(-1.3,-0.9)(-2,0)
\pscurve(-1,-1)(-0.7,-0.9)(0,0)
\pscurve(0,0)(0.7,0.9)(1,1)
\pscurve[linestyle=dashed](1,1)(1.3,0.9)(2,0)
\pscurve[linestyle=dashed](2,0)(2.7,-0.9)(3,-1)
\pscurve(3,-1)(3.3,-0.9)(4,0)
\pscurve(4,0)(4.7,0.9)(5,1)
\pscurve[linestyle=dashed](5,1)(5.3,0.9)(5.5,0.7)
\put(5.65,0.6){$\phi_H^1(L_0)$}
\pscircle*(0,0){0.08}
\pscircle*(4,0){0.08}
\pscircle*(-4,0){0.08}
\put(0.1,-0.3){\small $x_0$}
\put(4.1,-0.3){\small $x_1$}
\put(-3.9,-0.3){\small $x_{-1}$}
\pscurve[linestyle=dotted](-5,0)(-5.35,0.9)(-5.5,1)
\pscurve[linestyle=dotted](-4.5,-1)(-4.65,-0.9)(-5,0)
\pscurve[linestyle=dashed,dash=4pt 2pt](-4,0)(-4.35,-0.9)(-4.5,-1)
\pscurve[linestyle=dashed,dash=4pt 2pt](-4,0)(-3.65,0.9)(-3.5,1)
\pscurve[linestyle=dotted](-3,0)(-3.35,0.9)(-3.5,1)
\pscurve[linestyle=dotted](-2.5,-1)(-2.65,-0.9)(-3,0)
\pscurve[linestyle=dashed,dash=4pt 2pt](-2,0)(-2.35,-0.9)(-2.5,-1)
\pscurve[linestyle=dashed,dash=4pt 2pt](-2,0)(-1.65,0.9)(-1.5,1)
\pscurve[linestyle=dotted](-1,0)(-1.35,0.9)(-1.5,1)
\pscurve[linestyle=dotted](-0.5,-1)(-0.65,-0.9)(-1,0)
\pscurve[linestyle=dashed,dash=4pt 2pt](-0.5,-1)(-0.35,-0.9)(0,0)
\pscurve[linestyle=dashed,dash=4pt 2pt](0,0)(0.35,0.9)(0.5,1)
\pscurve[linestyle=dotted](0.5,1)(0.65,0.9)(1,0)
\pscurve[linestyle=dotted](1,0)(1.35,-0.9)(1.5,-1)
\pscurve[linestyle=dashed,dash=4pt 2pt](1.5,-1)(1.65,-0.9)(2,0)
\pscurve[linestyle=dashed,dash=4pt 2pt](2,0)(2.35,0.9)(2.5,1)
\pscurve[linestyle=dotted](2.5,1)(2.65,0.9)(3,0)
\pscurve[linestyle=dotted](3,0)(3.35,-0.9)(3.5,-1)
\pscurve[linestyle=dashed,dash=4pt 2pt](3.5,-1)(3.65,-0.9)(4,0)
\pscurve[linestyle=dashed,dash=4pt 2pt](4,0)(4.35,0.9)(4.5,1)
\pscurve[linestyle=dotted](4.5,1)(4.65,0.9)(5,0)
\pscurve[linestyle=dotted](5,0)(5.35,-0.9)(5.5,-1)
\put(5.65,-1){$\phi_H^2(L_0)$}
\pscircle*(2,0){0.08}
\put(1.45,-0.32){\small $\tilde{x}_1$}
\put(3.55,0.15){\small $\tilde{x}_2$}
\end{picture}
\caption{The wrapped Floer cohomology of $L_0=\R\times \{1\}$ in $X=\R\times S^1$.}
\label{fig:wrapcyl}
\end{figure}

There is a natural grading on $CW^*(L_0,L_0)$ (using the ``obvious''
trivialization of $TX$), for which the generators $x_i$ all
have degree zero. This implies immediately that the Floer differential
$\mu^1$ and the higher products $\mu^k$ for $k\ge 3$ vanish identically.
The vanishing of the differential can also be checked on Figure
\ref{fig:wrapcyl}: it is readily apparent that $L_0$ and $\phi_H^1(L_0)$ do
not bound any non-trivial pseudo-holomorphic strips. (Recall that,
in complex dimension 1, regardless of the almost-complex structure, rigid pseudo-holomorphic curves
correspond to immersed polygons with locally convex boundary.)

Since $L_0$ is invariant under the Liouville flow (which rescales
the $r$ coordinate), we can use the trick described at the end of \S
\ref{ss:wrap_quadratic} and view the product $\mu^2$ on the wrapped Floer
complex as a count of  pseudo-holomorphic discs with boundary on
$\phi_H^2(L_0)$, $\phi_H^1(L_0)$, and $L_0$. It is easy to check
that, for any $i,j\in \Z$, $\phi_H^1(x_i)\in
\phi_H^2(L_0)\cap \phi_H^1(L_0)$ and $x_j\in \phi_H^1(L_0)\cap L_0$ are
the vertices of a unique immersed triangle, whose third vertex
$\tilde{x}_{i+j}\in \phi_H^2(L_0)\cap L_0$ is mapped to $x_{i+j}\in
\phi_H^1(L_0)\cap L_0$ under the Liouville rescaling $r\mapsto 2r$. 
Hence, \begin{equation}\label{eq:cyl_mu2}\mu^2(x_j,x_i)=x_{i+j}.\end{equation}
For instance, the triangle shown on Figure \ref{fig:wrapcyl} contributes to $\mu^2(x_0,x_1)=x_1$. 

Renaming the generator $x_j$ to $x^j$, we conclude that
\begin{equation}\label{eq:HWC*}HW^*(L_0,L_0)\simeq \C[x,x^{-1}]\end{equation}
as algebras (or in fact as $A_\infty$-algebras with $\mu^k=0$ for $k\neq 2$).

Instead of the quadratic Hamiltonian $H$, one could instead use the
approach of \S \ref{ss:wrap_loc} with linear Hamiltonians. The resulting
picture looks like a truncation of Figure \ref{fig:wrapcyl}: the
angular coordinate $\theta$ increases from $-t$ at one end of
$L_0^t=\phi_h^t(L_0)$ to $+t$ at the other end, so $CF(L_0^t,L_0)$ 
only accounts for the generators $x_i$ with $|i|<t$. Taking the limit as
$t\to \infty$, one recovers \eqref{eq:HWC*}.

Either way, we find that $HW^*(L_0,L_0)$ is isomorphic to $\mathrm{Ext}^*(\O,\O)\simeq
\C[x,x^{-1}]$ for the structure sheaf on the mirror $X^\vee=\C^*=\mathrm{Spec}\,\C[x,x^{-1}]$.

By a result of Abouzaid, 
$L_0$ generates the wrapped Fukaya
category $\W(X)$: this means that every object is quasi-isomorphic to an
iterated mapping cone built from (finitely many) copies of $L_0$. (A weaker
notion, that of ``split-generation'', adds formal direct summands in such
iterated mapping cone; it is not needed here). Similarly, the structure
sheaf $\O$ generates $\mathrm{Coh}(X^\vee)$.
By standard homological algebra, this implies that there is a derived 
equivalence between the wrapped Fukaya category of $X=\R\times S^1$ and 
the category of coherent sheaves on $X^\vee=\C^*$.

In fact, $\W(X)$ and $D^b\mathrm{Coh}(X^\vee)$ are both
equivalent to the category of perfect complexes of modules over the algebra
$\mathcal{A}=\C[x,x^{-1}]$, via Yoneda embedding: on the symplectic side, the
$A_\infty$-module associated to an object $\Theta\in\W(X)$ is the wrapped
Floer complex $CW^*(L_0,\Theta)$ viewed as an $A_\infty$-module over $CW^*(L_0,L_0)$
(where the structure maps of the module are induced by those of
the wrapped Fukaya category), while on the mirror side, the module structure
just comes from multiplication by regular functions.

The argument extends in a straightforward manner to the case
of $X=(\C^*)^n\simeq T^*T^n$, whose wrapped Fukaya category is generated by
$L_0=(\R_+)^n$. Indeed, the standard quadratic
Hamiltonian ($H=\frac12 \sum r_i^2$, where $r_i=\log |z_i|$) preserves the
product structure, and holomorphic triangles with boundary on
$\phi_H^2(L_0)$, $\phi_H^1(L_0)$ and $L_0$ can be studied by projecting to
each coordinate. One finds that
\begin{equation}
HW^*(L_0,L_0)=CW^*(L_0,L_0)\simeq \C[x_1^{\pm 1},\dots,x_n^{\pm 1}],
\end{equation}
which agrees with the ring of functions of the mirror
$X^\vee=(\C^*)^n$. Viewing $\W((\C^*)^n)$ and $D^b\Coh((\C^*)^n)$ in terms of perfect
complexes of
modules over $\C[x_1^{\pm 1},\dots,x_n^{\pm 1}]$, homological mirror 
symmetry follows.
 
\section{Example: the pair of pants}\label{s:pants1}

In this section, we consider the pair of 
pants $X=\PP^1\setminus \{0,-1,\infty\}=\C^*\setminus \{-1\}$.
Homological mirror symmetry for this example has been studied in
\cite{AAEKO}; we review the results of that paper from a slightly different
perspective.

The details of the Liouville structure on $X$ are not particularly
important, except as a warmup for the general setup considered in the
following sections. Viewing $X$ as the hypersurface in $(\C^*)^2$ defined by
the equation $x_1+x_2+1=0$, we use the Liouville structure induced by that of
$(\C^*)^2$. Namely, writing $x_j=\exp(r_j+i\theta_j)$, we set
$\omega=\frac12 dd^c(r_1^2+r_2^2)$, and the Liouville form is
$\lambda=\frac12 d^c(r_1^2+r_2^2)=r_1\,d\theta_1+r_2\,d\theta_2$. (In terms of the coordinate
$z$ on $\C^*\setminus \{-1\}$, $r_1=\log |z|$ and $r_2=\log |z+1|$.)

One could instead view $X$ as the complement of the hypersurface defined by
$f(z)=z+1=0$ in $\C^*$. A natural choice of K\"ahler potential is then
$\frac12(\log |z|)^2+\frac12(\log |f|)^2$, which gives exactly the same
formula. However, we will often prefer to modify this prescription, using
cut-off functions so
that the K\"ahler potential is in fact equal to $\frac12(\log |f|)^2$ near $-1$
and to $\frac12(\log |z|)^2$ outside of a neighborhood of $-1$. This offers the advantage that
the Liouville structure is the same as that of $\C^*$ outside of a neighborhood
of the deleted hypersurface, and ``standard'' near the hypersurface.
Likewise, the Hamiltonian used to define the wrapped Fukaya category of
$X$ can be chosen to coincide with that used for $\C^*$ away from
a neighborhood of~$-1$.

In the same vein, when viewing $X$ as the complement of a hypersurface in
$\C^*$ the most natural choice of gradings in Floer theory uses
the trivialization of the tangent bundle induced by that of $\C^*$. 
This means that the puncture
at $-1$ is graded differently from those at $0$ and $\infty$, across which
the trivialization does not extend.

\begin{figure}[b]
\setlength{\unitlength}{1cm}
\begin{picture}(11,2.5)(-4.5,-1.5)
\psset{unit=\unitlength,dash=1pt 2.5pt}
\newgray{gray30}{0.8}
\newgray{gray15}{0.9}
\pscurve[linestyle=none,fillstyle=solid,fillcolor=gray15]%
 (3,-1)(2.7,-0.9)(2,0)(2,0)(1.3,0.9)(1,1)(1,1)(1,1)(2.5,1)(2.5,1)(2.5,1)%
 (2.65,0.9)(3,0)(3,0)(3.35,-0.9)(3.5,-1)
\pscurve[linestyle=none,fillstyle=solid,fillcolor=gray30]%
 (3,-1)(3.3,-0.9)(4,0)(4,0)(4,0)(3.65,-0.9)(3.5,-1)
\pscurve[linestyle=none,fillstyle=solid,fillcolor=gray30]%
 (0,0)(0.7,0.9)(1,1)(1,1)(1,1)(2.5,1)(2.5,1)(2.5,1)%
 (2.35,0.9)(2,0)
\psline(-5.2,-1)(-0.2,-1)
\psline(1.2,-1)(5.2,-1)
\psline(-5.2,1)(5.2,1)
\pscurve(-0.2,-1)(0,-1)(0.2,-1.3)(0.2,-1.6)
\pscurve(1.2,-1)(1,-1)(0.8,-1.3)(0.8,-1.6)
\psellipticarc(0.5,-1.5)(0.3,0.1){-180}{0}
\psellipticarc[linestyle=dashed](0.5,-1.5)(0.3,0.1){0}{180}
\put(1,-1.5){\small $-1$}
\psline[linestyle=dashed](-5.5,-1)(-5.2,-1)
\psline[linestyle=dashed](-5.5,1)(-5.2,1)
\psline[linestyle=dashed](5.5,-1)(5.2,-1)
\psline[linestyle=dashed](5.5,1)(5.2,1)
\psellipticarc(-5,0)(0.2,1){-90}{90}
\psellipticarc[linestyle=dashed](-5,0)(0.2,1){90}{270}
\psellipticarc(5,0)(0.2,1){-90}{90}
\psellipticarc[linestyle=dashed](5,0)(0.2,1){90}{270}
\psline(-5.5,0)(5.5,0)
\put(5.65,-0.1){$L_0$}
\pscurve[linestyle=dashed](-5,-1)(-5.3,-0.9)(-5.5,-0.7)
\pscurve(-4,0)(-4.7,-0.9)(-5,-1)
\pscurve(-4,0)(-3.3,0.9)(-3,1)
\pscurve[linestyle=dashed](-2,0)(-2.7,0.9)(-3,1)
\pscurve[linestyle=dashed](-1,-1)(-1.3,-0.9)(-2,0)
\pscurve(-1,-1)(-0.7,-0.9)(0,0)
\pscurve(0,0)(0.7,0.9)(1,1)
\pscurve[linestyle=dashed](1,1)(1.3,0.9)(2,0)
\pscurve[linestyle=dashed](2,0)(2.7,-0.9)(3,-1)
\pscurve(3,-1)(3.3,-0.9)(4,0)
\pscurve(4,0)(4.7,0.9)(5,1)
\pscurve[linestyle=dashed](5,1)(5.3,0.9)(5.5,0.7)
\put(5.65,0.6){$\phi_H^1(L_0)$}
\pscircle*(0,0){0.08}
\pscircle*(4,0){0.08}
\pscircle*(-4,0){0.08}
\put(0.1,-0.3){\small $x_0$}
\put(4.1,-0.3){\small $x_1$}
\put(-3.9,-0.3){\small $x_{-1}$}
\pscurve[linestyle=dotted](-5,0)(-5.35,0.9)(-5.5,1)
\pscurve[linestyle=dotted](-4.5,-1)(-4.65,-0.9)(-5,0)
\pscurve[linestyle=dashed,dash=4pt 2pt](-4,0)(-4.35,-0.9)(-4.5,-1)
\pscurve[linestyle=dashed,dash=4pt 2pt](-4,0)(-3.65,0.9)(-3.5,1)
\pscurve[linestyle=dotted](-3,0)(-3.35,0.9)(-3.5,1)
\pscurve[linestyle=dotted](-2.5,-1)(-2.65,-0.9)(-3,0)
\pscurve[linestyle=dashed,dash=4pt 2pt](-2,0)(-2.35,-0.9)(-2.5,-1)
\pscurve[linestyle=dashed,dash=4pt 2pt](-2,0)(-1.65,0.9)(-1.5,1)
\pscurve[linestyle=dotted](-1,0)(-1.35,0.9)(-1.5,1)
\pscurve[linestyle=dotted](-0.5,-1)(-0.65,-0.9)(-1,0)
\pscurve[linestyle=dashed,dash=4pt 2pt](-0.5,-1)(-0.35,-0.9)(0,0)
\pscurve[linestyle=dashed,dash=4pt 2pt](0,0)(0.35,0.9)(0.5,1)
\pscurve[linestyle=dotted](0.5,1)(0.65,0.9)(1,0)
\pscurve[linestyle=dotted](1,0)(1.35,-0.9)(1.5,-1)
\pscurve[linestyle=dashed,dash=4pt 2pt](1.5,-1)(1.65,-0.9)(2,0)
\pscurve[linestyle=dashed,dash=4pt 2pt](2,0)(2.35,0.9)(2.5,1)
\pscurve[linestyle=dotted](2.5,1)(2.65,0.9)(3,0)
\pscurve[linestyle=dotted](3,0)(3.35,-0.9)(3.5,-1)
\pscurve[linestyle=dashed,dash=4pt 2pt](3.5,-1)(3.65,-0.9)(4,0)
\pscurve[linestyle=dashed,dash=4pt 2pt](4,0)(4.35,0.9)(4.5,1)
\pscurve[linestyle=dotted](4.5,1)(4.65,0.9)(5,0)
\pscurve[linestyle=dotted](5,0)(5.35,-0.9)(5.5,-1)
\put(5.65,-1){$\phi_H^2(L_0)$}
\pscircle*(2,0){0.08}
\put(1.45,-0.32){\small $\tilde{x}_1$}
\put(3.55,0.15){\small $\tilde{x}_2$}
\end{picture}
\caption{The wrapped Floer cohomology of $L_0=\R_+$ in $X=\C^*\setminus \{-1\}$.}
\label{fig:wrappants}
\end{figure}
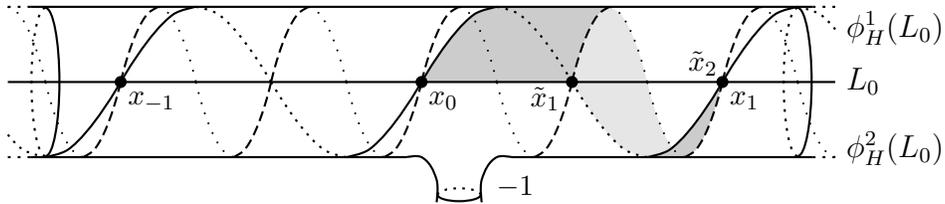

With this understood, we consider again $L_0=\R_+\subset X$.  Since $L_0$
(and its image under the flow generated by $H$) stay away from the puncture at
$z=-1$, the calculation of the wrapped Floer complex closely parallels the
case of the cylinder (cf.\ \S \ref{ss:cyl_calc}). Namely, the generators of
$CW^*(L_0,L_0)$ are still evenly spaced along the real positive axis,
$\mathcal{X}(L_0,L_0)=\{x_i,\, i\in \Z\}$, with $\deg(x_i)=0$, the
operations $\mu^k$ $(k\neq 2)$ vanish for degree reasons, and the
calculation of the product structure $\mu^2$ proceeds as before; see Figure
\ref{fig:wrappants}.
The only difference with the case of the cylinder is that only immersed 
triangles that do not pass through the puncture at $-1$ contribute to
$\mu^2$. Observing that the triangles of Figure \ref{fig:wrapcyl} that pass
through the central region containing $-1$ are exactly those whose inputs
lie in opposite ends of the cylinder, we find that
\begin{equation}
\mu^2(x_j,x_i)=\begin{cases} 
x_{i+j}\quad \text{if}\ ij\ge 0,\\
0 \quad \text{if}\ ij<0.
\end{cases}
\end{equation}
Thus, renaming $x_{-j}$ to $z_1^j$ and $x_{j}$ to $z_2^j$ for $j>0$, we have:
\begin{equation}
HW^*(L_0,L_0)\simeq \C[z_1,z_2]/(z_1z_2=0).
\end{equation}
Denoting this algebra by $\mathcal{A}$, this suggests that a mirror to $X$
might be
\begin{equation}\label{eq:xy=0}
X^\vee=\mathrm{Spec}\,\mathcal{A}=\{(z_1,z_2)\in \C^2\,|\,z_1z_2=0\}.
\end{equation}
This is indeed the case, but not for any obvious reason as far as we know. Indeed,
$L_0$ does not split-generate the wrapped Fukaya category, and it is not
readily evident that the Yoneda functor from $\W(X)$
into the derived category of $A_\infty$-modules over $\mathcal{A}$, given
by $\Theta \mapsto CW^*(L_0,\Theta)$ on objects, is fully faithful and its
image (which does not solely consist of perfect complexes)
agrees with the derived category of coherent sheaves on $X^\vee$.

Observe by the way that, if we view $L_0$ as an object of the {\em
relative} wrapped Fukaya category of $(\C^*,\{-1\})$ in the sense of Seidel, 
i.e.\ if we count holomorphic curves in $\C^*$ which
pass $k$ times through $-1$ with a coefficient of $t^k$, then the
wrapped Floer cohomology of $L_0$ becomes $HW^*(L_0,L_0)\simeq
\C[z_1,z_2,t]/(z_1z_2=t)$. This exhibits the mirror to the complement
$X=\C^*\setminus \{-1\}$, given by \eqref{eq:xy=0}, as the central fiber of a degeneration of the mirror to
$\C^*$. This is a general feature, as noted in the discussion after
Conjecture \ref{conj:hmscompl}.

Returning to our study of the wrapped Fukaya category, 
one can show that $\W(X)$ is (split) generated by the three
components of the real locus of $X$, namely
$L_0=\R_+$, $L_1=(-\infty,-1)$, and $L_2=(-1,0)$. (In fact, any two of these
suffice.) Calculating their wrapped Floer complexes and the product
structures is a simple exercise similar to the above case of $L_0$; the
outcome is as follows (see \cite {AAEKO} for details). First, we have:
\begin{align*}
HW^*(L_0,L_0)&\simeq\C[z_1,z_2]/(z_1z_2=0),\\
HW^*(L_1,L_1)&\simeq\C[z_2,z_0]/(z_2z_0=0),\\
HW^*(L_2,L_2)&\simeq\C[z_1,z_0]/(z_1z_0=0),
\end{align*}
as (formal) graded ($A_\infty$-)algebras. Here we denote by $z_1$ (resp.\ $z_2,z_0$) the generators corresponding to Reeb
chords that wrap once around $0$ (resp.\ $\infty,-1$). With our
choice of trivialization of $TX$, $\deg z_1=\deg z_2=0$, whereas $\deg z_0=2$.
For $i\neq j$, $HW^*(L_i,L_j)$ is an ($A_\infty$) bimodule over
$HW^*(L_i,L_i)$ and $HW^*(L_j,L_j)$. As such it is generated by a single
generator $u_{ij}$ corresponding to a Reeb chord that wraps halfway around
the common cylindrical end, and we have
\begin{align*}
HW^*(L_0,L_1)&\simeq \C[z_2]\,u_{01}, & HW^*(L_1,L_0)&\simeq \C[z_2]\,u_{10},\\
HW^*(L_1,L_2)&\simeq \C[z_0]\,u_{12}, & HW^*(L_2,L_1)&\simeq \C[z_0]\,u_{21},\\
HW^*(L_2,L_0)&\simeq \C[z_1]\,u_{20}, & HW^*(L_0,L_2)&\simeq \C[z_1]\,u_{02},
\end{align*}
with the bimodule structure implied by the notations (any variable not
present in the notation acts by zero), and vanishing higher module maps.
(Here $u_{12}$ and $u_{21}$ have degree 1 and the other generators have 
degree 0.)

Moreover, for $\{i,j,k\}=\{0,1,2\}$ we have
$\mu^2(u_{ji},u_{ij})=z_k$ and $\mu^2(u_{jk},u_{ij})=0$, whereas
$\mu^3(u_{ki},u_{jk},u_{ij})=-\mathrm{id}_{L_i}$. In particular there are
two exact triangles
\begin{equation}
\label{eq:2exacttri}
L_2\stackrel{u_{20}\,}{\longrightarrow} L_0\stackrel{u_{01}\,}{\longrightarrow}
L_1\stackrel{u_{12}\,}{\longrightarrow} L_2[1]\quad \text{and}\quad
L_1\stackrel{u_{10}\,}{\longrightarrow} L_0\stackrel{u_{02}\,}{\longrightarrow}
L_2\stackrel{u_{21}\,}{\longrightarrow} L_1[1].
\end{equation}
It is shown in \cite{AAEKO} that this completely determines the
$A_\infty$-structure up to homotopy.

The corresponding calculation on $X^\vee$ is as follows: we consider the
structure sheaf $\O$, and the structure sheaves $\O_A$ and $\O_B$ of the two
irreducible components of $X^\vee$, $A:\{z_1=0\}$ and $B:\{z_2=0\}$. 
In the language of
modules over $\A=\C[z_1,z_2]/(z_1z_2)$, these correspond to $\A$, $\A/(z_1)$, and $\A/(z_2)$.
To calculate Ext groups between these objects, we use the (infinite,
2-periodic) projective resolution
$$\dots \stackrel{z_1}{\longrightarrow} \O \stackrel{z_2}{\longrightarrow}
\O \stackrel{z_1}{\longrightarrow} \O \longrightarrow \O_A \to 0$$
and similarly for $\O_B$ (exchanging $z_1$ and $z_2$).
For example, applying $\mathrm{Hom}(-,\O_A)$ to this resolution we find that 
$\Ext^*(\O_A,\O_A)$ is given by the cohomology of
$$0\to \A/(z_1)\stackrel{0}{\longrightarrow}\A/(z_1)
\stackrel{z_2}{\longrightarrow}\A/(z_1)\stackrel{0}{\longrightarrow}\dots$$
This gives $\Hom(\O_A,\O_A)\simeq \C[z_2]$, and $\Ext^{2k}(\O_A,\O_A)=\C$ for all
$k\ge 1$. Denoting by $z_0$ the generator of $\Ext^2(\O_A,\O_A)$, further
calculations show that, as an algebra, $\Ext^*(\O_A,\O_A)\simeq
\C[z_2,z_0]/(z_2z_0=0).$  Similarly for the other Ext groups and module
structures; the outcomes match exactly the calculations on the symplectic
side. Moreover, there are two short exact sequences
$$0\to \O_A \to \O \to \O_B \to 0$$
(where the first map is the homomorphism from $\A/(z_1)$ to $\A$ given by
multiplication by $z_2$, and the second map is the projection from $\A$ to
$\A/(z_2)$) and
$$0\to \O_B \to \O \to \O_A \to 0,$$
which give rise to two exact triangles in the derived category. This in turn
suffices to conclude that $\W(X)\simeq D^b\Coh(X^\vee)$.

Given that the pair of pants has a 3-fold symmetry that is not apparent on
the mirror $X^\vee$, it is natural to ask for a more symmetric version of
mirror symmetry. The answer comes in the form of the Landau-Ginzburg model
$(Y=\C^3,W=-z_0z_1z_2)$: namely, $\W(X)$ is also equivalent to the
triangulated category of singularities of $(Y,W)$ \cite{Orlov}, i.e.\ 
the quotient of the derived category of coherent sheaves of the zero
fiber $Z=W^{-1}(0)=\{z_0z_1z_2=0\}$ by the subcategory of perfect complexes,
$D^b_\sg(Z)=D^b\Coh(Z)/\Perf(Z)$. (Since $Z$ is affine, 
$\Perf(Z)$ is generated by $\O_Z$.)

The category $D^b_\sg(Z)$ is generated by $\O_{Z_0},\O_{Z_1},\O_{Z_2}$, where $Z_i$ are
the irreducible components of $Z$, i.e.\ the hyperplanes
$\{z_i=0\}$.
The short exact sequence of sheaves 
$$0\to \O_{Z_0}\to \O_Z\to\O_{Z_1\cup Z_2}\to 0$$ induces an isomorphism
$\O_{Z_1\cup Z_2}\simeq \O_{Z_0}[1]$ in $D^b_\sg(Z)$, and hence we have 
two exact triangles in $D^b_\sg(Z)$, \begin{equation}
\O_{Z_1}\to \O_{Z_0}[1]\to \O_{Z_2} \to \O_{Z_1}[1]
\quad\text{and}\quad
\O_{Z_2}\to \O_{Z_0}[1]\to \O_{Z_1} \to \O_{Z_2}[1].
\end{equation}
Note that the natural grading on $D^b_\sg(Z)$ is only by $\Z/2$.
($\Z$-gradings manifestly exist in this instance, but break the symmetry
between the coordinates.) 
The most efficient method of computation of morphisms in $D^b_\sg(Z)$ is
via 2-periodic resolutions, as 
\begin{equation}\label{eq:hom_is_ext}
\Hom^i_{D^b_\sg(Z)}(\mathcal{E}_1,\mathcal{E}_2)\cong 
\Ext^{2k+i}_{D^b\Coh(Z)}(\mathcal{E}_1,\mathcal{E}_2)\qquad \text{for $k\gg 0$.}
\end{equation}
(see Proposition 1.21 of \cite{Orlov}).

The calculations for the generators $\O_{Z_i}$ are carried out in
\cite{AAEKO}, and yield the same answers as the corresponding calculations
in $D^b\Coh(X^\vee)$ and $\W(X)$. 
Hence, we have equivalences $\W(X)\simeq D^b\Coh(X^\vee)\simeq
D^b_{\sg}(Z)$, under which the generators $L_0,L_1,L_2$ of $\W(X)$ correspond to
$\O,\O_A,\O_B$ in $D^b\Coh(X^\vee)$, and to $\O_{Z_0}[1],\O_{Z_1},\O_{Z_2}$
in $D^b_\sg(Z)$.

The equivalence between $D^b\Coh(X^\vee)$ and $D^b_\sg(Z)$ is a special case
of Orlov's Kn\"orrer periodicity result \cite{Orlov,OrlovKnorrer}.  

From our
perspective, the Landau-Ginzburg model
$(\C^3,-z_0z_1z_2)$ is the mirror to the pair of pants viewed as a
hypersurface $x_1+x_2+1=0$ in $(\C^*)^2$, and the equivalence $\W(X)\simeq
D^b_\sg(Z)$ is an instance of Conjecture \ref{conj:hms}. Meanwhile,
$X^\vee=\{z_1z_2=0\}\subset\C^2$ is the zero fiber of the
Landau-Ginzburg model $(\C^2,-z_1z_2)$, which our construction
associates to $\{-1\}$ viewed as a hypersurface in $\C^*$. Thus, viewing the pair
of pants as the complement $\C^*\setminus \{-1\}$, we are now in the setting
of Conjecture \ref{conj:hmscompl}, and the equivalence
$\W(X)\simeq D^b\Coh(X^\vee)$ is the top row of the diagram \eqref{eq:hmscompl}.
The bottom row is the
equivalence $\W(\{-1\})\simeq D^b_\sg(X^\vee)$, which maps the generator
(the Lagrangian consisting of the point itself) to the
generator of $D^b_\sg(X^\vee)\ (\simeq D^b\mathrm{Vect})$.
It is then easy to check that the diagram
\eqref{eq:hmscompl} commutes. Indeed,
the restriction functor $\rho$ maps $L_0$ (which avoids the puncture
at $-1$) to the zero object, while $L_1$ and $L_2$ (which each have one end 
at $-1$) map to the generator of $\W(\{-1\})$ and its shift by one.
This is in agreement with the images of $\O$, $\O_A$ and $\O_B$ under the
quotient functor $D^b\Coh(X^\vee)\to D^b_\sg(X^\vee)$.

\section{The complement of $H$ and the restriction functor $\rho$}\label{s:hmscompl}
Let $H=f^{-1}(0)\subset (\C^*)^n$ be a smooth algebraic hypersurface as in
the introduction. 
The standard Liouville 
structure on $H$ is induced by that of $(\C^*)^n$ as
follows. Expressing the coordinates on $(\C^*)^n$ in the form $x_j=\exp(r_j+i\theta_j)$,
 the standard K\"ahler form
of $(\C^*)^n$ is given by the K\"ahler potential $\Phi=\frac12\sum r_j^2$, i.e.\
$\omega=dd^c\Phi$, and the standard Liouville form is $\lambda=d^c\Phi=
\sum r_j\,d\theta_j$. The K\"ahler potential, symplectic form and Liouville form of $H$ are
then simply the restrictions of $\Phi$, $\omega$ and $\lambda$ to $H$, given 
by the same formulas in coordinates. (However, one may also choose a
different Liouville structure in the same deformation class as convenient
for calculations.)

The natural choice of Liouville structure on the complement
$(\C^*)^n\setminus H$ is given by the K\"ahler potential
$\hat\Phi=\Phi + \frac12(\log |f|)^2$, i.e.\ the Liouville form is
$\hat\lambda=d^c\hat\Phi=\sum r_j\,d\theta_j+\log |f|\,d\arg(f)$.
Observing that $(\C^*)^n\setminus H$ is isomorphic to the
hypersurface in $(\C^*)^{n+1}$ defined by $f(x_1,\dots,x_n)+x_{n+1}=0$,
the K\"ahler potential $\hat\Phi$ and Liouville form $\hat\lambda$ are
exactly those induced by the standard choices on $(\C^*)^{n+1}$.

In order to construct the restriction functor $\rho:\W((\C^*)^n\setminus H)
\to \W(H)$, it is advantageous to deform the Liouville structure (which
does not modify the wrapped Fukaya category) in order to make it apparent that
$(\C^*)^n\setminus H$ contains a Liouville subdomain equivalent to the
product of $H$ with a punctured disc $\mathbb{D}^*$. Namely, for $K\gg 0$
sufficiently large, the potential $\hat\Phi_K=\Phi+\frac12(\log |f|+K)^2$
defines the same K\"ahler form on $(\C^*)^n\setminus H$ (since $dd^c\log |f|=0$), but the
corresponding Liouville form is $\hat\lambda_K=\hat\lambda+K\,d\,\arg(f)$
and the Liouville vector fields differ by $K\,\nabla\log |f|$. Thus,
for $K$ sufficiently large, the Liouville vector field of $\hat\lambda_K$
is transverse and outward pointing along arbitrarily large compact subsets
of the hypersurface $|f|=\epsilon$ (for fixed $\epsilon$ with 
$e^{-K}\ll \epsilon\ll 1$).  We note that this modification amounts to
rescaling $f$ to $e^K f$.

With this understood, intersecting the subset of $(\C^*)^n\setminus H$ 
where $|f|<\epsilon$ with a large compact subset of $(\C^*)^n$ defines a Liouville
subdomain which is a topologically trivial fibration over the punctured
disc, and Liouville deformation equivalent to the product of 
$H^{in}\subset H$ with a punctured disc.  The completion of this subdomain is
(up to Liouville deformation equivalence) $H\times \C^*$.
In this setting, the work of Abouzaid and Seidel~\cite{AS} yields a
restriction functor
\begin{equation}
r:\W((\C^*)^n\setminus H)\to \W(H\times \C^*).
\end{equation}

The diagram of Conjecture \ref{conj:hmscompl} relies on the use of a particular $\Z$-grading on
$\W((\C^*)^n\setminus H)$, defined by a choice of trivialization of the
determinant line bundle of the tangent bundle of $(\C^*)^n\setminus H$.
We use the trivialization obtained by restricting to the complement of
$H$ the standard trivialization for $(\C^*)^n$. Hence, the $\Z$-grading
that we consider on $\W(H\times \C^*)$ is not the
``usual'' one, but rather comes from a trivialization that extends to
$H\times\C$, which shifts by $2k$ the degree of Reeb chords that 
wrap $k$ times around the origin in $\C^*$ (whereas in the $H$ factor we have the
trivialization induced by that of $(\C^*)^n$ via interior product with $df$).

The second step in the construction of $\rho$ is to define
a $\Z/2$-graded ``projection'' functor
$p:\W(H\times \C^*)\to \W(H)$ as an adjoint to the
inclusion $i:\W(H)\to \W(H\times \C^*)$ which maps $\ell$ to
$i(\ell)=\ell\times \R_+$.
Given any two objects $\ell_1,\ell_2\in \W(H)$, we have
\begin{equation}\label{eq:CWtensor}
CW^*(\ell_1\times \R_+,\ell_2\times \R_+)\cong CW^*(\ell_1,\ell_2)\otimes_\C
\C[z^{\pm 1}],\end{equation}
where $\deg(z)=2$, with all $A_\infty$-operations
extended linearly. This allows us to define $i$ on morphisms by $i(x)=x\otimes
1$; the higher terms vanish. 
We first define a version of $p$ which takes values in a module category,
$$\hat{p}:\W(H\times \C^*)\to \text{mod-}\W(H).$$
Given an object $L$ of $\W(H\times\C^*)$, the $\W(H)$-module $\hat{p}(L)$
associates to $\ell \in \W(H)$ the chain complex
$\hat{p}(L)(\ell):=CW^*(\ell\times \R_+,L)$. The structure maps of the
$A_\infty$-module $\hat{p}(L)$ come from the $A_\infty$-operations in $\W(H\times\C^*)$ 
(via the inclusion $i$). The definition of $\hat{p}$ on morphisms is
tautological and parallels the construction of the Yoneda embedding.

Given $\ell\in \W(H)$, the cohomological unit $e_\ell\in CW^0(\ell,\ell)$
gives rise to a degree 2 automorphism $e_\ell\otimes z\in CW^2(\ell\times \R_+,
\ell\times \R_+)$. In particular, multiplication by $e_\ell\otimes z$ induces
quasi-isomorphisms
\begin{equation}\label{eq:CWperiodic}
CW^*(\ell\times \R_+,L)\stackrel{\simeq}{\longrightarrow}
CW^{*+2}(\ell\times \R_+,L)
\end{equation}
for all $\ell$ and $L$, so that the modules
$\hat{p}(L)$ are 2-periodic. 

Identifying the graded pieces of $\hat{p}(L)$ of given parity
via the quasi-isomorphisms \eqref{eq:CWperiodic}, we arrive at a
$\Z/2$-graded module that we denote by $\bar{p}(L)$. In fact, $\hat{p}$
induces a functor
$\bar{p}$ from $\W(H\times \C^*)$ to a category of $\Z/2$-graded
modules over $\W(H)$. Constructing $\bar{p}$ carefully involves a significant
amount of work; conceptually, the key point
is that the isomorphisms $e_\ell \otimes z$ are part of a natural
transformation from the identity functor of $\W(H\times\C^*)$ to the shift functor $[2]$ induced by rotation
of the $\C^*$ factor.

Next, we observe that the wrapped Fukaya category of $H\times \C^*$ is
split-generated by products $\ell\times \R_+$. 
Moreover, the isomorphism \eqref{eq:CWtensor} implies that,
as a $\Z/2$-graded module, $\bar{p}(\ell\times \R_+)$ 
is isomorphic to the Yoneda module of $\ell$.
It follows that $\bar{p}$ is representable, i.e.\ 
there is a functor
\begin{equation}
p:\W(H\times \C^*)\to \W(H)
\end{equation}
into the ($\Z/2$-graded, split-closed derived) wrapped Fukaya
category of $H$ such that $\bar{p}$ is the composition of $p$ with Yoneda
embedding. Finally, we set
\begin{equation}
\rho = p\circ r:\W((\C^*)^n\setminus H)\to \W(H).
\end{equation}

\begin{remark}\label{rmk:nattrans}
The family of closed orbits of the Reeb vector field which wrap once around 
$H$ in unit time determines a class $\theta \in SH^2((\C^*)^n\setminus H)$
which, via the closed-open map, induces a natural transformation
$\Theta:\mathrm{id}\to [2]$ acting on $\W((\C^*)^n\setminus H)$.
The restriction of $\Theta$ to the subdomain $H\times \C^*$ is exactly the degree 2
natural transformation used to construct $\bar{p}$ from $\hat{p}$.
With this understood, $\rho$ can be characterized in terms
of localization along the natural transformation $\Theta$: for $L_1,L_2\in
\W((\C^*)^n\setminus H)$,
\begin{equation}\label{eq:HW_is_limit}
HW^i(\rho(L_1),\rho(L_2))\cong \varinjlim_{k\to \infty}
HW^{2k+i}(L_1,L_2),\end{equation}
where the direct limit on the right-hand side is with respect to 
multiplication by $[\Theta_{L_1}]\in HW^2(L_1,L_1)$ (or equivalently,
$[\Theta_{L_2}]\in HW^2(L_2,L_2)$).
This can be viewed both as a ``global'' version of \eqref{eq:CWtensor}
and as a mirror counterpart to \eqref{eq:hom_is_ext}.
\end{remark}

\begin{remark}\label{rmk:framed}
Many of the Lagrangian submanifolds of $(\C^*)^n\setminus H$ that we will
consider below are ``framed'', i.e.\ have the property that $\arg(f)$ is equal 
to zero (or some other fixed constant value) near $H$. Near $H$ such a
Lagrangian is obtained by parallel transport of a Lagrangian
submanifold of $H$ in the fibers of $f$ over a radial arc. 
Thus, the restriction to the
subdomain $H\times \C^*$ is a product $\ell\times\R_+$, and the image
under $\rho$ is simply $\ell$.
\end{remark}

\section{The Fukaya category of the Landau-Ginzburg model $((\C^*)^n,f)$}\label{s:FS}

From now on, we assume that the Laurent polynomial $f$ has a non-trivial
constant term; rescaling $f$ is necessary we will assume that the constant
term is equal to $1$.
We briefly review Abouzaid's version of the Fukaya-Seidel category of
$((\C^*)^n,f)$ \cite{AbToric,AbToric2}, modified to suit our purposes. 
(Various other constructions are also worth mentioning: see \cite{SeLef,Sylvan,AS}.
For our purposes each of these brings with it some desirable features and
some unwanted complications.)

Fix a regular value $c_0$ of $f$, and a simply connected domain 
$\Omega\subset \C$ such that $c_0$ lies on the boundary of $\Omega$.
(Typically we require $\Omega$ to contain all the critical values
of~$f$.)
The objects of $\F((\C^*)^n,f)$ are properly embedded {\em admissible} exact Lagrangian submanifolds of $(\C^*)^n$
with boundary in the fiber $f^{-1}(c_0)$. A Lagrangian submanifold $L$ is
said to be admissible if $f(L)\subset\Omega$ and, in a neighborhood of
$\partial L$, $f_{|L}$ takes values in a smoothly embedded arc $\gamma$
(e.g.\ a straight half-line) whose tangent vector at $c_0$ points into the interior of $\Omega$.
Note that, near its boundary, an admissible Lagrangian $L$ is obtained by parallel transport of
$\partial L\subset f^{-1}(c_0)$ in the fibers of $f$ over the arc $\gamma$.

If the objects of interest include Lagrangian submanifolds which are
non-compact in the fiber direction, we further assume that $f^{-1}(c_0)$
is preserved by the Liouville flow, and that the wrapping Hamiltonian $H$
is well-behaved on admissible Lagrangians. (If necessary this can be 
ensured by a modification of the Liouville structure to a local product model
near $f^{-1}(c_0)$.)


We say that a pair of admissible Lagrangians $(L_0,L_1)$ projecting
to arcs $\gamma_0,\gamma_1$ near their boundary is in positive
position, and write $L_0<L_1$, if the tangent vector to $\gamma_1$ at $c_0$
(pointing into the interior of $\Omega$) points ``to the left'' 
(counterclockwise) from that of $\gamma_0$. It is always possible to perturb
$L_0$ or $L_1$ by a Hamiltonian isotopy supported near $f^{-1}(c_0)$
in order to ensure that the pair lies in positive position. 

If $(L_0,L_1)$ are in positive position, then the morphism space
$\hom(L_0,L_1)$ in the Fukaya-Seidel category $\F((\C^*)^n,f)$ is the
portion of the (wrapped) Floer complex spanned by generators that lie away from the
boundary fiber $f^{-1}(c_0)$. Similarly, given a collection of admissible
Lagrangians $L_0,\dots,L_k$ such that $L_0<L_1<\dots<L_k$, the
$A_\infty$-products are defined as usual by counts of (perturbed)
holomorphic discs (only involving generators outside of $f^{-1}(c_0)$). 

Starting with this partial definition for objects in positive position,
there are two ways to define morphism spaces and $A_\infty$-operations 
for arbitrary objects. One option is to define $\F((\C^*)^n,f)$ by
localization with respect to a suitably defined class of morphisms 
from any admissible Lagrangian to its admissible pushoffs in the positive
direction. The other option is to strengthen
the admissibility condition to fix the tangent direction to the
arc $\gamma$ at $c_0$ (so in fact all objects are required to 
approach $f^{-1}(c_0)$ from the same direction), and
perturb Floer's equation by an auxiliary Hamiltonian $h$ that rotates a neighborhood of
$f^{-1}(c_0)$ in the negative direction so that $\phi_h^1(L_0)<L_1$ for
every pair of objects. 

The Fukaya category $\F((\C^*)^n,f)$ is invariant under deformations
of the domain $\Omega$ and does not depend on the choice of the reference
point $c_0\in \partial \Omega$, as long as no critical value or other
``special fiber'' of $f$ crosses into $\Omega$ or out of it during the
deformation. This justifies omitting these choices from the notation.
(However, we note that the equivalence induced by an isotopic deformation
of $(\Omega,c_0)$ to some other choice $(\Omega',c'_0)$ does depend on the
choice of isotopy.)

Next, recall that in our case $f$ is a Laurent polynomial of the form
\eqref{eq:f}, near the tropical limit.  The image of $H$ (or more generally
$f^{-1}(c_0)$ for fixed $c_0\neq 1$ independent of~$\tau$) under the
logarithm map $$\mathrm{Log}:(x_1,\dots,x_n)\mapsto \frac{1}{|\log \tau|}
(\log |x_1|,\dots,\log |x_n|)$$ converges as $\tau\to 0$ to
the tropical hypersurface $\Gamma\subset \R^n$ defined
by the tropicalization $\varphi$, i.e.\ the set of points where the maximum
in \eqref{eq:tropf} is not unique.  The components of
$\R^n\setminus \Gamma$ correspond to the regions where the different terms
in \eqref{eq:tropf} achieve the maximum. For $\alpha\in A$, we
denote by $\Delta_\alpha$ the component of
$\R^n\setminus \Gamma$ on which $\alpha$
achieves the maximum in \eqref{eq:tropf}, and focus our attention on
the component $\Delta_0$ corresponding to the constant term. Note that
$\mathcal{U}_0=\mathrm{Log}^{-1}(\Delta_0)\subset (\C^*)^n$ is the set of
points where the constant term dominates all the other monomials that appear
in $f$. Enlarging $\Delta_0$ slightly, let $\Delta_0^+\subset \R^n$ be the 
$\delta$-neighborhood of $\Delta_0$ for fixed $\delta\ll 1$. For $\tau$
small enough, the portion of the amoeba $\mathrm{Log}(H)$ (resp.\ 
$\mathrm{Log}(f^{-1}(c_0))$) 
that converges to $\partial \Delta_0\subset \Gamma$ is contained inside
$\Delta_0^+$, and it makes sense to consider admissible Lagrangians which are
entirely contained inside
$\mathcal{U}_0^+=\mathrm{Log}^{-1}(\Delta_0^+)\subset (\C^*)^n$.

\begin{definition} We denote by
$\F^\circ((\C^*)^n,f)$ the full subcategory of
$\F((\C^*)^n,f)$ whose objects are admissible Lagrangians supported inside
$\mathcal{U}_0^+$.
\end{definition}

\noindent
In fact, Abouzaid only considers Lagrangians which are
sections of the logarithm map over the appropriate component of $\R^n\setminus 
\mathrm{Log}(f^{-1}(c_0))$; these are expected to split-generate
$\F^\circ((\C^*)^n,f)$.

Recalling that Lefschetz thimbles of critical
points of $f$ are an important source of objects of the Fukaya-Seidel
category, restricting to $\F^\circ\subset \F((\C^*)^n,f)$ basically amounts
to discarding all the critical points of $f$ where the constant
terms is not one of the dominant monomials in \eqref{eq:f}. In most cases
these correspond to the critical values which tend to infinity as $\tau\to
0$, so it is quite often the case that $\F^\circ((\C^*)^n,f)$ can be
constructed directly by choosing $\Omega$ to be a suitable bounded domain -- 
for example, the unit disc centered at $1$ in the complex plane, or a slight
enlargement thereof, is a natural choice.

The introduction of the restricted Fukaya-Seidel category $\F^\circ((\C^*)^n,f)$ is
motivated by homological mirror symmetry. Returning to the setup in the
introduction, since the components $\Delta_\alpha$ of $\R^n\setminus \Gamma$ 
arise as facets of the moment polytope $\Delta_Y$ defined by
\eqref{eq:delta_Y}, they can be identified
with the moment polytopes for the irreducible toric divisors $Z_\alpha$ of
the toric variety $Y$. In particular, $\Delta_0$ is the moment polytope
for the distinguished divisor $Z_0$ considered in the introduction.

Assume that $Z_0$ is compact,
i.e.\ the component $\Delta_0$ of $\R^n\setminus \Gamma$ is 
bounded, which happens precisely when $0$ is an interior point of 
$\mathrm{Conv}(A)$. In this case, ignoring slight differences in setup,
Abouzaid's thesis \cite{AbToric2} essentially shows that
\begin{equation}\label{eq:abouzaid}
\F^\circ((\C^*)^n,f) \simeq D^b\Coh(Z_0).
\end{equation}
In fact, Abouzaid's strategy of proof can be extended to the case where
$\Delta_0$ is unbounded, and one expects that \eqref{eq:abouzaid} continues
to hold in full generality. (This is by no means non-trivial, but it should
follow in a fairly straightforward way from the construction of the wrapped
category by localization.)

In general the category $\F^\circ((\C^*)^n,f)$ is a strict subcategory of
$\F((\C^*)^n,f)$. For example, let
$$f(x_1,x_2)=1+x_1^{-1}+x_2^{-1}+\tau x_2+\tau^{k+1}x_1x_2^k\quad
\text{for $k\ge 3$},$$
in which case $Z_0$ is the non-Fano Hirzebruch surface $\mathbb{F}_k=
\PP(\O_{\PP^1}\oplus \O_{\PP^1}(-k))$. By \hbox{\cite[\S 5]{AKO}}, in this
case $f$ has $k-2$ critical points outside of $\mathcal{U}_0^+$,
and $\F((\C^*)^n,f)$ is strictly larger than $D^b\Coh(\mathbb{F}_k)$,
whereas Abouzaid's result holds for $\F^\circ((\C^*)^n,f)$.

When assumption \eqref{eq:linebundlecase} holds, we expect the two
categories to coincide:

\begin{lemma}\label{l:convergeto1}
Assume that $0\in A$ is a vertex of every maximal cell of the polyhedral
decomposition $\mathcal{P}$.
Then, for $\tau$ sufficiently small, all the critical points of $f$ lie in
$\mathcal{U}_0$, and the critical values of $f$ converge to $1$.
\end{lemma}

\begin{proof}
Let $x$ be a critical point of $f$. Denote by $B\subseteq A$ the set of
leading order terms of $f$ near $x$, i.e.\ those $\alpha$ which come close
to achieving the maximum in \eqref{eq:tropf} at $\xi=\mathrm{Log}(x)$.
The elements of $B$ are the vertices of some cell of the
polyhedral decomposition $\mathcal{P}$, and if $\tau$ is sufficiently small
the other terms in $f$ are much smaller than those indexed by $\alpha\in B$.
Since the cells of $\mathcal{P}$ are simplices (by the maximal degeneration
assumption), the assumption that every maximal cell contains $0$ as a vertex
implies that the non-zero elements of $B$ are linearly independent.
Using logarithmic derivatives, the critical points are the solutions of
$$\sum_{\alpha \in A} \left(c_\alpha
\tau^{\rho(\alpha)}x^\alpha\right)\alpha = 0.$$
Assume that $B\neq \{0\}$. Then the leading order terms in this equation at
$x$ correspond to $\alpha \in B\setminus \{0\}$. However these terms are
linearly independent in $\R^n$ and hence cannot cancel out. This leads to a
contradiction if $\tau$ is sufficiently small.
Thus $B=\{0\}$, i.e.\ $x$ lies in the region where the constant term
dominates.
\end{proof}

\noindent
An alternative argument based on tropical geometry is that, when
\eqref{eq:linebundlecase} holds, the tropicalization of $f-c$ is
``tropically smooth'' and combinatorially similar to that of $f$ whenever
$c$ is sufficiently different from $1$; this implies that for small
enough $\tau$ the fibration $f$
is locally trivial outside of a small disc centered at $1$.
\medskip

We finish this discussion by recalling two important functors relating
$\F^\circ((\C^*)^n,f)$ to other Fukaya categories. The first one is {\em restriction to the
reference fiber},
$$\cap:\F^\circ((\C^*)^n,f)\to \W(f^{-1}(c_0)).$$
Given an admissible Lagrangian $L$ in $(\C^*)^n$, we define
$\cap(L)=\partial L\subset f^{-1}(c_0)$. On
morphisms, $\cap$ is defined by counting (perturbed) holomorphic discs with boundary on given
admissible Lagrangians (isotoped to lie in positive position), with inputs mapped to 
given generators in the
interior (away from $f^{-1}(c_0)$) and output a generator which lies
on the boundary (in $f^{-1}(c_0)$).

A folklore statement 
(which has so far only been verified in specific cases but should in this
setting be well within reach) is as
follows (see also \cite[\S 5]{Au09} for related considerations).

\begin{conj}\label{conj:hmsfiber}
Let $D_0$ be the union of all the irreducible toric divisors of $Z_0$,
and denote by $i_{D_0}:D_0\hookrightarrow Z_0$ the inclusion. Assume
that \eqref{eq:linebundlecase} holds. Then $f^{-1}(c_0)$ is mirror to $D_0$,
and the functor $\cap:\F^\circ((\C^*)^n,f)\to \W(f^{-1}(c_0))$ 
corresponds under mirror
symmetry to the restriction functor $i_{D_0}^*:D^b\Coh(Z_0)\to
D^b\Coh(D_0)$.
\end{conj}

\begin{remark}\label{rmk:periodlinebundlecase}
When \eqref{eq:linebundlecase} holds, $Y$ is isomorphic to
the total space of the canonical bundle of $Z_0$, and Orlov's results
\cite{OrlovKnorrer,OrlovGraded} give an equivalence $D^b_\sg(Z)\simeq D^b\Coh(D_0)$.
Thus, the statement that $f^{-1}(c_0)$ is mirror to $D_0$ is consistent with
Conjecture \ref{conj:hms}.
\end{remark}

The {\em acceleration} functor $\alpha:\F((\C^*)^n,f)\to \W((\C^*)^n)$,
meanwhile, amounts to completing admissible Lagrangians
to properly embedded Lagrangians in $(\C^*)^n$ via parallel transport over
an arc $\eta$ that connects $c_0$ to infinity in the complement of $\Omega$.

Constructing $\alpha$ in our setting is less straightforward than in some
other approaches to the Fukaya-Seidel category \cite{SeLef,AS2,Sylvan}.
One option is to set up the
completion of admissible Lagrangians in such a way that the
generators which lie inside $f^{-1}(\Omega)$ form a subcomplex of the
wrapped Floer complex.
Namely, choosing the arc $\eta$ suitably and/or modifying the Liouville
structure, we can assume that $f^{-1}(\eta)$ is preserved by the Liouville
flow. Given an admissible Lagrangian $L\subset f^{-1}(\Omega)$ with 
boundary in $f^{-1}(c_0)$, we get a properly embedded Lagrangian $\hat{L}$
by attaching to $L$ the cylinder obtained by parallel transport of $\partial
L\subset f^{-1}(c_0)$ 
over the arc $\eta$ (and rounding the corners at $c_0$ if the projections do not match
smoothly). Furthermore, we perturb the wrapping Hamiltonian by a term that
pushes $f^{-1}(c_0)$ slightly in the positive direction along the boundary of
$f^{-1}(\Omega)$. This has the effect of getting rid of the intersections in
$f^{-1}(c_0)$, whose existence would prevent the interior intersections from
forming a subcomplex. With this understood, given admissible Lagrangians
$L_0<\dots<L_k$ and their completions $\hat{L}_0,\dots,\hat{L}_k$,
the wrapped Floer complexes $CW^*(\hat{L}_i,
\hat{L}_j)$ contain two types of generators: those which lie in
$f^{-1}(\Omega)$, and those which lie over the arc $\eta$. Setting up the
Liouville structure carefully over $f^{-1}(\eta)$, one can ensure (using
e.g.\ a maximum principle and the local structure near Reeb chords) that 
the output of a perturbed J-holomorphic disc with inputs in $f^{-1}(\Omega)$
also lies in $f^{-1}(\Omega)$. Thus, $CW^*(\hat{L}_i,\hat{L}_j)$ contains
a subcomplex quasi-isomorphic to $\hom_{\F((\C^*)^n,f)}(L_i,L_j)$, and the
inclusion of these subcomplexes is part of an $A_\infty$-functor.

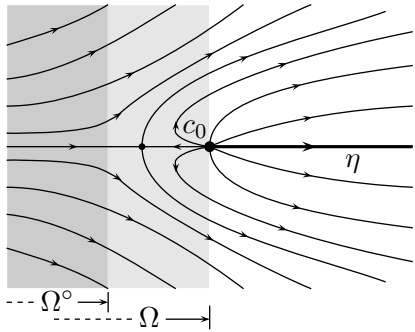
\begin{figure}[b]
\setlength{\unitlength}{9mm}
\begin{picture}(6,4.3)(-3,-2.3)
\psset{unit=\unitlength,linewidth=0.5pt}
\newgray{gray90}{0.9}
\newgray{gray85}{0.8}
\psframe[fillcolor=gray90,fillstyle=solid,linestyle=none](-3,-2.1)(0,2.1)
\psframe[fillcolor=gray85,fillstyle=solid,linestyle=none](-3,-2.1)(-1.5,2.1)
\psline[linestyle=dashed,dash=2pt 2pt](-3,-2.3)(-2.6,-2.3)
\put(-2.5,-2.42){\small $\Omega^\circ$}
\psline{->}(-1.95,-2.3)(-1.5,-2.3)
\psline[linestyle=dashed,dash=2pt 2pt](-2.3,-2.55)(-1.15,-2.55)
\put(-1.05,-2.67){\small $\Omega$}
\psline{->}(-0.65,-2.55)(0,-2.55)
\psline(-1.5,-2.45)(-1.5,-2.15)
\psline(0,-2.7)(0,-2.35)
\pscircle*(0,0){0.08}
\pscircle*(-1,0){0.05}
\psline[ArrowInside=->,linewidth=1.2pt](0,0)(3,0)
\psline[ArrowInside=->](0,0)(-1,0)
\psline[ArrowInside=->](-3,0)(-1,0)
\psbezier[ArrowInside=->](-1,0)(-1,1.1)(1.5,1.6)(2.6,2)
\psbezier[ArrowInside=->](0,0)(1,0.5)(3,0.6)
\pscurve(0,0)(-0.5,0.2)(-0.5,0.3)
\psbezier[ArrowInside=->](-0.5,0.3)(-0.5,0.6)(0.5,1.2)(3,1.7)
\psline{->}(-0.5,0.28)(-0.5,0.38)
\psbezier[ArrowInside=->](0,0)(0,0.6)(1,0.95)(3,1.2)
\pscurve{->}(-3,0.2)(-1.7,0.3)(-1.3,0.5)
\psbezier[ArrowInside=->](-1.35,0.47)(-0.8,1)(1,1.9)(1.5,2.1)
\psbezier[ArrowInside=->](-3,0.6)(-1.5,0.6)(0.1,1.8)(0.5,2.1)
\psbezier[ArrowInside=->](-3,1)(-2,1.1)(-1,1.8)(-0.5,2.1)
\psbezier[ArrowInside=->](-3,1.5)(-2.5,1.65)(-2,1.8)(-1.5,2.1)
\psbezier[ArrowInside=->](-1,0)(-1,-1.1)(1.5,-1.6)(2.6,-2)
\psbezier[ArrowInside=->](0,0)(1,-0.5)(3,-0.6)
\pscurve(0,0)(-0.5,-0.2)(-0.5,-0.3)
\psbezier[ArrowInside=->](-0.5,-0.3)(-0.5,-0.6)(0.5,-1.2)(3,-1.7)
\psline{->}(-0.5,-0.28)(-0.5,-0.38)
\psbezier[ArrowInside=->](0,0)(0,-0.6)(1,-0.95)(3,-1.2)
\pscurve{->}(-3,-0.2)(-1.7,-0.3)(-1.3,-0.5)
\psbezier[ArrowInside=->](-1.35,-0.47)(-0.8,-1)(1,-1.9)(1.5,-2.1)
\psbezier[ArrowInside=->](-3,-0.6)(-1.5,-0.6)(0.1,-1.8)(0.5,-2.1)
\psbezier[ArrowInside=->](-3,-1)(-2,-1.1)(-1,-1.8)(-0.5,-2.1)
\psbezier[ArrowInside=->](-3,-1.5)(-2.5,-1.65)(-2,-1.8)(-1.5,-2.1)
\put(2,-0.3){\small $\eta$}
\put(-0.4,0.2){\small $c_0$}
\end{picture}
\caption{Acceleration as restriction to a Liouville subdomain.} \label{fig:accelflow}
\end{figure}

An alternative and perhaps more elegant construction of the acceleration
functor is to set things up so that $f^{-1}(\Omega)$ contains a Liouville
subdomain whose completion is Liouville deformation equivalent to the total
space (here, $(\C^*)^n$). For example, one can arrange for the Liouville
structure in a neighborhood of $f^{-1}(\eta)$ to be a product one, where
in the base of the fibration $f$ the Liouville flow is as depicted in 
Figure~\ref{fig:accelflow}. Then $f^{-1}(\Omega^\circ)$ is a Liouville 
subdomain (since the Liouville flow is everywhere transverse to its
boundary); the total space of the fibration contains additional cancelling
pairs of handles which are not present in the completion of
$f^{-1}(\Omega^\circ)$, but the two are nonetheless deformation equivalent
as Liouville manifolds. Requiring admissible Lagrangians to approach
$f^{-1}(c_0)$ along the horizontal axis of Figure \ref{fig:accelflow}, 
their restrictions to $f^{-1}(\Omega^\circ)$ are properly embedded and
define objects of the wrapped Fukaya category.  In this context, $\alpha$ is
simply Abouzaid and Seidel's restriction functor \cite{AS} to
$\W(f^{-1}(\Omega^\circ))\simeq \W((\C^*)^n)$.

\section{Acceleration, restriction, and lifting}\label{s:conjFS}

\subsection{The acceleration functors $\alpha_0$ and $\alpha_\infty$}\label{s:FSaccel}

In this section we define two acceleration functors $\alpha_0$ and
$\alpha_\infty$ from $\F^\circ((\C^*)^n,f)$ to $\W((\C^*)^n\setminus H)$.

First we observe that the critical values of $f$ relevant to the category
$\F^\circ((\C^*)^n,f)$ converge to 1 as $\tau\to 0$, by the proof of Lemma
\ref{l:convergeto1}. (Indeed, going over the proof, since $\F^\circ$ only considers the region where the
constant term of $f$ is among the largest monomials, it is a given that
$0\in B$ and the assumption of the lemma is not necessary in order to conclude that
$B=\{0\}$.) It follows that the domain $\Omega$ to which admissible
Lagrangians are required to project can be chosen to be a
neighborhood of 1; our preferred choice is the unit disc centered at 1, with
$c_0=0\in \partial \Omega$, or a slightly smaller disc, with $c_0$ on the
positive real axis near the origin.

The simplest way to construct the acceleration functor $\alpha_0$ is to view
admissible
Lagrangian submanifolds of $(\C^*)^n$ with boundary in $f^{-1}(0)=H$ as 
properly embedded Lagrangian submanifolds of $(\C^*)^n\setminus H$. This
determines $\alpha_0$ on objects. On morphisms, $\alpha_0$ is defined by
Floer-theoretic continuation
maps from the Hamiltonian perturbation used to construct $\F^\circ((\C^*)^n,f)$
to the quadratic Hamiltonian for $\W((\C^*)^n\setminus H)$; these are well-defined
because, even after correcting the former to account for the change in
symplectic structure, the latter Hamiltonian has a faster growth rate near $H$.

More precisely, say that we construct $\F^\circ((\C^*)^n,f)$ by modifying the 
Liouville structure of $(\C^*)^n$ to a product one near the hypersurface
$H$, considering only admissible Lagrangians whose projection under 
$f$ approaches the origin from a fixed direction (e.g.\ the real positive
axis), and using an auxiliary Hamiltonian perturbation that
rotates a neighborhood of the origin in the clockwise direction in order
to ensure positive position. 
As seen in \S \ref{s:hmscompl}, removing $H$ from $(\C^*)^n$ entails
a change in the Liouville structure. With respect to the new symplectic
structure, the ``positive position'' perturbation is achieved by
a Hamiltonian which grows linearly, with a small positive
slope, with respect 
to the radial coordinate of the cylindrical end near $H$. 
The Hamiltonians used to define morphisms and compositions in
$\W((\C^*)^n\setminus H)$ have a faster growth rate along $H$, and
hence there are well-defined continuation maps.

Another way to construct $\alpha_0$, which fits into the general framework
of acceleration functors discussed at the end of the previous section, is to
set up $\F^\circ((\C^*)^n,f)$ using a domain $\Omega$ which stays away
from the origin, say a disc of radius $1-\epsilon$ centered at~$1$,
and a reference point located near (but not at) the origin, say $c_0=\epsilon \in
\partial \Omega$. Since the origin lies outside of $\Omega$, we can just as
well remove the fiber over zero and work in $(\C^*)^n\setminus H$ with the
Liouville structure constructed in \S \ref{s:hmscompl} (suitably modified 
near $f^{-1}(\epsilon)$ for the needs of the construction of
$\F^\circ((\C^*)^n,f)$). Viewing the restriction of $f$ to
$(\C^*)^n\setminus H$ as a fibration over $\C^*$ (instead of $\C$),
we choose an arc $\eta_0$ connecting $c_0=\epsilon$ to the origin 
(instead of infinity); the canonical choice is the interval $(0,\epsilon]$
in the real axis. We then construct $\alpha_0$ as in \S \ref{s:FS}, either
by extending admissible Lagrangians with boundary in $f^{-1}(\epsilon)$ by
parallel transport in the fibers of $f$ over the interval $(0,\epsilon]$,
or by restriction to a Liouville subdomain (disjoint from 
$f^{-1}((0,\epsilon])$) whose completion is deformation
equivalent to $(\C^*)^n\setminus H$.

To define the other acceleration functor $\alpha_\infty$, 
we construct $\F^\circ((\C^*)^n,f)$ by setting
$\Omega$ to be the disc of radius $1-\epsilon$ centered at 1, 
and observe again that, since the origin lies outside of $\Omega$,
we can just as well work with the restriction $f:(\C^*)^n\setminus H\to \C^*$.
Choose an arc $\eta_\infty$ that connects $c_0\in \partial\Omega$ to infinity in the complement of
$\Omega$ (avoiding the origin and any critical values of $f$ that may lie
outside of $\Omega$).  When there are no critical values outside
of $\Omega$ (e.g.\ when \eqref{eq:linebundlecase} holds), the most natural
choice is to take $c_0=2-\epsilon$ and $\eta_\infty$ the interval
$[2-\epsilon,\infty)$ in the real axis. In the general case, when
$f$ has additional critical values near infinity, the functor
$\alpha_\infty$ genuinely depends on the choice of the arc $\eta_\infty$,
as we shall see on an explicit example in \S \ref{s:Pn}.

In any case, the construction described at
the end of \S \ref{s:FS} then provides an acceleration functor
$\alpha_\infty:\F^\circ((\C^*)^n,f)\to \W((\C^*)^n\setminus H)$.
By construction, $\rho\circ \alpha_\infty=0$, since the objects in the
image of $\alpha_\infty$ remain away from a neighborhood of $H$.

For the purpose of comparing $\alpha_0$ and $\alpha_\infty$ as we will do in
\S \ref{s:FSj} below, it is
useful to have both functors defined on the same model of the category
$\F^\circ((\C^*)^n,f)$, i.e.\ choose $c_0=\epsilon$ rather than
$c_0=2-\epsilon$. Whenever necessary, we
identify the categories corresponding to the choices $c_0=\epsilon$ and
$c_0=2-\epsilon$ via the isotopy that moves the reference point along the
lower half of the boundary of the disc $\Omega$.

The functors $\alpha_0$ and $\alpha_\infty$ are very similar to each
other at first glance, as should be apparent by 
viewing $(\C^*)^n\setminus H$ as the total space of a fibration over $\C^*$
(the restriction of $f$).
One can then consider admissible Lagrangians whose projections under $f$
go towards either end of the cylinder, and the corresponding acceleration
functors; see Figure \ref{fig:functors}. 
Thus, the constructions of $\alpha_0$ and $\alpha_\infty$ extend in a straightforward manner to more general
symplectic fibrations over the cylinder.
However, in our case an important feature that breaks the symmetry 
between the two ends of Figure \ref{fig:functors} is that the monodromy
around $0$ is trivial, which is a crucial feature needed to
define the restriction functor $\rho$, whereas the monodromy around $\infty$
is not.

\begin{figure}[b]
\setlength{\unitlength}{1.2cm}
\begin{picture}(4.5,2)(-1.5,-1)
\psset{unit=\unitlength}
\pscircle[linewidth=0.5pt,linestyle=dotted](0.9,0){0.9}
\pscircle[fillstyle=solid,fillcolor=white](-0.1,0){0.08}
\psline(0.85,0.15)(0.95,0.25)
\psline(0.85,0.25)(0.95,0.15)
\psline(0.9,-0.15)(1,-0.25)
\psline(0.9,-0.25)(1,-0.15)
\psline(1,0.05)(1.1,-0.05)
\psline(1,-0.05)(1.1,0.05)
\psline(1.05,0)(-0.02,0)
\pscurve(0.9,0.2)(0.7,0.03)(0.4,0)
\pscurve(0.95,-0.2)(0.7,-0.03)(0.4,0)
\pscurve(0.6,0)(0.5,-0.03)(0.5,-0.4)(1,-0.6)(1.45,-0.4)(1.65,-0.1)(1.8,-0.03)(2.5,0)(3,0)
\psline(-0.18,0)(-1.6,0)
\put(-0.3,0.15){\small 0}
\put(0.7,0.4){\tiny crit\,$f$}
\put(0.1,0.15){\small $\alpha_0$}
\put(-1,0.18){\small $j$}
\put(2.2,0.15){\small $\alpha_\infty$}
\put(1.75,0.7){\small $\Omega$}
\end{picture}
\qquad\qquad
\setlength{\unitlength}{2cm}
\begin{picture}(3.3,1.4)(-0.6,-0.7)
\psset{unit=\unitlength}
\psline(-0.5,-0.7)(2.5,-0.7)
\psline(-0.5,0.7)(2.5,0.7)
\psellipticarc(2.5,0)(0.15,0.7){-90}{90}
\psellipticarc[linestyle=dashed](2.5,0)(0.15,0.7){90}{270}
\psellipse(-0.5,0)(0.15,0.7)
\pscircle*(0.7,0.2){0.03}
\pscircle*(0.9,0.2){0.03}
\pscircle*(1.1,0.2){0.03}
\psline(0.7,0.2)(-0.35,0.2)
\pscurve(0.9,0.2)(0.8,0.1)(0.6,0.08)
\pscurve(1.1,0.2)(0.9,0)(0.6,-0.04)
\psline(0.6,0.08)(-0.35,0.08)
\psline(0.6,-0.04)(-0.35,-0.04)
\psline(1.1,0.2)(2.65,0.2)
\pscurve(0.9,0.2)(1,0.1)(1.2,0.08)
\pscurve(0.7,0.2)(0.9,0)(1.2,-0.04)
\psline(1.2,0.08)(2.65,0.08)
\psline(1.2,-0.04)(2.65,-0.04)
\put(0.7,0.35){\small crit\,$f$}
\put(-0.85,0){\small 0}
\put(2.75,0){\small $\infty$}
\put(0,0.3){\small $\alpha_0$}
\put(1.5,0.3){\small $\alpha_\infty$}
\psline(-0.4,-0.5)(2.6,-0.5)
\put(0.8,-0.4){\small $j$}
\end{picture}
\caption{The fibration $f:(\C^*)^n\setminus H\to \C^*$ and the functors
$\alpha_0,\alpha_\infty,j.$}\label{fig:functors}
\end{figure}
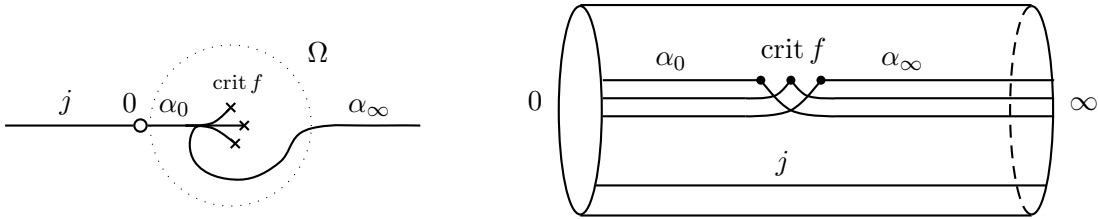

As mentioned in the introduction, it should follow from the construction
of $\alpha_0$ that its composition with the restriction functor
$\rho:\W((\C^*)^n\setminus H)\to \W(H)$ introduced in~\S \ref{s:hmscompl}
coincides with the ``restriction to the fiber'' functor described in
\S \ref{s:FS}: 
\begin{equation}\label{eq:rhoalpha0}
\rho\circ \alpha_0=\cap:\F^\circ((\C^*)^n,f)\to \W(H).
\end{equation}
While the statement is clear at the level of objects, the proof requires some work due to
the differences between the two constructions. A possible approach is to
consider Floer theory for admissible Lagrangians perturbed by suitably
chosen Hamiltonians on $(\C^*)^n\setminus H$ with linear growth near $H$,
whose flow extends over $H$ and wraps around it by a finite number of
turns $t$. In this setting, one can count rigid solutions to Floer's
equation with inputs away from $H$ and outputs in $H$ as in the
definition of $\cap$. For $0<t<1$ this gives $\cap$, while every time $t$
passes through an integer there is a bifurcation and the map changes by
composition with a degree 2 element of the (truncated) wrapped Floer
cohomology; observing that this element is yet another instance of the
natural transformation  $\Theta$ discussed in Remark \ref{rmk:nattrans},
the statement should then follow by taking the 
limit as $t\to \infty$.

The interpretation of \eqref{eq:rhoalpha0} under homological mirror symmetry
is as follows. By Conjectures \ref{conj:hmscompl} and \ref{conj:FS},  we
expect a commutative diagram
\begin{equation}\label{eq:diagrhoalpha0}
\begin{CD}
\F^\circ((\C^*)^n,f) @>\alpha_0>> \W((\C^*)^n\setminus H) @>\rho>> \W(H)\\
@V{\simeq}VV @V{\simeq}VV @VV{\simeq}V\\
D^b\Coh(Z_0) @>i_*>> D^b\Coh(Z) @>q>> D^b_\sg(Z)
\end{CD}
\end{equation}
where the vertical equivalences are instances of homological mirror
symmetry, and in the bottom row $i_*$ is the pushforward by the inclusion
map $i:Z_0\hookrightarrow Z$ and $q$ is the quotient by $\Perf(Z)$.
When \eqref{eq:linebundlecase} holds, $Y$ is the total space of the canonical
bundle of $Z_0$, and Orlov's work \cite{OrlovKnorrer,OrlovGraded} gives an equivalence $\varepsilon:
D^b_\sg(Z)\stackrel{\simeq}{\longrightarrow} D^b\Coh(D_0)$ where
$D_0$ is the union of the irreducible toric divisors of $Z_0$.
 It is not hard to check that the
composition $\varepsilon\circ q\circ i_*:D^b\Coh(Z_0)\to D^b\Coh(D_0)$ coincides with
pullback by the inclusion $i_{D_0}:D_0\hookrightarrow Z_0$.
Using \eqref{eq:rhoalpha0}, the diagram \eqref{eq:diagrhoalpha0} then reduces to
$$\begin{CD}
\F^\circ((\C^*)^n,f) @>\rho\alpha_0=\cap\ >> \W(H)\\
@V{\simeq}VV @VV{\simeq}V\\
D^b\Coh(Z_0) @>\varepsilon q i_*=i_{D_0}^*\ >> D^b\Coh(D_0)
\end{CD}
$$
which is precisely the content of Conjecture \ref{conj:hmsfiber}.

On the other hand, the identity $\rho\circ \alpha_\infty=0$ expresses the
property that the counterpart of $\alpha_\infty:\F^\circ((\C^*)^n,f)\to
\W((\C^*)^n\setminus H)$ under mirror symmetry is a functor from
$D^b\Coh(Z_0)$ to $D^b\Coh(Z)$ whose image is annihilated by the
quotient $q:D^b\Coh(Z)\to D^b_\sg(Z)$, i.e.\ it is contained in the
subcategory $\Perf(Z)$. 

The simplest instance of such a functor from $D^b\Coh(Z_0)$ to
$\Perf(Z)$ is the pullback by some projection map $\pi:Z\to Z_0$, when
one exists. For example, when
\eqref{eq:linebundlecase} holds, $Y$ is the total space of the canonical
bundle over $Z_0$ and we can take $\pi$ to be the restriction to $Z$ of the
projection $Y\to Z_0$.  Conjecture \ref{conj:FS} postulates that in this
case $\alpha_\infty$ as constructed above does indeed correspond to the
pullback $\pi^*$.

\subsection{The lifting functor $j$}\label{s:FSj}

We now construct a functor 
$j:\W(H)\to \W((\C^*)^n\setminus H)$ as follows. Choose a properly
embedded arc $\gamma$ in $\C^*$ that connects 0 to infinity and avoids the
critical values of $f$. Deforming the Liouville structure if necessary, we arrange for
the Liouville flow to be tangent
to $f^{-1}(\gamma)$ and pointing away from some interior fiber $f^{-1}(c_0)$,
$c_0\in\gamma$, which is also preserved by the wrapping Hamiltonian.

With this understood, given an object $L$ of $\W(H)$, we use parallel transport in the fibers of
$f$ over the arc $\gamma$  to obtain a
properly embedded Lagrangian submanifold $j(L)$ in \hbox{$(\C^*)^n\setminus H$}.
Moreover, we can ensure that, for any pair of objects $L_1,L_2\in \W(H)$, the generators of the
wrapped Floer complex $CW^*(j(L_1),j(L_2))$ which lie in the fiber
$f^{-1}(c_0)$ form a subcomplex isomorphic to $CW^*(L_1,L_2)$.
We define $j$ on morphisms via these inclusions of wrapped Floer complexes.

The functor $j$ is not canonical, as it depends on the choice of the
arc $\gamma$; the set of choices is essentially the same as for 
$\alpha_\infty$. When \eqref{eq:linebundlecase} holds there is a preferred
choice, namely we can take $\gamma$ to be the negative real axis
$(-\infty,0)$; see Figure \ref{fig:functors}. In the general case, 
we choose the arc $\gamma$ to be homotopic to the concatenation of the
arcs $\eta_0$ and $\eta_\infty$ used to define $\alpha_0$ and
$\alpha_\infty$ via a homotopy that does not cross any critical value of
$f$.

By construction (and for a suitable choice of grading conventions), 
$\rho j\simeq \mathrm{id}$, i.e.\ $j$ is a 
right (quasi)inverse to the restriction functor. 
Moreover, we expect to have an exact triangle
\begin{equation}\label{eq:jtriangle}
j\rho\alpha_0[-1] \to \alpha_\infty \to \alpha_0 \to
j \rho \alpha_0.
\end{equation}
Indeed, given any admissible Lagrangian $L\in \F^\circ((\C^*)^n,f)$, with boundary $\partial
L=\rho\alpha_0(L)\in \W(H)$, near $0$ (resp.\ $\infty$)  the ends of 
$\alpha_0(L)$ (resp.\ $\alpha_\infty(L)$) and $j(\partial L)$ are modelled
on the products of $\partial L$ with the positive and negative real axes.
With our grading conventions,
the family of Reeb chords that wrap halfway around the cylindrical end
gives rise to an element $\nu^0_L\in CW^0(\alpha_0(L),j(\partial L))$, 
resp.\ $\nu^\infty_L\in CW^1(j(\partial L),\alpha_\infty(L))$.
Meanwhile, the continuation map for the
isotopy $\psi$ of $f^{-1}(\Omega)$ induced by moving the reference fiber for
$\F^\circ((\C^*)^n,f)$ from $\epsilon$ to $2-\epsilon$ counterclockwise
along the lower half of $\partial\Omega$ determines an element of
$CF^0(\psi(L),L)$, which upon acceleration yields an element 
$\mu_L \in CW^0(\alpha_\infty(L), \alpha_0(L))$. It is then not hard to
check that
\begin{equation}\label{eq:jtriangleL}
j(\partial L)[-1]\stackrel{\nu^\infty_L}{\longrightarrow} \alpha_\infty(L)
\stackrel{\mu_L}{\longrightarrow} \alpha_0(L)
\stackrel{\nu^0_L}{\longrightarrow} j(\partial L)
\end{equation}
is an exact triangle in $\W((\C^*)^n\setminus H)$. This can be
viewed as an instance of the surgery exact triangle, observing that 
$\alpha_\infty(L)$ is Hamiltonian isotopic to the nontrivial component of the
Lagrangian obtained by wrapping $\alpha_0(L)$ halfway around the puncture
at zero so that it intersects $j(\partial L)$ cleanly along a copy of
$\partial L$, and performing Lagrangian surgery along these intersections
(see Figure \ref{fig:functors}).

We expect that the morphisms $\nu^\infty_L$, $\mu_L$ and $\nu^0_L$ are part of 
natural transformations between the functors $j\rho\alpha_0$,
$\alpha_\infty$ and $\alpha_0$, and the exact triangles
\eqref{eq:jtriangleL} assemble into the exact triangle \eqref{eq:jtriangle}.

On the mirror side, let $Z_{\neq 0}=\bigcup_{\alpha\neq 0} Z_\alpha$, 
and observe that we have a short exact sequence of sheaves
\begin{equation}\label{eq:ses_sheaves_Z0}
0\to \O_{Z_{\neq 0}}(-Z_0)\to \O_Z \to \O_{Z_0}\to 0,
\end{equation}
coming from the decomposition 
$Z=Z_{\neq 0} \cup Z_0$. We note that $Z_{\neq 0}\cap Z_0=D_0$ is the
union of the irreducible toric divisors of $Z_0$.

Assume that \eqref{eq:linebundlecase} holds. Then, recalling that $Y$ is the
total space of the canonical line bundle over $Z_0$ and observing that
$Z_{\neq 0}$ is the restriction of this line bundle to $D_0\subset Z_0$,
\eqref{eq:ses_sheaves_Z0} gives rise to an exact triangle of functors
$$\kappa\to \pi^*\to i_*\to \kappa[1]$$
where $\kappa:D^b\Coh(Z_0)\to D^b\Coh(Z)$ is the composition of restriction
to the anticanonical divisor $D_0=Z_{\neq 0}\cap Z_0$, pullback under the projection $\pi_{|Z_{\neq 0}}:Z_{\neq 0}\to
D_0$, twisting by $\O(-Z_0)$, and pushforward by the inclusion of $Z_{\neq 0}$ 
into $Z$. 
In this setting, $D^b_\sg(Z)\simeq D^b\Coh(D_0)$,
and restriction from $Z_0$ to $D_0$ corresponds to
$\rho\alpha_0$. The functor from $D^b\Coh(D_0)$ to $D^b\Coh(Z)$ consisting of
the remaining steps (pullback to $Z_{\neq 0}$, twisting by $\O(-Z_0)$, 
and pushforward to $Z$) is thus our conjectural counterpart to $j$ under
mirror symmetry, up to a grading shift. (This construction is very closely
related to Orlov's proof of the equivalence between $D^b\Coh(D_0)$ and 
$D^b_\sg(Z)$ \cite{OrlovKnorrer,OrlovGraded}.)

\subsection{Framings and gradings}

As mentioned in Remark \ref{rmk:rescale}, the construction of
the functors $\alpha_0,\alpha_\infty,j$ can be carried out using
a different ``framing'' of $H$, i.e.\ considering the defining equation
$x^{-\alpha} f$ instead of $f$, for any $\alpha\in A$. On the mirror 
side this amounts to considering the component $Z_\alpha$ of $Z$ 
instead of $Z_0$; this suggests that, taken together, the Fukaya categories
$\F^\circ((\C^*)^n,x^{-\alpha}f)$ for varying choices of framings give a
substantial amount of insight into $\W((\C^*)^n\setminus H)$.

Changing the framing of $H$ does not
affect the $\Z$-grading on $\W((\C^*)^n\setminus H)$, since the
chosen trivialization of the tangent bundle by restriction from $(\C^*)^n$
did not involve the Laurent polynomial $f$. On the other hand, it does modify the preferred choice
of grading on $\W(H)$, as the preferred trivialization of $\det(TH)$ is
induced from that of $\det(T(\C^*)^n)$ by interior product with $df$.

The analogue of this under mirror symmetry is the observation that, even
though $D^b\Coh(Z)$ has a canonical $\Z$-grading, its quotient $D^b_\sg(Z)$
does not. However, the choice of a $\C^*$-action on $Y$ (for which the
superpotential $W$ has weight~2) determines a $\Z$-grading
on $D^b_\sg(Z)$ \cite{OrlovGraded}.
For each choice of $\alpha\in A$, the divisor $Z_\alpha$ (or equivalently
the corresponding ray in the fan of $Y$) determines a $\C^*$-action, and
hence a $\Z$-grading on $D^b_\sg(Z)$. It is part of our general conjectural
setup that these gradings on $D^b_\sg(Z)$ match up with those on $\W(H)$.

Even though the restriction functor $\rho$ is only $\Z/2$-graded,
it admits a $\Z$-graded enhancement if one only considers framed
Lagrangians in the sense of Remark \ref{rmk:framed}. For instance,
the composition $\rho\alpha_0:\F^\circ((\C^*)^n,f)\to \W(H)$, which only
involves framed Lagrangians, is compatible with the $\Z$-gradings.  Similarly, the quotient functor
$q:D^b\Coh(Z)\to D^b_\sg(Z)$ is only $\Z/2$-graded, but its composition with
the inclusion pushforward,
$qi_*:D^b\Coh(Z_0)\to D^b_\sg(Z)$, is compatible
with the $\Z$-grading on $D^b_\sg(Z)$ determined by the divisor $Z_0$.

\section{Example: Local $\PP^n$}\label{s:Pn}

In this section, we consider the example where $H\subset (\C^*)^n$ is
the hypersurface defined by the Laurent polynomial
$$f(x_1,\dots,x_n)=x_1+\dots+x_n+\frac{\tau}{x_1\dots x_n}+1.$$
The polyhedral decomposition $\mathcal{P}$ (which gives the fan for $Y$),
the tropicalization of $f$ (which gives the moment polytope for $Y$), and the
Lefschetz fibration $f:(\C^*)^n\to \C$ are depicted (for $n=2$) on Figure \ref{fig:P2}.

One easily checks that $Y$ is the total space of the canonical bundle
$\O_{\PP^n}(-(n+1))$ over $\PP^n$. The facet $\Delta_0$ 
is a standard simplex, and $Z_0\simeq \PP^n$ is the zero
section, while the other components of $Z$ correspond to the total space of
$\O(-(n+1))$ over the various coordinate hyperplanes of
$\PP^n$, whose union forms the anticanonical divisor 
$D_0=\{z_0\dots z_n=0\}\subset \PP^n$.

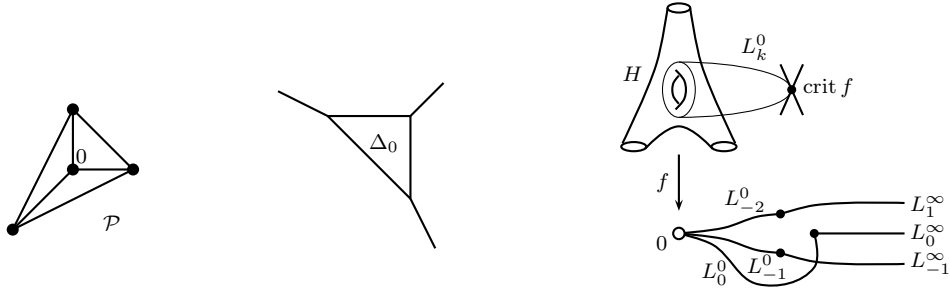
\begin{figure}[b]
\setlength{\unitlength}{8mm}
\begin{picture}(2,3)(-1,-2)
\psset{unit=\unitlength}
\pscircle*(-1,-1){0.1}
\pscircle*(0,0){0.1}
\pscircle*(1,0){0.1}
\pscircle*(0,1){0.1}
\psline(0,1)(0,0)(1,0)(0,1)(-1,-1)(1,0)
\psline(0,0)(-1,-1)
\put(0.05,0.1){\tiny 0}
\put(0.5,-1){\tiny $\mathcal{P}$}
\end{picture}
\qquad\qquad
\setlength{\unitlength}{1.1cm}
\begin{picture}(2,2.5)(-1.6,-2.1)
\psset{unit=\unitlength}
\psline(-1.6,0.3)(-1,0)(0,0)(0,-1)(-1,0)
\psline(0,0)(0.4,0.4)
\psline(0.3,-1.6)(0,-1)
\put(-0.5,-0.4){\tiny $\Delta_0$}
\end{picture}
\qquad\qquad
\setlength{\unitlength}{1.5cm}
\begin{picture}(3,2.5)(-0.8,-0.5)
\psset{unit=\unitlength}
\put(-0.2,-0.15){\tiny 0}
\pscircle*(1.2,0){0.04}
\pscircle*(0.9,0.173){0.04}
\pscircle*(0.9,-0.173){0.04}
\psline(1.2,0)(2,0)
\pscurve(0.9,-0.173)(1.2,-0.25)(2,-0.27)
\pscurve(0.9,0.173)(1.2,0.25)(2,0.27)
\put(2.05,-0.05){\tiny $L_0^\infty$}
\put(2.05,-0.3){\tiny $L_{-1}^\infty$}
\put(2.05,0.2){\tiny $L_1^\infty$}
\pscurve(1.2,0)(1.2,-0.3)(0.7,-0.45)(0.3,-0.1)(0,0)
\pscurve(0.9,-0.173)(0.7,-0.15)(0.4,-0.05)(0,0)
\pscurve(0.9,0.173)(0.7,0.15)(0.4,0.05)(0,0)
\pscircle[fillstyle=solid,fillcolor=white](0,0){0.06}
\put(0.2,-0.4){\tiny $L_0^0$}
\put(0.6,-0.35){\tiny $L_{-1}^0$}
\put(0.4,0.25){\tiny $L_{-2}^0$}
\pscurve(-0.1,2)(-0.2,1.5)(-0.5,0.8)
\pscurve(0.1,2)(0.2,1.5)(0.5,0.8)
\pscurve(-0.3,0.75)(0,0.95)(0.3,0.75)
\pscurve(-0.03,1.45)(0.05,1.35)(0.05,1.2)(-0.03,1.1)
\pscurve(0,1.4)(-0.05,1.33)(-0.05,1.22)(0,1.15)
\psellipse(0,2)(0.11,0.05)
\psellipse(-0.4,0.77)(0.12,0.05)
\psellipse(0.4,0.77)(0.12,0.05)
\psellipse[linewidth=0.4pt](0,1.275)(0.15,0.25)
\psellipticarc[linewidth=0.4pt](0,1.275)(1,0.25){-90}{90}
\psline(0.9,1.05)(1.1,1.5)
\psline(1.1,1.05)(0.9,1.5)
\pscircle*(1,1.275){0.04}
\put(1.1,1.25){\tiny crit\,$f$}
\put(0.55,1.6){\tiny $L_k^0$}
\put(-0.5,1.35){\tiny $H$}
\psline{->}(0,0.7)(0,0.2)
\put(-0.2,0.4){\tiny $f$}
\end{picture}
\caption{Example: $f(x_1,x_2)=x_1+x_2+\dfrac{\tau}{x_1x_2}+1$ and $Y=\O(-3)\to \PP^2$.}
\label{fig:P2}
\end{figure}

In \cite{AbToric}, Abouzaid constructs admissible Lagrangian submanifolds
$L_k$ of $(\C^*)^n$ which are sections of the logarithm map over
$\Delta_0\subset \R^n$, with boundary in
$$f^{-1}(2)=\{x_1+\dots+x_n+\tau x_1^{-1}\dots x_n^{-1}-1=0\}.$$
In the tropical limit $\tau\to 0$, $L_k$ is defined by
$\arg(x_j)=-2\pi k \log(|x_j|)$ for $j=1,\dots,n$ (where logarithms are taken in base
$\tau^{-1}$). Abouzaid shows the existence of an equivalence
$\F^\circ((\C^*)^n,f)\simeq D^b\Coh(\PP^n)$ under which $L_k$ corresponds to
$\O(k)$. 

An alternative description in terms of Lefschetz thimbles is as follows. 
The Laurent polynomial $f$ has $n+1$ critical points, located at
$x_1=\dots=x_n=\tau^{1/(n+1)} e^{2\pi i k/(n+1)}$, and the corresponding
critical values are $c_k=1+(n+1)\tau^{1/(n+1)} e^{2\pi i k/(n+1)}$. The
Lagrangian $L_k$ is then Hamiltonian isotopic to the
Lefschetz thimble associated to the arc $\gamma_k$ which runs from the
critical value
$c_k$ to the reference point 2 by first moving radially away from $1$, then clockwise by an angle
of $2\pi k/(n+1)$ to reach the real positive axis, then radially outwards again.
(The category $\F^\circ((\C^*)^n,f)$ is generated by the exceptional
collection $L_0,\dots,L_n$, but
one can just as well consider $L_k$ for all $k\in \Z$.)

The properly embedded Lagrangian submanifolds
$L_k^\infty=\alpha_\infty(L_k)$ can then be described as
Lefschetz thimbles for arcs $\gamma_k^\infty$ isotopic to the union of $\gamma_k$ with the 
interval $[2,+\infty)$ in the real axis; see Figure \ref{fig:P2} right.
In particular, $L_0^\infty$ is simply the real positive locus,
$L_0^\infty=(\R_+)^n\subset (\C^*)^n\setminus H$, which is consistent with
our expectation that it corresponds under mirror symmetry to the pullback
of $\O_{\PP^n}$ under the projection $\pi:Z\to Z_0=\PP^n$, i.e.\ $\O_Z$. 
More generally, for $|k|<(n+1)/2$ the arc
$\gamma_k^\infty$ is isotopic to a radial straight line from $c_k$ to infinity,
so we can take $L_k^\infty=(e^{2\pi i k/(n+1)}\,\R_+)^n$.

Meanwhile, the Lagrangian submanifolds $L_k^0=\alpha_0(L_k)$ are Lefschetz
thimbles for arcs $\gamma_k^0$ isotopic to the union of $\gamma_k$ with the
lower half of the unit circle centered at 1, i.e., running from
the critical value $c_k$ to the origin by moving radially away from 1 then
clockwise by an angle of $\pi + 2\pi k/(n+1)$, as shown in Figure
\ref{fig:P2}. In fact, $L_k^0$ can be described directly as a section of the
logarithm map over $\Delta_0$, with boundary in $f^{-1}(0)=H$, by modifying
Abouzaid's construction to account for the sign change. Namely, in the
tropical limit, $L_k^0$ is the Lagrangian section over $\Delta_0$  defined by
$$\arg(x_j)=-(2k+n+1)\pi \log(|x_j|)+\pi \ \text{for} \ j=1,\dots,n.$$

Conjecture \ref{conj:FS} predicts that $L_k^\infty$ corresponds under mirror
symmetry to the pullback $\pi^*\O_{\PP^n}(k)$, which is a line bundle over
$Z$ that we denote by $\O_Z(k)$, while $L_k^0$ corresponds to
$i_*\O_{\PP^n}(k)=\O_{Z_0}(k)$. Calculations of the wrapped Floer cohomology
groups of these Lagrangians inside $(\C^*)^n\setminus H$ (which are fairly
straightforward using knowledge of homological mirror symmetry for $\PP^n$
and the Lefschetz thimble descriptions) confirm these predictions. 
For instance, we have ring isomorphisms
\begin{multline*}
HW^*(L_0^\infty,L_0^\infty)\simeq \bigoplus_{d\ge 0}
\C[z_0,\dots,z_n]_{(n+1)d}/(z_0\dots z_n)\\[-10pt] \simeq \bigoplus_{d\ge 0}
H^0(D_0,\O(d(n+1)))\simeq H^0(Z,\O_Z).\end{multline*}
Namely, the homogeneous polynomials of degree $(n+1)d$ correspond to the
Reeb chords from $L_0^\infty$ to itself that wrap $d$ times around infinity
under projection by $f$.
Indeed, the Reeb chords are the same as in $(\C^*)^n$,
where the image of $L_0$ under wrapping $d$ times around infinity is isotopic
to $L_{(n+1)d}$.
By Abouzaid~\cite{AbToric} the boundary intersections between these two 
admissible Lagrangians correspond to monomials of degree $(n+1)d$ which are
not divisible by $z_0\dots z_n$. However, when computing the
product structure in $\W((\C^*)^n\setminus H)$ one should discard all
holomorphic discs whose projection under $f$ passes through the origin;
these correspond exactly to all product operations in $\W((\C^*)^n)$ 
where the projection under $f$ of the output Reeb chord wraps around 
infinity fewer times than the sum of the inputs, i.e.\ all those cases
where the product of two monomials is divisible by $z_0\dots z_n$.
A similar argument shows that
\begin{multline*}
HW^*(L_j^\infty,L_k^\infty)\simeq 
H^*(\PP^n,\O(k-j))\oplus \bigoplus_{d>0} H^*(D_0,\O(d(n+1)+k-j))\\[-8pt]
\simeq H^*(Z,\O_Z(k-j)).
\end{multline*}
Meanwhile, a similar calculation around the puncture at the origin
(recalling that the monodromy around zero is trivial up to a grading shift
by 2) shows that
$$HW^*(L_j^0,L_k^0)\simeq H^*(\PP^n,\O(k-j))\oplus \bigoplus_{d>0}
H^{*-2d}(D_0,\O(k-j))\simeq \mathrm{Ext}^*(\O_{Z_0}(j),\O_{Z_0}(k)),$$
while
$$HW^*(L_j^\infty,L_k^0)\simeq H^*(\PP^n,\O(k-j))\simeq
\mathrm{Ext}^*(\O_Z(j),\O_{Z_0}(k)).\bigskip$$

Let us now modify the framing and
treat $H$ as the zero set of 
$$f'=\tau^{-1}x_1\dots x_n f=\tau^{-1}x_1\dots x_n(x_1+\dots+x_n+1)+1.$$
The only critical point of $f'$ is at $x_1=\dots=x_n=-\frac{1}{n+1}$,
which lies outside of the region where the constant term dominates. However,
the fibers of $f'$, which are degree $n+1$ affine hypersurfaces
(e.g.\ three-punctured elliptic curves for $n=2$),
degenerate at the special value $1$, and
$f'{}^{-1}(1)=\{x_1+\dots+x_n+1=0\}$ is an $(n-1)$-dimensional pair of pants.

Choosing $f'{}^{-1}(2)$ as reference fiber for $\F^\circ((\C^*)^n,f')$,
Abouzaid's construction provides admissible Lagrangian submanifolds $L'_k$ which are
sections of the logarithm map over the (unbounded) region $\Delta'_0\subset
\R^n$ where the constant term of $f'$ dominates (on Figure \ref{fig:P2},
this is the lower-left component). In the tropical limit, $L'_k$ corresponds
again to $\arg(x_j)=-2\pi k\log(|x_j|)$. Writing the standard
symplectic form of $(\C^*)^n$ as $\omega_0=\sum dr_j\wedge
d\theta_j$, where $x_j=\exp(r_j+i\theta_j)$,
$L'_k$ is the image of $L'_0\subset (\R_+)^n$ by the
time 1 flow of the Hamiltonian $\varphi=-\pi k \sum r_j^2$.
This Hamiltonian has quadratic growth, and hence for $k\neq 0$
$L'_k$ is not conical at infinity (as $x_j\to 0$).

This can be fixed by the judicious use of cut-off functions to achieve 
linear growth. Namely, splitting $\varphi$ into the sum of
$\varphi_1=-\pi k \sum (r_j+\frac{1}{n+1})^2$ and
$\varphi_2=\frac{2\pi k}{n+1}\sum r_j+\frac{n\pi k}{n+1}$,
we replace $\varphi_1$ by a function $\tilde\varphi_1=
\tilde\varphi_1(|\vec{r}-\vec{r}_0|)$ of the Euclidean distance between
$\vec{r}=(r_1,\dots,r_n)$ and
$\vec{r}_0=(-\frac{1}{n+1},\dots,-\frac{1}{n+1})$,
which grows quadratically up to a certain point and linearly at infinity.
Because $\tilde\varphi_1$ only depends on $|\vec{r}-\vec{r}_0|$, it is still the case that along
the unbounded facet of $\Delta'_0$ where $2r_1+r_2+\dots+r_n=-1$, we have the equality
$\arg(x_1^2x_2\dots
x_n)=\partial_{(2,1,\dots,1)}(\tilde{\varphi}_1+\varphi_2)=2\pi k$, which is
the key property needed to ensure that $\partial L'_k\subset f'{}^{-1}(2)$.
Similarly for the other facets of $\Delta'_0$. The linear growth of
$\tilde{\varphi}_1$ along radial straight lines from $\vec{r}_0$ 
(and the overall linearity of $\varphi_2$, whose effect is simply to rotate
$\theta_j$ by $\frac{2\pi k}{n+1}$) implies that, after this modification,
$L'_k$ is conical at infinity with respect to the Liouville structure
$\lambda=\sum (r_j+\frac{1}{n+1})\,d\theta_j$.

A Floer homology calculation (wrapping by a Hamiltonian that grows
quadratically with $|\vec{r}-\vec{r}_0|$, and otherwise imitating Abouzaid's
arguments \cite{AbToric,AbToric2}) shows that the subcategory of
$\F^\circ((\C^*)^n,f')$ generated by the admissible Lagrangians $L'_k$
is equivalent to $D^b\Coh(Z'_0)$, where $Z'_0$ is the distinguished
component of $Z$, i.e.\ the total space of the line bundle $\O(-(n+1))\to
\PP^{n-1}$. The Lagrangian $L'_k$ corresponds under this equivalence to the
line bundle $\O_{Z'_0}(k)$.

We can extend $L'_k$ to a properly embedded Lagrangian 
$L'{}_k^\infty\subset (\C^*)^n\setminus H$ which is a section of the
logarithm map over all of $\R^n$ (rather than just $\Delta'_0$), by setting
$\arg(x_j)=-2k\pi \log(|x_j|)$ over a large bounded subset of $\R^n$, 
and arranging for $L'{}_k^\infty$ to be conical at infinity by the same cut-off trick as for
$L'_k$. By construction, each monomial of $f'$ is
real positive along the tropical hypersurface, and hence the leading order terms
of $f'$ cannot cancel out; this in turn implies that $L'{}_k^\infty$ is disjoint
from $H$.

It is not hard to check that $L'{}_k^\infty$ is Hamiltonian isotopic
to $L_k^\infty$, and hence isomorphic to it as an object of $\W((\C^*)^n\setminus
H)$.  Thus, $L'{}_k^\infty$ corresponds to the unique way (for $n\ge 2$) to extend
$\O_{Z'_0}(k)$ to a line bundle over all of $Z$, $\O_Z(k)$. However, the
images of $L'{}_k^\infty$ under $f'$ go to infinity along paths that
lie in different homotopy classes relative to the critical value of $f'$;
it might be the case that no choice of arc $\eta_\infty$ in the construction of the acceleration 
functor $\alpha_\infty$ gives $\alpha_\infty(L'_k)\simeq L'{}_k^\infty$ for
all $k$. 

The reason for this is particularly apparent in the case $n=1$.
where the admissible Lagrangians $L'_k$ are arcs connecting 
the origin to the point of $f'^{-1}(2)$ which lies near $x=\tau$, 
inside the region of $\C^*$ where $|x|\le \tau$. These arcs are all isotopic and
represent isomorphic objects of $\F^\circ(\C^*,f')\simeq D^b\Coh(\C)$,
as expected given that $Z'_0\simeq \C$ and hence $\O_{Z'_0}(k)\simeq
\O_{Z'_0}$. On the other hand, the arcs $L'{}_k^\infty$ that connect the origin
to infinity in $\C^*\setminus H$ are definitely not isotopic to each other:
for example, the intersection number of $L'{}_k^\infty$ with the portion of 
the negative real axis that lies between the two points of $H$ (located near
$-1$ and $-\tau$) is equal to $k$. This is consistent with mirror symmetry,
since the line bundles $\O_Z(k)$ are pairwise non-isomorphic. And, of
course, no functor from $\F^\circ(\C^*,f')$ to $\W((\C^*)\setminus H)$
(resp.\ $D^b\Coh(Z'_0)$ to $D^b\Coh(Z)$) can map the isomorphic objects
$L'_k$ (resp.\ $\O_{Z'_0}(k)$) to the non-isomorphic $L'{}_k^\infty$ (resp.\
$\O_Z(k)$).

On the other hand, this issue does not arise for the other acceleration 
functor $\alpha_0$. Namely, the Lagrangians $L'{}_k^0=\alpha_0(L'_k)$ can be
constructed by modifying the sign conventions in the definition of $L'_k$
in order to obtain admissible Lagrangians with boundary in $f'^{-1}(0)=H$.
Specifically, $L'{}_k^0$ is again a section of the logarithm map over
$\Delta'_0\subset \R^n$; in the tropical limit, we take $L'{}_k^0$ to be
defined by
$$\arg(x_j)=-(2k-1)\pi\log(|x_j|)\quad \text{for}\ j=1,\dots,n$$
over a large bounded subset of $\Delta'_0$, and use cut-off functions as in
the construction of $L'_k$ in order to make $L'{}_k^0$ conical at infinity.

For $n=1$ the arcs $L'{}_k^0$ connecting the origin to the point of
$H$ which lies near $x=-\tau$ are all isotopic to each other in
$\C^*\setminus H$, as expected, whereas for $n\ge 2$ these are genuinely
different objects of $\W((\C^*)^n\setminus H)$. For $n=2$ one can check explicitly that
$L'{}_0^k$ corresponds under mirror symmetry to $i_*\O_{Z'_0}(k)$;
presumably this remains true for all $n$.

\section{Higher dimensional pairs of pants}\label{s:pants}

\subsection{Setup and notations}
We consider the pair of pants $H=\Pi_{n-1}\subset(\C^*)^n$, defined by the
equation $$f(x_1,\dots,x_n)=x_1+\dots+x_n+1=0.$$
One easily checks that our construction gives
$(Y,W)=(\C^{n+1},-z_1\dots z_{n+1})$. Thus
$Z$ is the union of the coordinate hyperplanes
$Z_i:\{z_i=0\}$, where the distinguished component (previously called $Z_0$) that corresponds
to the constant term in $f$ is $Z_{n+1}$, while the
other components $Z_1,\dots,Z_n$ correspond to the monomials $x_1,\dots,x_n$.
For $I\subseteq \{1,\dots,n+1\}$, we set $Z_I=\bigcup_{i\in I} Z_i$.

\begin{figure}[b]
\setlength{\unitlength}{1.5cm}
\begin{picture}(3,3)(-2,-2.3)
\psset{unit=\unitlength}
\psline(-2,0)(0.8,0)
\psline(0,-2)(0,0.8)
\psline(0.7,-1.7)(-1.8,0.8)
\psline{->}(1,-1)(1.45,-1.45)
\psline{->}(0.8,-1.8)(0.95,-1.95)
\put(1.1,-1){\tiny $x_1+x_2+1$}
\psline(0.9,-2.3)(1.9,-1.3)
\pscircle*(1.1,-2.1){0.05}
\pscircle*(1.6,-1.6){0.05}
\put(1.2,-2.2){\tiny $0$}
\put(1.7,-1.7){\tiny $1$}
\put(-2,0.3){\small $L_1$}
\put(-1.1,-1.1){\small $L_{12}$}
\put(-0.4,-0.4){\small $L_3$}
\put(0.2,-1.9){\small $L_2$}
\put(-0.9,0.4){\small $L_{23}$}
\put(0.3,0.3){\small $L_{123}=L_0$}
\put(0.3,-0.8){\small $L_{13}$}
\put(0.27,-1.5){\tiny $\ell_2$}
\put(-0.73,-0.55){\tiny $\ell_{1\hspace*{-1pt}2}$}
\put(-1.42,0.2){\tiny $\ell_1$}
\psellipticarc[linewidth=0.5pt](-1.8,0)(0.05,0.1){45}{-45}
\put(-2,-0.2){\tiny $z_2$}
\psellipticarc[linewidth=0.5pt](0,-1.8)(0.1,0.05){135}{45}
\put(-0.25,-1.95){\tiny $z_1$}
\psellipticarc[linewidth=0.5pt](-1.6,0.6)(0.1,0.05){165}{90}
\put(-1.5,0.65){\tiny $z_0$}
\end{picture}
\caption{$\Pi_2\simeq (\C^*)^2\setminus \Pi_1$ and the components of
its real locus.}\label{fig:Pi2}
\end{figure}
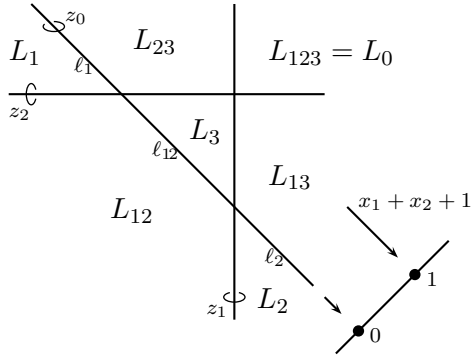

The complement
$(\C^*)^n\setminus \Pi_{n-1}$ is isomorphic to the higher dimensional
pants $\Pi_n$, whose wrapped Fukaya category we aim to study using the ideas introduced in
Sections \ref{s:hmscompl}--\ref{s:conjFS}. The key Lagrangian submanifolds
of interest to us will be the $2^{n+1}-1$ connected components $L_I$ of the real locus of
$\Pi_n=(\C^*)^n\setminus \Pi_{n-1}$, which we label by proper non-empty subsets
$I\subset\{0,\dots,n+1\}$, up to the equivalence
relation that identifies each subset with the complementary subset of
 $\{0,\dots,n+1\}$.
We usually choose the representative which does not contain the
element $0$ as the ``canonical'' label, except in the case of
$\{1,\dots,n+1\}\sim \{0\}$.
See Figure \ref{fig:Pi2} for the case $n=2$.

The labelling is as follows: (1) the positive
orthant $(\R_+)^n\subset (\C^*)^n\setminus \Pi_{n-1}$ is labelled
$L_{\{1,\dots,n+1\}}=L_{\{0\}}$; (2) whenever two components of the real
locus are adjacent to each other across the hyperplane $x_i=0$ ($1\le i\le n$), their
labelling sets differ exactly by adding or removing the element $i$; (3)
whenever two components are adjacent to each other across $\Pi_{n-1}$,
their labels differ by adding or removing the element $0$ (or equivalently,
they are labelled by complementary subsets of $\{1,\dots,n+1\}$).

A more symmetric viewpoint embeds the picture into $\PP^{n+1}$ with
homogeneous coordinates $(x_1\!:\!\dots\!:\!x_{n+1}\!:\!x_0)$ as follows.
We can realize $\Pi_n$ as the intersection of the hyperplane
$\{x_1+\dots+x_{n+1}+x_0=0\}\subset \PP^{n+1}$ with the complement of
the $n+2$ coordinate hyperplanes $x_i=0$; the identification with
$(\C^*)^n\setminus H$ is by mapping $(x_1\!:\!\dots\!:\!x_n\!:\!1\!:x_0)$ to
$(x_1,\dots,x_n)$. For $I\subset \{0,\dots,n+1\}$, 
$L_I$ is the set of real points that admit homogeneous coordinates
$(x_1\!:\!\dots\!:\!x_{n+1}\!:\!x_0)$
(with their sum equal to zero) satisfying $x_i>0$ for $i\in I$ and $x_i<0$ for
$i\not\in I$.

For $I\subseteq \{1,\dots,n+1\}$, the image of $L_I$ under the
logarithm map covers exactly those components of $\R^n\setminus 
\mathrm{Log}(\Pi_{n-1})$ that correspond to the elements of $I$ (recalling
that $n+1$ corresponds to the region where the constant term of $f$ dominates). Thus, as a
general principle we expect that under mirror symmetry $L_I$ corresponds to
an object of $D^b\Coh(Z)$ which is supported on $Z_I=\bigcup_{i\in I} Z_i$
and whose restriction to each $Z_i$, $i\in I$ is a (trivial) line bundle:
\begin{conj}\label{conj:pants}
There is an equivalence $\W(\Pi_n)\simeq D^b\Coh(Z)$ under which
the objects $L_I$ map to $\O_{Z_I}$ for all non-empty $I\subseteq
\{1,\dots,n+1\}$.
\end{conj}
We now explore how this prediction
fits with Conjectures \ref{conj:hmscompl} and \ref{conj:FS}; a proof of
Conjecture \ref{conj:pants} is sketched in Section
\ref{s:pantscomp} below.

\subsection{Restriction and lifting}
For non-empty $I\subseteq \{1,\dots,n\}$, denote by $\ell_I\subset \Pi_{n-1}$ 
the components of the real locus of the $(n-1)$-dimensional pair of pants,
labelled as above. More precisely, embedding $\Pi_{n-1}$ into $\PP^n$ with
homogeneous coordinates $(y_1\!:\!\dots\!:\!y_n\!:\!y_0)$ as the
intersection of the hypersurface $\{y_1+\dots+y_n+y_0=0\}$ with the
complement of the $n+1$ coordinate hyperplanes, we define $\ell_I$ to be
the set of real points where $y_i>0$ iff $i\in I$.
Meanwhile, embedding $\Pi_n$ into $\PP^{n+1}$ as
above, $f$ can be expressed in homogeneous coordinates as 
$(x_1\!:\!\dots\!:\!x_{n+1}\!:\!x_0)\mapsto -x_0/x_{n+1}$. The zero set
$\Pi_{n-1}=f^{-1}(0)$ then corresponds to setting $x_0=0$ (while still
requiring the other coordinates to be non-zero), i.e.\ in
projective coordinates we use the embedding 
\begin{equation}\label{eq:pantsembed}
(y_1\!:\!\dots\!:\!y_n\!:\!y_0)\mapsto
(x_1\!:\!\dots\!:\!x_{n+1}\!:\!x_0)=(y_1\!:\!\dots\!:\!y_n\!:\!y_0\!:\!0).
\end{equation}

By construction, for $I\subseteq \{1,\dots,n\}$, the portion of the boundary of
$L_I$ that lies on the hyperplane $x_0=0$ is exactly $\ell_I$, and similarly
for $L_{I\cup \{0\}}$; see Figure \ref{fig:Pi2}. Using Remark \ref{rmk:framed}, the
images of the objects $L_I$ under the
($\Z/2$-graded) restriction functor $\rho:\W(\Pi_n)\to \W(\Pi_{n-1})$ defined in \S \ref{s:hmscompl}
are therefore as follows:
\begin{lemma}\label{l:restrpants}
For all non-empty $I\subseteq \{1,\dots,n\}$, $\rho(L_I)\cong\ell_I$, while
$\rho(L_{I\cup \{0\}})\cong\ell_I[1]$, and $\rho(L_{\{0\}})=0$.
\end{lemma}

\noindent
This is consistent with Conjecture \ref{conj:hmscompl},
given our expectation that, for $I\subseteq \{1,\dots,n\}$, $L_I$ 
corresponds to $\O_{Z_I}\in D^b\Coh(Z)$. Indeed, recall that 
$D^b_\sg(Z)\simeq D^b\Coh(D)$, where
$D=Z_{\{1,\dots,n\}}\cap Z_{n+1}=\{(z_1,\dots,z_n,0)\in \C^n\times 0\,|\,
z_1\dots z_n=0\}$. For $I\subseteq \{1,\dots,n\}$, we set $D_I=Z_I\cap
Z_{n+1}\subseteq D$. Orlov's construction 
\cite{OrlovKnorrer,OrlovGraded} identifies $[\O_{Z_I}]\in D^b_\sg(Z)$ with
$\O_{D_I}\in D^b\Coh(D)$. Thus, denoting by $q:D^b\Coh(Z)\to D^b_\sg(Z)$ the
quotient functor, and by $\epsilon:D^b_\sg(Z)\to D^b\Coh(D)$ Orlov's
equivalence, we find that $\varepsilon\circ q(\O_{Z_I})\cong \O_{D_I}$.
Given that our mirror symmetry ansatz for pairs of pants matches $\O_{D_I}$
with $\ell_I$ and $\O_{Z_I}$ with $L_I$, this is consistent with $\rho(L_I)\cong
\ell_I$. 

Moreover, let $I'=\{1,\dots,n+1\}\setminus I$, so that $I\cup \{0\}\sim
I'$. We expect that $L_{I\cup\{0\}}=L_{I'}$ corresponds to $\O_{Z_{I'}}\in
D^b\Coh(Z)$. Observing that $Z=Z_I\cup Z_{I'}$, we have a short exact sequence 
$0\to \O_{Z_{I}}\to \O \to \O_{Z_{I'}}\to 0$, which implies that
$\O_{Z_{I'}}$ is isomorphic to $\O_{Z_I}[1]$ in the quotient category
$D^b_\sg(Z)$, hence $\varepsilon\circ q(\O_{Z_{I'}})\cong \O_{D_I}[1]$.
This is again consistent with $\rho(L_{I'})\cong \ell_I[1]$.
Finally, $L_{\{0\}}=L_{\{1,\dots,n+1\}}$ corresponds to $\O_Z\in
D^b\Coh(Z)$, which is annihilated by $q$, in agreement with
$\rho(L_{\{0\}})=0$.

Next we consider the lifting functor $j:\W(\Pi_{n-1})\to \W(\Pi_n)$ of
Section \ref{s:FSj}. Since $x_0$ and $x_{n+1}$ have the same sign on $L_I$ for
$I\subseteq \{1,\dots,n\}$, under $f$ the Lagrangian $L_I$ projects to the real
negative axis, and coincides with the parallel transport of $\ell_I\subset
\Pi_{n-1}=f^{-1}(0)$ in the fibers of $f$ over $\gamma=(-\infty,0)$. Thus:
\begin{equation}
j(\ell_I)\cong L_I\qquad \text{for } I\subseteq \{1,\dots,n\}.
\end{equation}
In particular, $\rho\circ j \cong \mathrm{id}$ as expected.

\subsection{The Fukaya-Seidel category}
We now consider the category $\F^\circ((\C^*)^n,f)$. While
$f=x_1+\dots+x_n+1$ does not have any critical points, the special
fiber $f^{-1}(1)$ is the complement of $n+1$ hyperplanes through the
origin in $\C^{n-1}$, which is diffeomorphic to $\C^*\times \Pi_{n-2}$,
whereas the regular fibers are complements of $n+1$ affine hyperplanes in generic 
position in $\C^{n-1}$, i.e.\ isomorphic to $\Pi_{n-1}$. This degeneration
gives rise to a (non-compact) ``vanishing cycle'', and the Lagrangian obtained
by parallel transport of this vanishing cycle over the interval 
$(1,+\infty)$ is simply the real positive locus $L_{\{0\}}=(\R_+)^n\subset (\C^*)^n$.

Taking the reference fiber to be $f^{-1}(2)$, the category 
$\F^\circ((\C^*)^n,f)$ is generated by the admissible Lagrangian $L_{adm}$ obtained 
by truncating $L_{\{0\}}$ (namely, the portion of $\R_+^n$ where $x_1+\dots+x_n\le
1$), which is a section of the logarithm map over the appropriate region in
$\R^n$ (essentially the negative orthant). A calculation then shows that
the endomorphisms of $L_{adm}$ form a polynomial algebra
$\C[z_1,\dots,z_n]$, where the generator $z_i$ corresponds to a Reeb chord 
that wraps once around the hyperplane $x_i=0$. This computation is
consistent with mirror symmetry, as it matches $\mathrm{End}(\O)$
in the derived category of $Z_{n+1}=\C^n$.

By construction, $\alpha_\infty(L_{adm})=L_{\{0\}}$, in agreement with the
behavior of pullback on the mirror. Namely, the pullback of $\O_{Z_{n+1}}$ under the
map $\pi:Z\to Z_{n+1}$ defined by $\pi(z_1,\dots,z_{n+1})=(z_1,\dots,z_n,0)$
is indeed $\O_Z$.

Meanwhile, $\alpha_0(L_{adm})$ is (up to isotopy) the Lagrangian obtained by
parallel transport of the ``vanishing cycle'' of $f$ over the
interval $(0,1)$, i.e., the portion of $(\R_-)^n$ where
$x_1+\dots+x_n+1>0$. Therefore, $\alpha_0(L_{adm})=L_{\{n+1\}}$.
This is in agreement with Conjecture \ref{conj:FS}, as the image of the
structure sheaf under pushforward by the inclusion $i:Z_{n+1}\hookrightarrow
Z$ is indeed $\O_{Z_{n+1}}$.

Finally, since $\rho$ maps $L_{\{n+1\}}=L_{\{0,1,\dots,n\}}$ to
$\ell_{\{1,\dots,n\}}[1]$ (cf.\ Lemma \ref{l:restrpants}), we have
$j\rho\alpha_0(L_{adm})\simeq L_{\{1,\dots,n\}}[1]$, and
the exact triangle \eqref{eq:jtriangle} takes the form
$$L_{\{1,\dots,n\}}\to L_{\{0\}}\to L_{\{n+1\}}\to L_{\{1,\dots,n\}}[1].$$ 
Under mirror symmetry, this corresponds to the exact triangle in $D^b\Coh(Z)$
induced by the short exact sequence of sheaves
$$0\to \O_{Z_1\cup \dots\cup Z_n}\to \O_Z \to \O_{Z_{n+1}}\to 0.$$

\section{Computing $\W(\Pi_n)$} \label{s:pantscomp}

We now sketch an approach to the calculation of $\W(\Pi_n)$, relying on a
mix of explicit computations and the structural considerations introduced in
the previous section. Conjecture \ref{conj:pants} should then follow as a
corollary.

\subsection{Liouville structure and wrapping Hamiltonian}

Embedding $\Pi_n$ into $\PP^{n+1}$ as the
hyperplane $\Sigma=\{x_1+\dots+x_{n+1}+x_0=0\}$ minus the $n+2$ coordinate
hyperplanes, it is clear that the labels $0,\dots,n+1$ should play
symmetric roles, with the important exception of gradings. 
(Since our preferred trivialization of the tangent bundle is inherited 
from $(\C^*)^n$ and extends across the hyperplane $x_0=0$ but not across the
others, Reeb orbits that wrap around the hyperplane $x_0=0$ are graded
differently from those that wrap around the other coordinate hyperplanes.)
Apart from this, one would like all calculations to be invariant under
the action of $\mathfrak{S}_{n+2}$ by permutation of the coordinates.
Thus, it is desirable to choose the Liouville structure on $\Pi_n$ and
the wrapping Hamiltonian to be $\mathfrak{S}_{n+2}$-invariant. 
(We note that the Lagrangians $L_I$ are conical at infinity, and in fact
invariant under the Liouville flow, for {\em any} choice of complex 
conjugation anti-invariant Liouville structure on $\Pi_n$.)

With this in mind, we stratify 
the hyperplane $\Sigma$ depending on which coordinates vanish, namely for
$I$ a proper subset of $\{0,\dots,n+1\}$ we set $\Sigma_I$ to be the
codimension $|I|$ subset of $\Sigma$
where $x_i=0$ exactly for $i\in I$; the pair of pants $\Pi_n$ is then
the open stratum $\Sigma_\emptyset$.
We claim that the Liouville structure on $\Pi_n$ can be chosen
in such a way that, in a neighborhood of $\Sigma_I$, there are $|I|$ 
commuting $S^1$-actions, generated by Hamiltonians $h_{I,i}$ ($i\in
I$), each of which essentially acts by
rotating one of the coordinates $x_i$, $i\in I$ around the origin. Moreover,
observing that for
$I\subset J$ we have $\Sigma_J\subset \overline{\Sigma}_I$,
we require that $h_{I,i}$ and $h_{J,i}$ agree near $\Sigma_J$ for all $i\in
I\subset J$.  We then take our (quadratic) wrapping Hamiltonian to be
$H_I=\frac12 \sum_{i\in I} h_{I,i}^2$ near $\Sigma_I$.

One possible approach to the construction is as follows.
Near $\Sigma_I$, we define local affine coordinates $x_{I,i}$
by \begin{equation}
x_{I,i}=\Bigl(\pm \prod_{j\not\in I} x_j\Bigr)^{-1/(n+2-|I|)} x_i,
\end{equation}
for some local choice of the $(n+2-|I|)$-th root. (Since we will only use
$|x_{I,i}|$ and $d\log x_{I,i}$, the choice of root is 
not important.) We note that $\Sigma_I$ still corresponds to the locus where
$x_{I,i}=0$ for all $i\in I$, and (for suitable choices of signs and roots)
these coordinates are still real-valued on the components of the real locus of
$\Pi_n$.
These coordinates patch as follows: if
$I\subset J$, then $\Sigma_J\subset \overline{\Sigma}_I$, and (up to a 
root of unity)
$$x_{I,i} = \Bigl(\prod_{j\in J\setminus I} x_{J,j}\Bigr)^{-1/(n+2-|I|)} x_{J,i}$$
and conversely
\begin{equation} \label{eq:xJi}
x_{J,i} = \Bigl(\prod_{j\in J\setminus I} x_{I,j}\Bigr)^{1/(n+2-|J|)} x_{I,i}.
\end{equation}

Writing $x_{I,i}=\exp(r_{I,i} + i\theta_{I,i})$, we
set up the Liouville structure so that, near $\Sigma_I$ and away from all
the lower-dimensional strata, the dominant term in the K\"ahler potential
and in the Hamiltonian is
\begin{equation}\label{eq:PhiI}
H_I=\Phi_I=\frac12 \sum_{i\in I} (r_{I,i}+K_{I})^2
\end{equation}
where $K_{I}=K_{|I|}>0$ is some fixed constant depending only on $|I|$.
As one approaches a lower-dimensional stratum $\Sigma_J$ ($I\subset J$),
this is patched together with the expression $H_J$ by making $K_{I}$ a
function of the variables $r_{I,j}$ (or equivalently $r_{J,j}$) 
for $j\in J\setminus I$, and also by introducing quadratic terms in the
variables $r_{J,j}$. 

In light of \eqref{eq:xJi}, near $\Sigma_J\subset
\overline{\Sigma}_I$ (and away from other smaller strata) we want
to set $K_{I}=\frac{1}{(n+2-|J|)}\sum_{j\in J\setminus I} r_{I,j}+K_{|J|}$
when the quantities $r_{I,j}$, $j\in J\setminus I$ are sufficiently
negative.
In other terms, over a neighborhood of $\Sigma_I$ our Hamiltonian $H$ is expressed in terms of the $x_{I,i}$
as some smooth approximation of
\begin{equation}\label{eq:Hpants}
H={\frac12} \sum_{i=0}^{n+1} \biggl[\min\Bigl(r_{I,i}+\min_{I\subseteq J,\ |J|\le n}
\Bigl\{{\frac{1}{(n+2-|J|)}} \sum_{j\in J\setminus I}
r_{I,j}+K_{|J|}\Bigr\},0\Bigr) \biggr]^2.
\end{equation}
where the smoothing of the minimum still only depends on the values of $r_{I,j}$ for $j\not\in I$.
Choosing the positive constants $K_{|I|}$ sufficiently large, the term
involving $r_{I,i}$ is supported in a small neighborhood of the hyperplane
$x_i=0$. Near the zero-dimensional strata the same formula can be used for the
K\"ahler potential, but along higher dimensional strata we need to add a term involving
only the coordinates $x_{I,j}$ for $j\not\in I$ (essentially, a K\"ahler form on the
$(n-|I|)$-dimensional pair of pants $\Sigma_I$).

For $|I|=n$, i.e.\ near a zero-dimensional stratum, the K\"ahler potential
is given by \eqref{eq:PhiI}, and the 
Hamiltonian $h_{I,i}=r_{I,i}+K_{I}$ generates the vector field $\partial/\partial
\theta_{I,i}$ which rotates $x_{I,i}$ while leaving all other $x_{I,j}$,
$j\in I$ unchanged. (Meanwhile, the two remaining homogeneous coordinates, which
are the largest, vary slightly as needed to preserve the condition
$\sum x_j=0$.)
Even when the ``constants'' $K_{I}$ in \eqref{eq:PhiI} are allowed to
vary and depend on $\{r_{I,j},\ j\in J\setminus I\}$ for some
$J\supset I$ with $|J|=n$, to arrive at an expression of the form
\eqref{eq:Hpants}, and the K\"ahler potential includes an extra term
depending only on $x_{I,j}$ for $j\in J\setminus I$, it remains true that
for $i\in I$ the Hamiltonian $h_{I,i}=r_{I,i}+K_{I}$ generates the vector
field $\partial/\partial \theta_{I,i}$ which rotates $x_{I,i}$ while leaving
$x_{I,j}$ unchanged for $j\in J\setminus \{i\}$.  

Further away from all the zero-dimensional strata there is no longer
a preferred $n$-element subset of the $r_{I,j}$'s on which we can assume the
K\"ahler potential (or even the term $K_I$) solely depends. Nonetheless,
we can arrange for that, near $\Sigma_I$, for $i\in I$ the Hamiltonian $h_{I,i}=r_{I,i}+K_I$
still generates an $S^1$-action which rotates $x_{I,i}$ while preserving
the other coordinates $x_{I,j}$, $j\in I\setminus \{i\}$; a priori none of
the coordinates $x_{I,j}$, $j\not\in I$ are preserved, though we can arrange
for them to vary only by small amounts. For simplicity we still
denote these vector fields by $\partial/\partial \theta_{I,i}$.

Putting everything together, we find that 
near $\Sigma_I$ and away from lower-dimensional strata the vector
field generated by the quadratic Hamiltonian $H$ takes the form
$\sum_{i\in I} h_{I,i}\,\partial/\partial \theta_{I,i}=\sum_{i\in I}
(r_{I,i}+K_I)\,\partial/\partial \theta_{I,i}$. (Note that
$h_{I,i}$ tends to $-\infty$ as $x_i$ approaches zero, i.e.\ we wrap clockwise
around the coordinate hyperplanes).

\subsection{Wrapped Floer cohomology}\label{ss:HWLI}

Given any subset $I\subset \{0,\dots,n+1\}$, denote by $\o{I}=\{0,\dots,n+1\}\setminus I$
the complementary subset. We will consider various quotients of the polynomial ring
$\C[z_0,\dots,z_{n+1}]$, graded with $\deg(z_0)=2$ and $\deg(z_i)=0$ for $i\ge 1$.
For convenience, we define
$z_I=\prod_{i\in I} z_i$. (By convention, $z_\emptyset=1$.)

\begin{prop}\label{prop:HWLI}
Given a non-empty proper subset $I\subset \{0,\dots,n+1\}$,
as a graded ring we have 
\begin{equation}\label{eq:HWLI}
HW^*(L_I,L_I)\simeq \C[z_0,\dots,z_{n+1}]/(z_I, z_{\o{I}}).
\end{equation}
\end{prop}

\proof
Recall that $L_I$ is the component of the real locus of $\Pi_n$ where
the coordinates $x_i$ are positive for $i\in I$ and negative for $i\in
\o{I}$. The closure of $L_I$ in $\Sigma$ intersects the stratum $\Sigma_J$
if and only if neither $I$ nor $\o{I}$ is a subset of $J$.

For such $J$, near $\o{L}_I\cap \Sigma_J$ the local coordinates 
$x_{J,j}$, $j\in J$ define a local projection to $\C^{|J|}$ under which $L_I$
maps to an orthant in the real locus, whereas the wrapping Hamiltonian flow
rotates each $x_{J,j}$ clockwise by increasing amounts as $|x_{J,j}|\to 0$.
Thus, for each tuple of positive integers $(k_j)_{j\in J}$, along
$\o{L}_I\cap \Sigma_J$ (and away from lower-dimensional strata) we have a family 
of time 1 trajectories of $X_H$ from $L_I$ to itself that wraps $k_j$ times
around the hyperplane $x_j=0$. 

To make things non-degenerate, we pick a
``convex'' bounded Morse function on $\o{L}_I$ which reaches its maximum 
at the corners and whose restriction to each stratum $\o{L}_I\cap \Sigma_J$ 
has a single critical point which is a minimum. (Such a function is easy to
construct using the contractibility of $\o{L}_I$ and all of its strata.)
After perturbing the Hamiltonian by a small positive multiple of this function,
there is a single non-degenerate time 1 chord of $X_H$ from $L_I$ to itself which wraps $k_j$ 
times around each hyperplane $x_j=0$ near $\o{L}_I\cap \Sigma_J$.
We label the corresponding generator of $CW^*(L_I,L_I)$
by the monomial $\prod_{j\in J} z_j^{k_j}$, and note that its degree (using
our chosen trivialization of the tangent bundle) is equal to $2k_0$ if $0\in
J$, and zero otherwise.

Letting $J$ vary over all subsets of $\{0,\dots,n+1\}$ which contain neither
$I$ nor $\o{I}$ (including the empty subset, which gives rise to a single
generator $z_\emptyset=1$ at the minimum of the chosen Morse function),
the generators of $CW^*(L_I,L_I)$ are labelled by all the
monomials in $\C[z_0,\dots,z_{n+1}]$ which are divisible neither by $z_I$
nor by $z_{\o{I}}$. Moreover, since all the generators have even degree, the
Floer differential necessarily vanishes; thus \eqref{eq:HWLI} holds as an
isomorphism of graded vector spaces.

Next we observe that $H_1(\Pi_n,L_I;\Z)\simeq H_1(\Pi_n,\Z)\simeq
\Z^{n+2}/(1,\dots,1)$, where the generators of $\Z^{n+2}$
correspond to meridian loops around the coordinate hyperplanes.
Under this isomorphism, the generator $\prod z_j^{k_j}$ of $CW^*(L_I,L_I)$ 
represents the homology class $(k_0,\dots,k_{n+1})$ mod $(1,\dots,1)$; these
homology classes are all distinct.
Whenever there is a perturbed holomorphic curve contributing to the Floer
product, the relative homology class of the output chord must be equal to
the sum of those of the input chords. Moreover, the degree of the output
generator must be the sum of those of the input generators. Since
$\deg(z_0)=2$, we can use the grading to lift the translation ambiguity:
there is an isomorphism $(2\Z) \times H_1(\Pi_n,\Z)\simeq \Z^{n+2}$ under
which the degree and homology class of the generator $\prod z_j^{k_j}$
map to the tuple $(k_0,\dots,k_{n+1})\in \Z^{n+2}$.

Any generator which appears
in the expression of the product of the generators
$\prod z_j^{k_j}$ and $\prod z_j^{\ell_j}$ of $HW^*(L_I,L_I)$ must have
degree $2(k_0+\ell_0)$ and represent
the homology class $(k_0+\ell_0,\dots,k_{n+1}+\ell_{n+1})$ mod
$(1,\dots,1)$. This implies that the product must be
zero if $k_j+\ell_j>0$ for all $j\in I$ or for all $j\in \o{I}$ (for lack of
a suitable generator of $HW^*(L_I,L_I)$), and
otherwise it must be a multiple of the generator $\prod z_j^{k_j+\ell_j}$.

In the latter case, to determine the number of solutions to the perturbed
holomorphic curve equation, we observe that since the set of $j$ such that
$k_j+\ell_j>0$ contains neither all of $I$ nor all of $\o{I}$, it must have
at most $n$ elements, and there exists $J$ with $|J|=n$, containing neither
$I$ nor $\o{I}$, such that
$k_j=\ell_j=0$ whenever $j\not\in J$. We claim that we can determine the
product $\bigl(\prod z_j^{k_j}\bigr)\cdot\bigl(\prod z_j^{\ell_j}\bigr)$
by working in a local model near the 0-dimensional stratum $\Sigma_J$,
using the affine coordinates $x_{J,j}$, $j\in J$.

In these coordinates, the wrapping Hamiltonian near $\Sigma_J$ is
modelled on a standard product Hamiltonian on a neighborhood of the origin
in $(\C^*)^n$,  given by $H=\frac12 \sum
h_{J,j}^2$, where $h_{J,j}=\min(r_{J,j}+K,0)$ up to some smoothing near
$r_{J,j}+K=0$ (recall that $r_{J,j}=\log |x_{J,j}|$),
and the Lagrangian $L_I$ is one of the orthants in the real locus.
We perturb the degenerate
minimum of $h_{J,j}^2$ to achieve non-degeneracy, in a manner such that
the minimum of the perturbed Hamiltonian lies within the local coordinate chart.  
Concretely, in the local chart we can take the perturbed Hamiltonian to be
$$H_\epsilon=\frac12 \sum_{j\in J} \Bigl(\min(r_{J,j}+K,0)^2 + \epsilon r_{J,j}\Bigr)$$
where $\epsilon>0$ is small.
(This choice of perturbation clearly depends on the choice of $J$, and is
different from the perturbation used above to compute the
overall chain complex; this is not an issue,
since we are only interested in a cohomology level computation 
of the product structure, and the isomorphism induced
by continuation between the two choices of perturbations is the obvious one.)

With this understood, $L_I$ and its image under the flow generated by
$H_\epsilon$ are product Lagrangians in the local coordinates
$(x_{J,j})_{j\in J}$,
and in each factor the picture looks exactly like the left half of 
Figure \ref{fig:wrapcyl} (up to just past the midpoint).
The chord from $L_I$ to itself which is labelled
by the monomial $\prod z_j^{k_j}$
wraps $k_j$ times around each coordinate hyperplane, and corresponds
to the intersection point labelled $x_{-k_j}$ in the left half of 
Figure \ref{fig:wrapcyl} for each of the $n$ coordinate factors. Similarly
for $\prod z_j^{\ell_j}$.

The maximum principle implies that any perturbed holomorphic disc which 
contributes to the product of these two generators must remain entirely
within the local chart. Moreover, the projection
to each coordinate factor $x_{J,j}$ is a perturbed holomorphic disc in
(a neighborhood of the origin in) $\C^*$ with boundary on the appropriate 
arcs. Conversely, every tuple of index 0 perturbed holomorphic discs in
the coordinate factors lifts to an index 0 perturbed holomorphic disc in 
the total space of the local chart. 
Recall from \S \ref{ss:cyl_calc} that, on the cylinder, 
the generators $x_{-k_j}$ and $x_{-\ell_j}$ are the inputs of a unique
triangle contributing to the Floer product, whose output is 
$x_{-k_j-\ell_j}$. Thus, we conclude that
$$\bigl(\prod z_j^{k_j}\bigr)\cdot\bigl(\prod z_j^{\ell_j}\bigr)=
\prod z_j^{k_j+\ell_j}.$$
Hence the ring structure on $HW^*(L_I,L_I)$ is as expected.
\endproof

This calculation of $HW^*(L_I,L_I)$ agrees with the expectation from
mirror symmetry. Indeed, switching $I$ and $\o{I}$ if needed, we can assume
that $0\not\in I$, and $L_I$ is expected to correspond to $\O_{Z_I}\in
D^b\Coh(Z)$, or equivalently, since $Z=\mathrm{Spec}\,R$ for 
$R=\C[z_1,\dots,z_{n+1}]/(z_1\dots z_{n+1})$, the $R$-module $R/(z_I)$. Set
$I'=\{1,\dots,n+1\}\setminus I$. Using the resolution
\begin{equation}\label{eq:resolveRzI}
\dots \longrightarrow R \stackrel{z_I}{\longrightarrow} R \stackrel{z_{I'}}{\longrightarrow}
R \stackrel{z_I}{\longrightarrow} R \longrightarrow R/(z_I) \to 0,
\end{equation}
we find that 
$$\Ext^{2k}(R/(z_I),R/(z_I))\simeq \begin{cases} 
R/(z_I) & \text{for}\ k=0,\\ (R/(z_I,z_{I'}))z_0^k &\text{for}\ k>0
\end{cases}$$ where $z_0$ is a generator of $\Ext^2(R/(z_I),R/(z_I))$ as
a module over $\mathrm{End}(R/(z_I))$.

Next, we consider pairs of objects, and show:

\begin{prop}\label{prop:HWLIJ}
Given non-empty proper subsets $I,J\subset \{0,\dots,n+1\}$, we have
\begin{align}\label{eq:HWLIJ}
HW^*(L_I,L_J)\simeq \quad & \C[z_0,\dots,z_{n+1}]/(z_{I\cap J}, z_{\o{I}\cap
\o{J}})\cdot u_{\o{Q}} \\ \oplus\ & \C[z_0,\dots,z_{n+1}]/(z_{I\cap \o{J}},
z_{\o{I}\cap J})\cdot u_Q \nonumber
\end{align}
as a graded $(HW^*(L_I,L_I),HW^*(L_J,L_J))$-bimodule,
where we set\/ $Q=(I\cap J)\cup (\o{I}\cap \o{J})$
and $\o{Q}=(I\cap \o{J})\cup (\o{I}\cap J)$, and 
the generator $u_Q$ $($resp.\ $u_{\o{Q}})$ has
degree $1$ if\/ $0\in Q$ $($resp.\ $0\in \o{Q})$, and~$0$ otherwise.
\end{prop}

As an additional piece of notation, we formally set $$u_Q=z_Q^{1/2}=
\prod_{j\in Q} z_j^{1/2}$$ and similarly for $u_{\o{Q}}$. This allows us
to view generators of $HW^*(L_I,L_J)$ as monomials in $z_0,\dots,z_{n+1}$
with half-integer exponents (and is consistent with gradings).

\proof
The argument is similar to the case of Proposition \ref{prop:HWLI}. 
First, we find a criterion for the closures of $L_I$ and $L_J$ in 
$\Sigma$ to intersect along the stratum $\Sigma_K$ for some $K\subset
\{0,\dots,n+1\}$. In terms of the homogeneous coordinates $(x_1\!:\!\dots
\!:\!x_{n+1}\!:\!x_0)$, the points of $\o{L}_I\cap \Sigma_K$ are those
where $x_i$ is positive for $i\in I\cap \o{K}$, zero for $i\in K$, and
negative for $i\in \o{I}\cap \o{K}$, or vice-versa exchanging $I$ and
$\o{I}$. Moreover, since the sum of the coordinates is zero, there must be
at least one positive and one negative coordinate. Thus, $\o{L}_I\cap \o{L}_J\cap \Sigma_K$ is non-empty in precisely
two cases:
\begin{enumerate}
\item $I\cap \o{K}=J\cap \o{K}\neq \emptyset$ and $\o{I}\cap\o{K}=\o{J}\cap
\o{K}\neq \emptyset$, or
\item $I\cap \o{K}=\o{J}\cap \o{K}\neq \emptyset$ and $\o{I}\cap\o{K}=J\cap
\o{K}\neq \emptyset$.
\end{enumerate}
In case (1), $K$ must contain the symmetric difference of $I$ and $J$, i.e.\
$\o{Q}\subseteq K$; but none of $I,\o{I},J,\o{J}$ can be a subset of $K$.
Similarly for case (2), $K$ must contain the symmetric difference of $I$ and
$\o{J}$, i.e.\ $Q$. Thus, we can reformulate our criterion as:
\begin{enumerate}
\item $\o{Q}\subseteq K$, but $K$ contains neither $I\cap J$ nor $\o{I}\cap
\o{J}$, or
\item $Q\subseteq K$, but $K$ contains neither $I\cap \o{J}$ nor
$\o{I}\cap J$.
\end{enumerate}
With this understood, in case (1), near $\Sigma_K$ the coordinates
$x_{K,j}$, $j\in K$ define a local projection to $\C^{|K|}$ in which
$L_I$ and $L_J$ map to orthants in the real locus; these orthants
correspond to real points whose coordinates have the same signs
for $j\in Q\cap K$ and different signs for $j\in \o{Q}$.
Thus, given any tuple $(k_j)_{j\in K}$ with
$k_j\in \Z_{>0}$ for $j\in Q\cap K$ and $k_j\in \Z_{\ge 0}+\frac12$ for $j\in \o{Q}$,
near $\Sigma_K$ there is a family of time 1 trajectories of $X_H$ from $L_I$ to $L_J$ that
wrap $k_j$ times around the hyperplane $x_j=0$ for each $j\in K$.
After perturbing the Hamiltonian slightly as in the proof of Proposition
\ref{prop:HWLI}, there is a single non-degenerate such trajectory, and we
label the corresponding generator of $CW^*(L_I,L_J)$ by the monomial
$\prod_{j\in K} z_j^{k_j}=\prod_{j\in K} z_j^{\lfloor k_j\rfloor}\,u_{\o{Q}}$.

Similarly in case (2), near $\Sigma_K$ the Lagrangians $L_I$ and $L_J$ map
to orthants where the coordinates have the same signs for $j\in \o{Q}\cap K$
and different signs for $j\in Q$, and there are time 1 trajectories of $X_H$
from $L_I$ to $L_J$ that wrap $k_j$ times around the hyperplane $x_j=0$,
with $k_j\in \Z_{>0}$ for $j\in \o{Q}\cap K$ and $k_j\in
\Z_{\ge 0}+\frac12$ for $j\in Q$. The corresponding generator of
$CW^*(L_I,L_J)$ is denoted by $\prod_{j\in K} z_j^{k_j}=\prod_{j \in K} z_j^{\lfloor
k_j\rfloor}\,u_Q$.

In all cases, with our choice of trivialization of the tangent bundle the
degree of these generators is $2k_0$ if $0\in K$, and zero otherwise.
Letting $K$ vary over all subsets which satisfy (1) or (2), we obtain
that $CW^*(L_I,L_J)$ is isomorphic as a graded vector space to the
right-hand side of \eqref{eq:HWLIJ}.

Next, we observe that, since $L_I$ and $L_J$ are contractible,
by choosing base points $*_I\in L_I$, $*_J\in L_J$, and $*\in \Pi_n$, and
reference paths from $*$ to $*_I$ and from $*$ to $*_J$, we can use the
reference paths to complete any arc connecting $L_I$ to $L_J$ into a closed
loop in $\Pi_n$, uniquely up to homotopy. In other terms, the space of
homotopy classes of paths from $L_I$ to $L_J$ is a torsor over
the fundamental group $\pi_1(\Pi_n,*)$, and can be identified
(non-canonically) with it. Passing to homology, we can use this to assign
elements of $H_1(\Pi_n,\Z)\simeq \Z^{n+2}/(1,\dots,1)$ to the generators of $CW^*(L_I,L_J)$. 
A more canonical choice in our case shifts by $\frac12$ the entries corresponding to
elements of $Q$ or $\o{Q}$, and takes values in the subset
$\Gamma_Q$ of $(\frac12\Z)^{n+2}/(\frac12,\dots,\frac12)$ consisting of
those tuples whose
non-integer entries correspond exactly to the elements of either $Q$ or
$\o{Q}$. (Note that $\Gamma_Q$ is an $H_1(\Pi_n,\Z)$-torsor, and
additively, $\Gamma_I+\Gamma_Q=\Gamma_J$.)

With this understood, the class associated to the generator of
$CW^*(L_I,L_J)$ that lies near $\Sigma_K$ and wraps $k_j>0$ times around the hyperplane $x_j=0$ for
each $j\in K$ (and setting $k_j=0$ for $j\not\in K$) is $(k_0,\dots,k_{n+1})\in \Gamma_Q\subset
(\frac12\Z)^{n+2}/(\frac12,\dots,\frac12)$.

As before, the grading on the Floer complex can be used to avoid quotienting
by the diagonal subgroup. Namely,
since the above-mentioned generator has degree $2k_0$, 
its degree and homology class can be encoded simultaneously by the tuple of
half-integers $(k_0,\dots,k_{n+1})\in (\frac12 \Z)^{n+2}$ (where the
non-integer entries correspond exactly to either $Q$ or $\o{Q}$; we denote
the $\Z\times H_1(\Pi_n,\Z)$-torsor of such elements by $\hat{\Gamma}_Q$).

With this understood, any two generators of $CW^*(L_I,L_J)$ related by the Floer
differential must represent the same homology class, while their degrees
differ by 1, hence the corresponding tuples must differ by $(\frac12,\dots,
\frac12)$. However, since the subsets $Q$ satisfying conditions (1) or (2)
above have at most $n$ elements, all the generators of $CW^*(L_I,L_J)$
correspond to tuples in $\hat{\Gamma}_Q\subset
(\frac12\Z)^{n+2}$ in which all entries are non-negative and at least two 
are zero. Two such tuples cannot differ by $(\frac12,\dots,\frac12)$.
Thus, the Floer differential must vanish identically, and \eqref{eq:HWLIJ} holds as an
isomorphism of graded vector spaces.

The statement about module structures is a special case of Proposition
\ref{prop:HWLIJprod}, which we state and prove below.
\endproof

This calculation again agrees with the mirror
symmetry prediction. Without loss of generality we can assume that $0\not\in
I$ and $0\not\in J$. $\Ext^*(R/(z_I),R/(z_J))$ can then be computed
using the resolution \eqref{eq:resolveRzI}; the outcome of the calculation
matches the right-hand side of \eqref{eq:HWLIJ}.

\begin{prop}\label{prop:HWLIJprod}
Indexing generators by monomials in $z_0,\dots,z_{n+1}$ with half-integer
exponents as in Proposition \ref{prop:HWLIJ},
the Floer product $HW^*(L_J,L_K)\otimes HW^*(L_I,L_J)\to HW^*(L_I,L_K)$
is simply given by multiplication of monomials (and quotienting by
the appropriate ideals).
\end{prop}

\proof
Recall from the proof of Proposition \ref{prop:HWLIJ} that the degree and
homology class of each generator of $HW^*(L_I,L_J)$ can be encoded by an
element of $\hat\Gamma_{Q=Q(I,J)}\subset (\frac12 \Z)^{n+1}$, which we call
its {\em class}, and that
this corresponds to the labelling by monomials. Namely, the class of the
generator of $HW^*(L_I,L_J)$
denoted by $\prod z_j^{k_j}$ is $(k_0,\dots,k_{n+1})$. We also recall
that the entries of this tuple are non-negative, and that the set $K$ of
its non-zero entries must contain either $Q(I,J)$
or its complement (one of which corresponds to the half-integer entries) but cannot entirely contain any of
$I,\o{I},J,\o{J}$.

Whenever there is a perturbed holomorphic curve contributing to the Floer
product, the relative homology class of the output chord must equal the sum
of those of the inputs, and its degree must also be the sum of those of the
input. It follows that the class of the output must be the sum of those of
the inputs. (Here we recall that, under addition of tuples,
$\hat\Gamma_{Q(I,J)}+\hat\Gamma_{Q(J,K)}=\hat\Gamma_{Q(I,K)}\subset (\frac12
\Z)^{n+2}$.)

Thus, given generators $\prod z_j^{k_j}\in HW^*(L_I,L_J)$ and $\prod
z_j^{\ell_j}\in HW^*(L_J,L_K)$, their product must be 
\begin{enumerate}
\item[(a)] a multiple of
$\prod z_j^{k_j+\ell_j}\in HW^*(L_I,L_K)$ if there is a generator of
$HW^*(L_I,L_K)$ 
representing the class $(k_0+\ell_0,\dots,k_{n+1}+\ell_{n+1})$,\smallskip
\item[(b)] zero otherwise.
\end{enumerate}
Case (b) obviously agrees with our expected product
formula; so it is enough to consider case (a).
Let $S=\{j\in \{0,\dots,n+1\}\,|\,k_j+\ell_j>0\}=\{j\,|\,k_j>0\}\cup
\{j\,|\,\ell_j>0\}$. By assumption, $S$
contains either $Q(I,J)$ or $\o{Q}(I,J)$, and it contains either $Q(J,K)$ or
$\o{Q}(J,K)$; but since we are in case (a), it does not contain any of $I,\o{I},K,\o{K}$.
Since $Q(I,J)$ (resp.\ $\o{Q}(I,J)$) is the symmetric difference of $I$ and 
$\o{J}$ (resp.\ $I$ and $J$), this implies that $S$ does not contain $J$ or
$\o{J}$ either. Moreover, by the same argument,
for arbitrary elements $i_1\in I\cap \o{S}$ and
$i_2\in \o{I}\cap \o{S}$, the $n$-element subset $T=\o{\{i_1,i_2\}}\supseteq
S$, which does not contain $I$ or $\o{I}$ by construction, also fails to 
contain any of $J,\o{J},K,\o{K}$.

The holomorphic curves contributing to the product can then be determined by
working in a local model near the 0-dimensional stratum $\Sigma_T$, using
the affine coordinates $x_{T,j}$, $j\in T$, and reducing to a product
situation, exactly as in the proof of Proposition \ref{prop:HWLI}. The only
difference is that $L_I,L_J,L_K$ now correspond to different orthants, hence
inside each $\C^*$ factor they project to arcs that may be either
$\R_+$ or $\R_-$. In $\C^*$, the generators of
$CW^*(\R_+,\R_-)$ or $(\R_-,\R_+)$ are naturally labelled by 
half-integers rather than integers. 
Nonetheless, for each pair of inputs $x_{-k_j}$ and $x_{-\ell_j}$ 
there is a unique triangle contributing to the Floer product, whose
output is $x_{-k_j-\ell_j}$. Thus, in the product $(\C^*)^n$ there is
a unique contribution to the Floer product, and we find that
$\bigl(\prod z_j^{k_j}\bigr)\cdot\bigl(\prod z_j^{\ell_j}\bigr)=
\prod z_j^{k_j+\ell_j}$ as expected.
\endproof

\subsection{Exact triangles and generators}

We expect that the $A_\infty$-category $\W(\Pi_n)$ is entirely determined by the
cohomology-level computations in Propositions 
\ref{prop:HWLI}--\ref{prop:HWLIJprod} and the existence of certain exact
triangles that follow from the general framework introduced in the previous
sections.

\begin{prop}\label{prop:exactuIJK}
Given any partition $\{0,\dots,n+1\}=I\sqcup J\sqcup K$ into three
non-empty disjoint subsets, with $0\in K$, there is an exact triangle 
\begin{equation}\label{eq:exactuIJK}
L_I\stackrel{u_J}{\longrightarrow} L_K \stackrel{u_I}{\longrightarrow} L_J
\stackrel{u_K}{\longrightarrow} L_I[1].
\end{equation}
\end{prop}

\noindent Note that $L_K=L_{I\sqcup J}$. Thus, under mirror symmetry the
exact triangle \eqref{eq:exactuIJK} corresponds to the triangle in
$D^b\Coh(Z)$ induced by the short exact sequence
$$0\to \O_{Z_I} \to \O_{Z_{I\sqcup J}} \to \O_{Z_J} \to 0.$$

\proof[Sketch of proof]
The easiest way to establish the existence of an exact triangle relating
$L_I,L_K,L_J$ in $\W(\Pi_n)$ is by induction on dimension,
using symmetry and the lifting functor $j$. The case $n=1$ holds
by \cite{AAEKO}, as reviewed in Section \ref{s:pants1} (see 
\eqref{eq:2exacttri}).

Assume first that $\{0,n+1\}\subset K$, so that $I$ and $J$ are
subsets of $\{1,\dots,n\}$. Then, as noted in \S \ref{s:pants}, the 
lifting functor $\rho:\W(\Pi_{n-1})\to \W(\Pi_n)$ coming from the 
identification $\Pi_n\simeq (\C^*)^n\setminus \Pi_{n-1}$ maps the 
objects $\ell_I$, $\ell_J$ and $\ell_{I\sqcup J}$ of $\W(\Pi_{n-1})$ to
$L_I$, $L_J$, and $L_{I\sqcup J}$.
Assuming the conjecture holds for $\Pi_{n-1}$, in $\W(\Pi_{n-1})$ we have 
an exact triangle
$$\ell_I\stackrel{v_J}{\longrightarrow} \ell_{I\sqcup J} \stackrel{v_I}{\longrightarrow} 
\ell_J \stackrel{v_{K'}}{\longrightarrow} \ell_I[1],$$
where $K'=\{0,\dots,n\}\setminus (I\cup J)$, and we change the notation for the generators
of the Floer complexes in $\W(\Pi_{n-1})$ to $v_I,v_J,v_{K'}$ to avoid
confusion.
Since $A_\infty$-functors are automatically exact, the image by $j$ of this
exact triangle in $\W(\Pi_{n-1})$ is an exact triangle relating
$L_I$, $L_{I\sqcup J}$ and $L_J$ in $\W(\Pi_n)$.

Moreover, recall that the action of $j$ on morphisms comes from the
inclusion of some fiber of $f=-x_0/x_{n+1}$ along which the wrapping Hamiltonian
reaches its minimum; for example, we can take the fiber above $-1$.
Also recall that we use the embedding \eqref{eq:pantsembed} to
match the pictures for $\Pi_{n-1}$ and $\Pi_n$; in this setting, the
embedding into the fiber $f^{-1}(-1)$ that gives rise to the functor $j$
is $(y_1\!:\!\dots\!:\!y_n\!:\!y_0)\mapsto
(y_1\!:\!\dots\!:\!y_n\!:\!y_0\!:\!y_0)$.

With this understood, it is not hard to check that $j(v_I)=u_I$ and $j(v_J)=u_J$. 
Meanwhile, because $v_{K'}$ wraps 
halfway around the hyperplane $y_0=0$, which maps to the base locus of $f$,
its image under the embedding is a trajectory 
that wraps halfway around {\em both} of the hyperplanes $x_0=0$ and
$x_{n+1}=0$. Hence $j(v_{K'})=u_{K'\cup \{n+1\}}=u_K$. This completes the
proof in the case where
$\{0,n+1\}\subset K$.

The remaining cases follow by symmetry under the action of
$\mathfrak{S}_{n+2}$. Namely, for $n\ge 2$ at least one of the subsets
$I,J,K$ must have cardinality greater than one. Observing that 
a cyclic permutation of $(I,J,K)$ amounts simply to a rotation
of the exact triangle \eqref{eq:exactuIJK}, and relaxing the setup to allow
$0$ to be in any of $I,J,K$, we can
assume without loss of generality that $|K|\ge 2$. We can then use
$\mathfrak{S}_{n+2}$-symmetry to relabel the elements of $\{0,\dots,n+1\}$
(with a grading change as needed if the permutation does not fix $0$)
in order to reduce to the case where $\{0,n+1\}\subset K$.
\endproof

The two remaining ingredients in the proof of homological mirror symmetry
for the pair of pants $\Pi_n$ are:

\begin{conj}
$\W(\Pi_n)$ is split-generated by the objects $L_{\{i\}}$, $i=0,1,\dots,n+1$.
\end{conj}

This statement follows from Zack Sylvan's work in progress \cite{Sylvan2},
as discussed in Remark \ref{rmk:sylvan}. Namely, $\W((\C^*)^n\setminus H)$ is 
generated by the images of the functors $i_1$ and $i_2$ in the pushout
diagram; translating to our setup, the image of $i_1$ coincides with that
of the lifting functor $j$, while $i_2=\alpha_\infty$. By induction on
dimension, assuming that $\W(\Pi_{n-1})$ is split-generated by $\ell_{\{0\}}=
\ell_{\{1,\dots,n\}},\ell_{\{1\}},\dots,\ell_{\{n\}}$, the image of $j$
is split-generated by $j(\ell_{\{1,\dots,n\}})=L_{\{1,\dots,n\}}=L_{\{0,n+1\}}$
and $j(\ell_{\{i\}})=L_{\{i\}}$, $i=1,\dots,n$. Meanwhile, as seen in \S
\ref{s:pants}, $\alpha_\infty$ maps the generator of $\F^\circ((\C^*)^n,f)$ 
to $L_{\{0\}}$. Thus, $\W(\Pi_n)$ is split-generated by $L_{\{0,n+1\}}$,
$L_{\{1\}},\dots,L_{\{n\}}$, and $L_{\{0\}}$. 
Finally, by Proposition~\ref{prop:exactuIJK} the objects
$L_{\{0\}}$, $L_{\{0,n+1\}}$ and $L_{\{n+1\}}$ are related by an exact
triangle, so any two of them generate the third one. 

\begin{conj}
Up to homotopy, there is a unique $A_\infty$-structure on the algebra 
$\bigoplus_{I,J} HW^*(L_I,L_J)$ which is compatible with the grading and
satisfies the two conditions:

$(1)$ any generator appearing in the output of a higher product represents
a relative homology class in $H_1(\Pi_n,\Z)$ equal to the sum of those of the inputs,
and

$(2)$ $\mu^3(u_I,u_J,u_K)=\pm \mathrm{id}$ for all $I\sqcup J\sqcup
K=\{0,\dots,n+1\}$ $($as implied by the exact 
triangles~\eqref{eq:exactuIJK}$)$.
\end{conj}

\noindent
Representing generators by monomials in $z_0,\dots,z_{n+1}$ with
half-integer exponents as in Proposition \ref{prop:HWLIJprod},
condition (1) and compatibility with the grading can be restated as:
the class of the output of a Floer product $\mu^k$ differs from the
sum of those of its inputs by $(2-k)(\frac12,\dots,\frac12)$.
In other terms, given generators $\gamma_i=\prod z_j^{\ell_{i,j}}\in HW^*(L_{I_{i-1}},
L_{I_i})$ for $i=1,\dots,k$, the product $\mu^k(\gamma_k,\dots,
\gamma_1)$
must be a multiple of $\gamma_{out}=\prod z_j^{\ell_{out,j}}$,
where $\ell_{out,j}=\sum_{i=1}^k \ell_{i,j}+\frac{2-k}{2}$ for
$j=0,\dots,n+1$,
if $HW^*(L_{I_0},L_{I_k})$ contains such a generator, and zero otherwise.
The conjecture states that, given these constraints, the
$A_\infty$-structure is entirely determined by the $\mu^3$ in condition (2).
The case $n=1$ is established in \cite{AAEKO} by an explicit Hochschild
cohomology calculation.


\begin{thebibliography}{99}
\bibitem{AbToric}
  M. Abouzaid,
  {\sl Homogeneous coordinate rings and mirror symmetry for toric varieties},
  Geom. Topol. {\bf 10} (2006), 1097--1157.
\bibitem{AbToric2}
   M. Abouzaid,
   {\sl Morse homology, tropical geometry, and homological mirror symmetry
   for toric varieties}, Selecta Math.\ {\bf 15} (2009), 189--270.
\bibitem{AbGenerate}
   M. Abouzaid,
   {\sl A geometric criterion for generating the Fukaya category}, Publ.\
   Math.\ IH\'ES {\bf 112} (2010), 191--240.
\bibitem{AA}
   M.~Abouzaid, D.~Auroux, {\sl Homological mirror symmetry for
   hypersurfaces in $(\C^*)^n$}, in preparation.
\bibitem{AAEKO}
   M.~Abouzaid, D.~Auroux, A.~I.~Efimov, L.~Katzarkov, D.~Orlov,
   {\sl Homological mirror symmetry for punctured spheres}, J.\ Amer.\
   Math.\ Soc.  {\bf 26} (2013), 1051--1083.
\bibitem{AAK}
   M.\ Abouzaid, D.\ Auroux, L.\ Katzarkov,
   {\sl Lagrangian fibrations on blowups of toric varieties and mirror symmetry
   for hypersurfaces},
   Publ.\ Math.\ IH\'ES {\bf 123} (2016), 199--282.
\bibitem{AG}
   M. Abouzaid, S. Ganatra, {\sl Generating Fukaya categories of LG models}, in preparation.
\bibitem{AS}
   M. Abouzaid, P. Seidel,
   {\sl An open string analogue of Viterbo functoriality},
   Geom.\ Topol.\ {\bf 14} (2010), 627--718.
\bibitem{AS2}
   M. Abouzaid, P. Seidel, {\sl Lefschetz fibration techniques in wrapped
   Floer cohomology}, in preparation.
\bibitem{Au09}
   D.\ Auroux, {\sl Special Lagrangian fibrations, wall-crossing, and mirror
   symmetry}, Surveys in Differential Geometry {\bf 13}, 
   Intl. Press, 2009, 1--47. 
\bibitem{AuGuide}
   D.\ Auroux, {\sl A beginner's introduction to Fukaya categories},
   Contact and Symplectic Topology, F.\ Bourgeois et al.\ Eds.,
   Bolyai Soc.\ Math.\ Stud.\ {\bf 26}, Springer, 2014, 85--136.
\bibitem{AKO}
   D. Auroux, L. Katzarkov, D. Orlov,
   {\sl Mirror symmetry for weighted projective planes and their noncommutative
   deformations},
   Ann.\ Math.\ {\bf 167} (2008), 867--943.
\bibitem{Ganatra}
   S.\ Ganatra, {\sl Symplectic cohomology and duality for the wrapped
   Fukaya category}, arXiv: 1304.7312.
\bibitem{GPS}
   S. Ganatra, J. Pardon, V. Shende, {\em Localizing the Fukaya category of
   a Weinstein manifold}, in preparation.
\bibitem{HV}
   K.\ Hori, C.\ Vafa, {\sl Mirror symmetry},
   arXiv:hep-th/0002222.
\bibitem{Lee}
   H.\ Lee, {\sl  Homological mirror symmetry for open Riemann surfaces from pair-of-pants
   decompositions}, arXiv:1608.04473.
\bibitem{MWW}
   S.\ Ma'u, K.\ Wehrheim, C.\ Woodward, {\sl $A_\infty$ functors for
   Lagrangian correspondences}, arXiv: 1601.04919.
\bibitem{NadlerWrap}
   D.\ Nadler, {\sl Wrapped microlocal sheaves on pairs of pants},
   arXiv:1604.00114.
\bibitem{Orlov}
   D. Orlov, {\sl Triangulated categories of singularities and D-branes in
   Landau-Ginzburg models}, Proc.\ Steklov Inst.\ Math.\ {\bf 246} (2004),
   227--248.
\bibitem{OrlovKnorrer}
   D. Orlov, {\sl Triangulated categories of singularities and equivalences
   between Landau-Ginzburg models}, Sb.\ Math.\ {\bf 197} (2006), 1827--1840.
\bibitem{OrlovGraded}
   D. Orlov, {\sl Derived categories of coherent sheaves and triangulated
   categories of singularities }, Algebra, arithmetic, and geometry: in honor
   of Yu.\ Manin, Vol.\ II, Progr.\ Math.\ {\bf 270}, Birkh\"auser, 2009,
   503--531.
\bibitem{SeBook}
   P. Seidel,
   {\it Fukaya categories and Picard-Lefschetz theory},
   Zurich Lect.\ in Adv.\ Math., European Math.\ Soc., Z\"urich, 2008.
\bibitem{SeLef}
   P. Seidel, 
   {\sl Fukaya $A_\infty$-structures associated to Lefschetz
   fibrations I}, J. Symplectic Geom.\ {\bf 10} (2012), 325--388.
\bibitem{Sheridan}
   N. Sheridan,
   {\sl  On the homological mirror symmetry conjecture for pairs of pants},
   J.\ Differential Geom.\ {\bf 89} (2011), 271--367.
\bibitem{Sheridan2}
   N.\ Sheridan, {\sl  Homological mirror symmetry for Calabi-Yau hypersurfaces in projective
   space}, Inventiones Math.\ {\bf 199} (2015), 1--186.
\bibitem{Sylvan}
   Z.~Sylvan, {\sl On partially wrapped Fukaya categories},
   arXiv:1604.02540.
\bibitem{Sylvan2}
   Z. Sylvan, {\sl A gluing formula for wrapped Fukaya categories}, in preparation.
\end{thebibliography}
\end{document}